\documentclass{amsart}
\usepackage{amssymb,latexsym,amsfonts}
\usepackage{amsthm}
\usepackage[T1]{fontenc}
\usepackage{textcomp}
\usepackage{fontenc}
\usepackage{amsmath}
\usepackage{graphicx}
\usepackage{stmaryrd}
\usepackage{geometry}
\usepackage[all]{xy}
\usepackage[francais]{babel}
\usepackage[latin1]{inputenc}

\DeclareMathOperator{\Rep}{Rep}

\DeclareMathOperator{\Div}{Div}

\DeclareMathOperator{\Kern}{Ker}
\DeclareMathOperator{\Bun}{Bun_{G}}
\DeclareMathOperator{\Bung}{Bun_{G'}}
\DeclareMathOperator{\Bunz}{Bun_{Z}}

\DeclareMathOperator{\Bunl}{Bun_{G,\lambda}}
\DeclareMathOperator{\Bunp}{Bun_{G}^{par}}

\DeclareMathOperator{\Buns}{Bun_{S}}
\DeclareMathOperator{\Bunpl}{Bun_{G,\lambda}^{par}}

\DeclareMathOperator{\Spec}{Spec}
\DeclareMathOperator{\inv}{inv}
\DeclareMathOperator{\supp}{supp}
\DeclareMathOperator{\Id}{Id}
\DeclareMathOperator{\Hom}{Hom}
\DeclareMathOperator{\soc}{soc}
\DeclareMathOperator{\End}{End}
\DeclareMathOperator{\dime}{dim}
\DeclareMathOperator{\codim}{codim}
\DeclareMathOperator{\Conv}{Conv}

\DeclareMathOperator{\Gr}{Gr}

\DeclareMathOperator{\Aut}{Aut}
\DeclareMathOperator{\ad}{ad}
\DeclareMathOperator{\Ad}{Ad}

\DeclareMathOperator{\Ker}{Ker}

\DeclareMathOperator{\Tr}{Tr}
\DeclareMathOperator{\car}{car}
\DeclareMathOperator{\rg}{rg}
\DeclareMathOperator{\occ}{occ}
\DeclareMathOperator{\out}{Out}

\DeclareMathOperator{\Res}{Res}
\DeclareMathOperator{\Gal}{Gal}

\DeclareMathOperator{\val}{val}
\DeclareMathOperator{\mi}{min}
\DeclareMathOperator{\amp}{amp}
\DeclareMathOperator{\act}{act}

\DeclareMathOperator{\Lie}{Lie}

\DeclareMathOperator{\cl}{cl}

\DeclareMathOperator{\Fr}{Fr}
\DeclareMathOperator{\vol}{vol}

\newtheorem{thm}{Théorème}[section]
\newtheorem{prop}[thm]{Proposition}
\newtheorem{lem}[thm]{Lemme}
\newtheorem{defi}[thm]{Définition}

\newtheorem{cor}[thm]{Corollaire}

\newcommand{\rmq}{\noindent\textbf{Remarque :}}

\newcommand{\cH}{\mathcal{H}}

\newcommand{\cR}{\mathcal{R}}

\newcommand{\GL}{GL}

\newcommand{\co}{\mathcal{O}}
\newcommand{\bv}{\bar{v}}
\newcommand{\so}{\textbf{SO}}
\newcommand{\cov}{\mathcal{O}_{v}}
\newcommand{\cobv}{\bar{\mathcal{O}}_{v}}
\newcommand{\cP}{\mathcal{P}}
\newcommand{\cB}{\mathcal{B}}
\newcommand{\cL}{\mathcal{L}}
\newcommand{\cm}{\mathcal{M}}

\newcommand{\f}{\phi}
\newcommand{\bO}{\textbf{O}}
\newcommand{\g}{\gamma}

\newcommand{\la}{\lambda}

\newcommand{\eps}{\epsilon}

\newcommand{\chl}{\overline{\mathcal{H}}_{\lambda}}
\newcommand{\chm}{\overline{\mathcal{H}}_{\lambda}}

\newcommand{\cmdan}{\overline{\mathcal{M}}_{\lambda}^{ani}}

\newcommand{\tcmdanf}{\overline{\mathcal{M}}_{\lambda,\infty}^{ani,\flat}}
\newcommand{\cma}{\overline{\mathcal{M}}_{\lambda,v}(a)}
\newcommand{\cmaH}{\overline{\mathcal{M}}_{\lambda,v}^{H}(a)}

\newcommand{\tcmdan}{\overline{\mathcal{M}}_{\lambda,\infty}^{ani}}
\newcommand{\tcPan}{\tilde{\mathcal{P}}^{ani}}
\newcommand{\cPaan}{\mathcal{P}^{aug,ani}}
\newcommand{\cPinf}{\mathcal{P}^{\infty,ani}}

\newcommand{\abdan}{\mathcal{A}_{\lambda}^{ani}}
\newcommand{\abdban}{\mathcal{A}_{\lambda}^{ani,\flat}}

\newcommand{\tabdan}{\tilde{\mathcal{A}}_{\lambda}^{ani}}
\newcommand{\tabdanf}{\tilde{\mathcal{A}}_{\lambda}^{ani,\flat}}
\newcommand{\tabd}{\tilde{\mathcal{A}}_{\lambda}}
\newcommand{\tabdk}{\tilde{\mathcal{A}}_{\lambda}(\bar{k})}
\newcommand{\tabdp}{\tilde{\mathcal{A}}_{\lambda,\psi}}
\newcommand{\tabdi}{\tilde{\mathcal{A}}_{\lambda,(I_{-},W_{-})}}
\newcommand{\tabdde}{\tilde{\mathcal{A}}_{\lambda,\delta}}

\newcommand{\abdde}{\mathcal{A}_{\lambda,\delta}}

\newcommand{\tabdeth}{\tilde{\mathcal{A}}_{\eta^{*}\lambda,H}}
\newcommand{\abdeth}{\mathcal{A}_{\eta^{*}\lambda,H}}

\newcommand{\tabdethb}{\tilde{\mathcal{A}}_{\eta^{*}\lambda,H}^{bon}}

\newcommand{\tabdethanf}{\tilde{\mathcal{A}}_{\eta^{*}\lambda,H}^{ani,\flat}}
\newcommand{\bql}{\overline{\mathbb{Q}}_{l}}

\newcommand{\tabdanka}{\tilde{\mathcal{A}}_{\lambda,\kappa}^{ani}}
\newcommand{\cmdo}{\mathcal{M}_{\lambda}}
\newcommand{\cmD}{\mathcal{M}_{D}}
\newcommand{\cmDan}{\mathcal{M}_{D}^{ani}}

\newcommand{\cmd}{\overline{\mathcal{M}}_{\lambda}}
\newcommand{\cmdaan}{\overline{\mathcal{M}}^{aug,ani}_{\lambda}}
\newcommand{\cmdhp}{\overline{\mathcal{M}}_{\lambda}^{\heartsuit, +}}
\newcommand{\cmdb}{\overline{\mathcal{M}}_{\lambda}^{\flat}}

\newcommand{\cmdbD}{\overline{\mathcal{M}}_{\lambda}^{\flat,\leq d}}

\newcommand{\cmdh}{\overline{\mathcal{M}}_{\lambda}^{\heartsuit}}

\newcommand{\cmdhk}{\overline{\mathcal{M}}_{\lambda}^{\heartsuit}(\bar{k})}

\newcommand{\cmda}{\overline{\mathcal{M}}_{\lambda}(a)}
\newcommand{\cmavH}{\overline{\mathcal{M}}_{\lambda,v}^{H}(a_{H})}

\newcommand{\cmdaa}{\overline{\mathcal{M}}_{\lambda}^{aug}}

\newcommand{\kc}{\mathfrak{C}_{+}}

\newcommand{\kch}{\mathfrak{C}_{+,H}}

\newcommand{\kC}{\mathfrak{C}}
\newcommand{\kD}{\mathfrak{D}}
\newcommand{\kCH}{\mathfrak{C}_{H}}

\newcommand{\kcd}{\mathfrak{C}_{+}^{\lambda}}
\newcommand{\kcdr}{\mathfrak{C}_{+}^{\lambda,rs}}
\newcommand{\kcdfr}{\mathfrak{C}_{+}^{\lambda,f-rs}}
\newcommand{\kg}{\mathfrak{g}}

\newcommand{\kt}{\mathfrak{t}}
\newcommand{\kX}{\mathfrak{X}}

\newcommand{\ev}{ev}

\newcommand{\bo}{\textbf{O}}
\newcommand{\bh}{\textbf{H}}

\newcommand{\cQ}{\mathcal{Q}}

\newcommand{\abd}{\mathcal{A}_{\lambda}}
\newcommand{\abD}{\mathcal{A}_{D}}
\newcommand{\abDan}{\mathcal{A}_{D}^{ani}}
\newcommand{\abda}{\mathcal{A}^{aug}_{\lambda}}
\newcommand{\abdaan}{\mathcal{A}^{aug,ani}_{\lambda}}
\newcommand{\abdanf}{\mathcal{A}_{\lambda}^{ani,\flat}}
\newcommand{\abdbon}{\mathcal{A}_{\lambda}^{bon}}
\newcommand{\abde}{\mathcal{A}_{\lambda,H}}
\newcommand{\abdae}{\mathcal{A}^{aug}_{\lambda,H}}

\newcommand{\abdah}{\mathcal{A}^{aug,\heartsuit}_{\lambda}}
\newcommand{\abdb}{\mathcal{A}_{\lambda}^{\flat}}
\newcommand{\abdbf}{\mathcal{A}_{\lambda}^{ani,\flat}}
\newcommand{\abdbD}{\mathcal{A}_{\lambda}^{\flat,\leq d}}
\newcommand{\abdD}{\mathcal{A}_{\lambda}^{\leq d}}

\newcommand{\abdhk}{\mathcal{A}_{\lambda}^{\heartsuit}(\bar{k})}

\newcommand{\abdk}{\mathcal{A}_{\lambda}(\bar{k})}

\newcommand{\abdh}{\mathcal{A}_{\lambda}^{\heartsuit}}
\newcommand{\abdhp}{\mathcal{A}_{\lambda}^{\heartsuit,+}}

\newcommand{\abdd}{\mathcal{A}_{\lambda}^{\diamondsuit}}
\newcommand{\abddk}{\mathcal{A}_{\lambda}^{\diamondsuit}(\bar{k})}
\newcommand{\fx}{F_{x}}
\newcommand{\fv}{F_{v}}

\newcommand{\grl}{\overline{\Gr}_{\la}}

\newcommand{\bfv}{\bar{F}_{v}}

\newcommand{\La}{\Lambda}
\newcommand{\ab}{\mathbb{A}}
\newcommand{\bF}{\mathbb{F}}

\newcommand{\abdank}{\mathcal{A}_{\lambda}^{ani}(\bar{k})}

\title{Géométrisation du lemme fondamental pour l'algèbre de Hecke}
\author{Alexis Bouthier}

\begin{document}

\maketitle

\tableofcontents
\selectlanguage{francais}
\textbf{Abstract:}
This article is the third one of the series \cite{Bt1}-\cite{Bt2} on Hitchin-Frenkel-Ngô fibration and Vinberg semigroup. Ngô \cite{N} proved the fundamental lemma for Lie algebras in equal characteristics as a consequence of geometric stabilization. This article shows the geometric stabilization in the group case which was conjectured by Frenkel and Ngô \cite{FN}. Along the proof, we establish an identity between orbital integrals, which is analog to Langlands-Shelstad fundamental lemma. From this equality, we deduce a formula for Langlands-Shelstad transfer factors which was previously only known for Lie algebras.
\bigskip

\textbf{Résumé:}
Cet article est le troisième de la série \cite{Bt1}-\cite{Bt2} sur la fibration de Hitchin-Frenkel-Ngô et le semigroupe de Vinberg. Ngô \cite{N} a démontré le lemme fondamental pour les algèbres de Lie en égales caractéristiques comme conséquence de la stabilisation géométrique. Cet article s'attache à démontrer la stabilisation géométrique dans le cas des groupes qui a été conjecturé par Frenkel et Ngô \cite{FN}. Il permet également d'en déduire une identité entre intégrales orbitales analogue au lemme fondamental de Langlands-Shelstad pour l'algèbre de Hecke. Cette identité nous permet en particulier d'obtenir une formule pour les facteurs de transfert pour les groupes, connue auparavant que pour les algèbres de Lie.
\bigskip

\section*{Introduction}
\subsection{La fibration de Hitchin-Frenkel-Ngô}
Cette fibration est l'outil de base pour entreprendre une étude géométrique des intégrales orbitales pour les groupes. Elle a été introduite par Frenkel et Ngô dans \cite{FN}.
\medskip

Soit $k$ un corps algébriquement clos ou fini. Soit $X$ une courbe projective lisse géométriquement connexe, $F$ son corps de fonctions. Pour  cette introduction, on considère $G$ semisimple, simplement connexe et déployé sur $k$. Soit $(B,T)$ une paire de Borel. On note $X_{*}(T)$ le réseau des cocaractères et $X_{*}(T)^{+}$ le cône des cocaractères dominants. On considère $\chi:G\rightarrow T/W$, le morphisme issu du théorème de Chevalley \cite[Thm. 6.1]{S}, dit  polynôme caractéristique. Comme $G$ est simplement connexe c'est un espace affine de dimension $r$ dont les coordonnées sont les $\chi_{i}:=\Tr(\rho_{\rho_{\omega_{i}}})$ avec $\rho_{\omega_{i}}$ la représentation irréductible de plus haut poids $\omega_{i}$.
On se donne une somme formelle sur les points fermés $\la=\sum\limits_{x\in X}\la_{x}[x]$ avec les $\la_{x}\in X_{*}(T)^{+}$ nuls presque partout. Posons $S=\supp(\la):=\{x\in X~\vert~\la_{x}\neq 0\}$.
Partant de telles données, on peut construire la fibration de Hitchin-Frenkel-Ngô:
\begin{center}
$f:\cmd\rightarrow\abd$.
\end{center}
L'espace total $\cmd$ classifie les paires $(E,\phi)$ constituées d'un $G$-torseur $E$ sur $X$ et d'un automorphisme de $E$ en dehors de $S$, avec des pôles bornés par $\la_{s}$ en chaque point $s\in S$. La base de Hitchin $\abd$ est l'espace affine classifiant les polynômes caractéristiques des sections $\phi$; quant à la flèche $f$, elle est donnée par le polynôme caractéristique de la section $\phi$.
On a la description adélique suivante de la fibration: 
\begin{center}
$\cmd(k):=G(F)\backslash\{(\g,(g_{x}))\in G(F)\times G(\ab)/G(\co_{\ab})\vert~ g_{x}^{-1}\g g_{x}\in\overline{G(\co_{x})\pi_{x}^{\la}G(\co_{x})}\}$,
\end{center}
où $G(F)$ agit par $h.(\g,(g_{x}))=((h\g h^{-1},(hg_{x})))$. La base de Hitchin admet la description suivante:
 \begin{center}
 $\abd(k)=\{(a_{1},\dots,a_{r})\in F^{r}\vert~ \forall~ x, a_{i}\in\pi_{x}^{-\left\langle \omega_{i},-w_{0}\la\right\rangle}\co_{x}\}$.
 \end{center}
Enfin, la flèche $f$ revient à considérer le $r$-uplet $(\chi_{i}(\g))_{1\leq i\leq r}$.
\medskip

\subsection{La stratégie de Ngô}
Partant de la fibration de Hitchin usuelle $f_{D}:\cmD\rightarrow\abD$, où $D$ est un diviseur sur la courbe, Ngô montre qu'il existe un certain ouvert $\abDan$, dit anisotrope, au-dessus duquel la fibration $f_{D}^{ani}:\cmDan\rightarrow\abDan$ est propre. Comme de surcroît, l'espace source est lisse, nous savons par Deligne et Beilinson-Bernstein-Deligne-Gabber que le complexe $f_{D,*}^{ani}\overline{\mathbb{Q}}_{l}$ est pur et semisimple. Tout le jeu est alors d'étudier les supports des faisceaux pervers qui interviennent dans cette décomposition.

Au groupe $G$, Langlands \cite{Lan} et Kottwitz \cite {Kot3} associent une certaine famille de groupes que l'on appelle des groupes endoscopiques.
A chaque groupe endoscopique $H$, la base $\mathcal{A}_{D,H}^{ani}$ s'identifie à un sous-schéma fermé de $\abDan$. On note $\nu_{H}:\mathcal{A}_{D,H}^{ani}\rightarrow\abDan$ l'immersion fermée.
De plus, sur le complexe $f_{D,*}^{ani}\bql$, nous avons une action naturelle d'un groupe fini $\pi_{0}(P)$, qui nous permet de définir pour chaque caractère $\kappa:\pi_{0}(P)\rightarrow\overline{\mathbb{Q}}_{l}$, une composante $\kappa$-isotypique $(f_{D,*}^{ani}\bql)_{\kappa}$.
Si $\kappa=1$, on note $(f_{D,*}^{ani}\bql)_{st}$, cette composante et on l'appelle la composante stable.
Pour chaque $\kappa$, on a un ensemble fini $S_{\kappa}$, canoniquement associé, constitué de groupes endoscopiques et l'énoncé de stabilisation géométrique de Ngô \cite[Thm.6.4.1]{N} établit qu'au-dessus de $\abdan$, on a un isomorphisme de complexes sur $\bar{k}$, avec un certain décalage:
\begin{center}
$(f_{D,*}^{ani}\bql)_{\kappa}[2r](r)=\bigoplus\limits_{H\in S_{\kappa}}\nu_{H,*}(f_{D,*}^{H,ani}\bql)_{st}$.
\end{center}
\medskip

\subsection{Le théorème principal}
Dans le cas qui nous occupe, nous disposons également d'un ouvert anisotrope $\abdan$ au-dessus duquel la fibration $f^{ani}:\cmdan\rightarrow\abdan$ est propre.
Néanmoins, en général, l'espace total $\cmdan$ n'est pas lisse et donc $f_{*}^{ani}\overline{\mathbb{Q}}_{l}$ n'a pas de raison apparente d'être pur. On remplace donc le faisceau constant par le complexe d'intersection $IC_{\cmdan}$ et l'on s'interroge sur les supports qui interviennent dans $f_{*}^{ani}IC_{\cmdan}$.
Pour obtenir un tel énoncé pour la fibration de Hitchin-Frenkel-Ngô, nous avons d'abord besoin de contrôler le complexe d'intersection.
On considère alors un ouvert $\abdb\subset\abd$ et si de plus, $k$ est fini, on démontre dans \cite{Bt2} un énoncé de transversalité dont un des corollaires est l'identité suivante:
\begin{center}
$\forall~ a\in\abdban(k),\Tr(\Fr_{a},(f_{*}^{ani}IC_{\cmdan})_{a})=\so_{a}(\phi_{\la})$.
\end{center}
où $\so_{a}$ désigne l'intégrale orbitale stable, avec $\phi_{\la}:=\bigotimes\limits_{x\in X}\phi_{\la_{x}}$, où les fonctions $\phi_{\la_{x}}$ sont les fonctions de Kazhdan-Lusztig de l'algèbre de Hecke sphérique en $x$ $\cH_{G,x}$. Ce théorème nous permet en particulier de relier le complexe d'intersection de $\cmdb$ avec la grassmannienne affine et l'ouvert considéré est suffisamment gros pour des applications locales.
On se place alors sur l'ouvert $\abdanf$. Pour simplifier les notations, on note $IC_{\la}:=IC_{\cmdb}$, $f:=f^{ani,\flat}$ et pour les groupes endoscopiques on ajoute un exposant \og H\fg. Le théorème principal est le suivant:
\begin{thm}\label{1}
Soit $G$ un groupe connexe réductif déployé tel que $G_{der}$ est simplement connexe. On suppose que la caractéristique de $k$ est première à l'ordre du groupe de Weyl $W$ et que l'endoscopie est déployée. Alors nous avons un isomorphisme de complexes purs et semisimples sur $\bar{k}$:
\begin{center}
$(f_{*}IC_{\la})_{\kappa}=\nu_{H,*}(f_{*}^{H}S_{\eta^{*}\la,H})_{st}$,
\end{center}
où $S_{\eta^{*}\la,H}$ est un faisceau pervers pur, obtenu par l'équivalence de Satake géométrique à l'aide du morphisme $\eta:\cH_{G}\rightarrow\cH_{H}$ au niveau des algèbres de Hecke sphériques (cf. section \ref{stabI}).
\end{thm}
Nous nous  sommes placés sous les mêmes hypothèses que Chaudouard-Laumon \cite{CL}, cela a l'avantage d'alléger et de simplifier la présentation; dans ce cas l'ensemble $S_{\kappa}$ est réduit à un singleton. Pour étendre les résultats au cas quasi-déployé, la différence notable est que l'on a besoin de choisir un relèvement de la flèche $\hat{\eta}:\hat{H}\rightarrow\hat{G}$ aux duaux de Langlands $\mathstrut^{L}\eta:\mathstrut^{L}H\rightarrow \mathstrut^{L}G$ et qu'il n'y a pas de choix canonique. Cet énoncé a été en partie conjecturé par Frenkel et Ngô \cite[Conj. 4.2]{FN} dans le cadre de leur programme sur la géométrisation de la formule des traces. Il doit être une première pierre pour établir de nouveaux cas de la fonctorialité.

Pour établir un tel énoncé, on commence par établir un théorème de support, puis il nous  suffit ensuite de compter les points, c'est lors de ce comptage qu'apparaît une identité similaire au lemme fondamental de Langlands-Shelstad.
\subsection{Une identité entre intégrales orbitales}
On suppose $k$ fini, $\co=k[[t]]$, $F$ son corps de fractions. On garde les hypothèses du théorème \ref{1} et on pose $K=G(\co)$. On dispose de l'algèbre de Hecke sphérique $\cH_{G}$ constituée des fonctions à support compact $K$-invariantes à droite et à gauche, munie du produit de convolution. D'après Kazhdan-Lusztig, elle admet une base naturelle formée par les fonctions $\phi_{\la}$ qui correspondent par le dictionnaire faisceaux-fonctions au complexe d'intersection des strates $\overline{Kt^{\la}K}/K$, pour un cocaractère dominant $\la\in X_{*}(T)^{+}$. Soit $\kappa$ une donnée endoscopique de $G$ (cf. Déf.\ref{endos1}), $H$ le groupe endoscopique correspondant.
Soit $\cH_{H}$ l'algèbre de Hecke pour $H$, par l'isomorphisme de Satake, on a un morphisme naturel 
\begin{center}
$b:\cH_{G}\rightarrow\cH_{H}$
\end{center}
dit de changement de base. Le théorème est alors le suivant:
\begin{thm}\label{2}
Soit un cocaractère dominant $\la\in X_{*}(T)^{+}$, on a l'égalité:
\begin{center}
$(-1)^{\val(\mathfrak{D}_{H}(a_{H}))}\Delta_{H}(a_{H})\so_{a_{H}}(b(\phi_{\la}))=(-1)^{\val(\mathfrak{D}_{G}(a))}\Delta_{G}(a)\bo^{\kappa}_{a}(\phi_{\la})$
\end{center}
associées aux classes de conjugaison fortement régulières semisimples $a$ et $a_{H}$ de $G(F)$ et $H(F)$ qui se correspondent, avec $\Delta(a)=q^{-\frac{\val(\mathfrak{D}_{G}(a))}{2}}$, $\Delta_{H}(a_{H})=q^{-\frac{\val(\mathfrak{D}_{H}(a_{H}))}{2}}$ et $\kD_{G}$, $\kD_{H}$ les fonctions discriminants de $G$ et $H$.
\end{thm}
\subsection{Les facteurs de transfert}
Expliquons le lien avec le lemme fondamental de Langlands-Shelstad.
Il a été démontré par Ngô \cite{N} et Waldspurger \cite{W1} (sans restriction aucune) et établit l'égalité suivante:
\begin{center}
$\so_{\g_{H}}(\phi_{\la}^{H})=\Delta(\g_{H},\g)\bo^{\kappa}_{\g}(\phi_{\la})$
\end{center}
avec des classes de conjugaison fortement régulières semisimples $\g$, $\g_{H}$ dans $G(F)$ et $H(F)$ qui se correspondent et $\Delta(\g_{H},\g)$ est le facteur de transfert de Langlands-Shelstad \cite{LS}. Il s'écrit comme une certaine puissance de $q$ bien connue, multipliée par un signe compliqué. Ce facteur de transfert prend en compte le fait que la $\kappa$-intégrale orbitale dépend du choix de $\g$ dans sa classe de conjugaison stable. Dans le cas des algèbres de Lie, on sait d'après Kottwitz \cite{Kot} qu'en prenant la section de Kostant \cite{Ko}, le signe est égal à un. En revanche, dans le cas des groupes, cela n'est pas connu.
On note $\eps:T/W\rightarrow G$ la section de Steinberg \cite[Thm 8.1]{S}. On conserve les hypothèses du théorème \ref{1}, la conjonction du théorème \ref{2} et du lemme fondamental de Langlands-Shelstad nous donne alors l'identité suivante:
\begin{thm}\label{3}
Soit $\g$ et $\g_{H}$ des classes de conjugaison comme ci-dessus. Posons $a=\chi(\g)$ (resp. $a_{H}=\chi_{H}(\g_{H})$) le polynôme caractéristique de $\g$ (resp. $\g_{H}$) et $\g_{0}=\eps(a)\in G(F)$;
\begin{center}
$\Delta(\g_{H},\g_{0})=(-1)^{\val(\kD_{G}(a))-\val(\kD_{H}(a_{H}))}$.
\end{center}
\end{thm}
Cet égalité corrige et précise une question de Ngô \cite[sect. 2.2]{N3}. On devrait pouvoir également la prouver directement en reprenant les formules de Langlands-Shelstad \cite{LS}.
\bigskip

\subsection{Ce qui diffère des algèbres de Lie}
Pour la fibration de Hitchin-Frenkel-Ngô, les premières originalités concernent l'action d'un champ de Picard et l'étude du complexe d'intersection de l'espace total lesquelles ont été étudiées dans \cite{Bt1} et \cite{Bt2}.
A ces premières difficultés s'ajoutent celles du présent article dont nous mentionnons les principales.

Tout d'abord pour établir un théorème du support, nous devons généraliser l'énoncé de Ngô \cite[sect. 7]{N} au cas d'un espace source singulier où l'on remplace le faisceau constant par le complexe d'intersection.
Pour ce faire, nous introduisons une forme faible de résolution des singularités qui, combinée au théorème de transversalité établi  dans \cite{Bt2}, nous permet de démontrer un tel énoncé. Cette forme faible de résolution des singularités contient la cohomologie d'intersection comme facteur direct et nous permet d'appliquer des résultats de dégénérescence des suites spectrales.
Un autre outil pour pouvoir appliquer cet énoncé abstrait du théorème du support est la stratification à $\delta$ constant. Cette stratification doit vérifier des inégalités de codimension, connues grâce à Goresky-Kottwitz-McPherson \cite{GKM2} pour les algèbres de Lie, mais pas dans le cas des groupes.

De plus, pour établir une stabilisation géométrique, il est nécessaire d'obtenir des résultats pour un groupe connexe réductif général. Néanmoins, la fibration ne se comporte bien que pour les groupes connexes réductifs à groupe dérivé simplement connexe. En effet, dans le cas général, la fibration n'admet pas de champ de Picard, pas de section, la base de Hitchin est  singulière et nous n'avons pas de stratification à $\delta$ constant. En particulier, certaines fibres de Hitchin peuvent être vides.

D'ordinaire, en considérant des $z$-extensions, au niveau des fonctions, on peut se ramener au cas où $G_{der}$ simplement connexe, néanmoins du point de vue de la géométrie, cela ne fonctionne pas et la base de Hitchin pour une $z$-extension $G'$ de $G$ n'est même pas surjective sur la base de Hitchin de $G$.
On commence donc par montrer qu'en se restreignant au-dessus de l'ouvert anisotrope, la fibration de Hitchin-Frenkel-Ngô devient surjective et nous construisons une \og fibration de Hitchin augmentée\fg~ qui, si $G$ est semisimple, est une interpolation entre $G$ et son revêtement simplement connexe.
Sur cette fibration, on montre alors que nous avons l'action d'un champ de Picard et une stratification à $\delta$ constant avec les bonnes dimensions qui nous permettent d'étendre les résultats pour un groupe quelconque.

Enfin, l'identité obtenue entre ces intégrales orbitales nous permet d'en déduire une formule pour les facteurs de transfert qui n'était pas connue dans le cas des groupes et simplifie la stabilisation de la formule des traces dans le cas où $G_{der}$ est simplement connexe avec une endoscopie déployée.
\bigskip

Passons en revue l'organisation de l'article. Les quatre premiers chapitres rappellent les résultats établis dans \cite{Bt1} et \cite{Bt2}. On commence par étudier le semi-groupe de Vinberg et le quotient adjoint. Dans le chapitre suivant, on introduit la fibration de Hitchin-Frenkel-Ngô munie de l'action d'un champ de Picard. Dans les chapitres trois et quatre, on énonce le théorème de transversalité qui identifie le complexe d'intersection de $\cmd$ au-dessus de $\abdb$ et nous rappelons les résultats de dimension sur les fibres de Springer affines pour les goupes, ainsi que l'analyse locale des symétries d'icelles.
\medskip

Le cinquième chapitre étudie l'action du champ de Picard sur la fibration de Hitchin et passe en revue un certain nombre de propriétés générales de la fibration de Hitchin. On commence par introduire le revêtement caméral et étudions comment le centralisateur régulier s'exprime en fonction de ce revêtement.
Par la suite, nous prouvons qu'en considérant le modèle de Néron du centralisateur régulier, cela nous mène naturellement à un dévissage du champ de Picard en une partie affine et une partie abélienne. En un sens, les singularités des fibres de Hitchin sont contenues dans la partie affine, dont la taille sera un invariant incontournable pour stratifier la base de Hitchin. On calcule également la dimension de la base ainsi que celle du champ de Picard et établissons la densité de l'ouvert régulier en utilisant une formule du produit analogue à celle de Ngô \cite[sect. 4.15]{N}.
\medskip

Le sixième chapitre introduit une stratification par des invariants monodromiques qui nous permet d'en déduire une description explicite du groupe des composantes connexes du champ de Picard, qui contrôle la cohomologie ordinaire de degré maximal des fibres de Hitchin, indispensable dans la suite pour le théorème du support. Cette partie est sans doute la plus proche des énoncés de Ngô.
\medskip

Le septième chapitre est intimement lié au chapitre précédent. Il s'agit d'étudier la codimension des strates à $\delta$ constant. On étudie alors une stratification plus fine, inspirée de Goresky-Kottwitz-McPherson \cite{GKM2}, par les valuations radicielles. Elle nous permet d'obtenir l'énoncé de codimension quitte à ce que l'invariant $\delta$ soit suffisamment petit par rapport à $\la$. Il est à noter qu'en caractéristique zéro pour les algèbres de Lie, on arrive à obtenir un résultat optimal, à savoir que la codimension est précisément $\delta$. En revanche, dans notre cas, même en caractéristique nulle, nous avons besoin de l'hypothèse que $\la$ soit grand devant $\delta$.
\medskip

Le huitième chapitre s'attelle à l'étude du lieu anisotrope; on  démontre la propreté de la fibration de Hitchin au-dessus de $\abdan$ et expliquons comment apparaissent les groupes endoscopiques. Cela nous permet de formuler un énoncé cohomologique sur les supports qui interviennent dans la décomposition en faisceaux pervers irréductibles du complexe pur $f_{*}^{ani}(IC_{\cmd})$.
Essentiellement, sous certaines conditions techniques, on montre que ces supports correspondent aux strates issues des groupes endoscopiques. On démontre ce théorème dans le chapitre onze.
\medskip

Le neuvième chapitre étend tous les résultats nécessaires pour un groupe qui n'est pas nécessairement à groupe dérivé simplement connexe.
Il est à noter que la fibration de Hitchin-Frenkel-Ngô, dans ce degré de généralité, est plutôt incommode. Un exemple est que l'on ne dispose ni de sections, ni de champ de Picard qui agit. On introduit donc une fibration augmentée qui est finie surjective au-dessus d'elle et qui elle, admet les propriétés escomptées.
\medskip

Le dixième chapitre est consacré à la construction de ce que nous appelons une présentation géométrique du complexe d'intersection de $\cmdb$. Il consiste en un espace qui remplace avantageusement une résolution des singularités, qui contient dans sa cohomologie, la cohomologie du complexe d'intersection comme facteur direct et qui nous sera utile pour formuler le théorème du support dans un contexte singulier. Il s'obtient par changement de base d'une résolution des singularités d'un certain  champ de Hecke augmenté $\overline{\cH}_{\la}^{par}$.
Le changement de base propre ainsi que le théorème de transversalité nous assureront que le complexe d'intersection $IC_{\cmdb}$ apparaît bien comme facteur direct de la cohomologie de la présentation géométrique.
\medskip

Le onzième chapitre démontre un énoncé abstrait pour une fibration $f:M\rightarrow S$ lorsque $M$ est singulier, muni d'une action d'un schéma en groupes commutatif lisse $g:P\rightarrow S$, pour le faisceau $f_{*}IC_{M}$. Il généralise l'énoncé de Ngô dans un contexte singulier. La preuve du théorème abstrait diffère de Ngô à deux endroits, l'argument de dualité de Goresky-McPherson et la liberté ponctuelle de la cohomologie des fibres sur l'algèbre de Pontryagin du module de Tate du champ de Picard $\cP_{a}$ pour $a\in\abdb$.
L'argument de dualité s'étend  pour peu qu'on ajoute des inégalités sur la dimension des strates tandis que la liberté ponctuelle s'obtient en utilisant la présentation géométrique.
On applique ensuite ce théorème à la fibration de Hitchin pour en déduire le théorème de détermination des supports \ref{detsupport}.
\medskip

Le dernier chapitre se consacre au comptage. On compte les points des fibres de Springer affines et des fibres de Hitchin et on les relie aux intégrales orbitales. On démontre dans ce chapitre les énoncés clés \ref{1}, \ref{2}.
\bigskip

Je remercie sincèrement Ngô Bao Châu et Gérard Laumon pour l'aide qu'ils m'ont apporté tout au long de ce travail.
Je remercie Jean-Loup Waldspurger pour m'avoir signalé une erreur sur les facteurs de transfert dans une version antérieure de ce papier. Je remercie également Yakov Varshavsky pour les discussions que nous avons eu ainsi que Raphaël Beuzart-Plessis pour son utile remarque sur les facteurs de transfert. Enfin, ce travail a été réalisé en partie lors de mon séjour au MSRI à l'automne 2014 et à l'Université Hébraïque de Jerusalem, je leur exprime ma reconnaissance pour les conditions de travail fort agréables qu'ils ont pu me fournir.
\bigskip

\section{Le semi-groupe de Vinberg}
\subsection{Le quotient adjoint et le centralisateur régulier}\label{introsemi}
Dans cette section on rappelle les résultats établis dans \cite{Bt1} et \cite{Bt2}, on commence par le semi-groupe de Vinberg.
\medskip

Soit $k$ un corps. On considère un groupe connexe réductif $G$  tel que $G_{der}$ soit simplement connexe, déployé sur $k$.  Soit $(B, T)$ une paire de Borel de $G$, $R$ l'ensemble des racines, $\Delta$ l'ensemble des racines simples, $r$ le rang semisimple de $G$ et $W$ le groupe de Weyl. On note $X_{*}(T)$ (resp. $X^{*}(T)$) le réseau des cocaractères (resp. caractères). Enfin, on note d'un exposant \og+\fg~ les cocaractères dominants (resp. caractères dominants).
On commence par construire le semi-groupe de Vinberg $V_{G}$ sur $k$.
Choisissons une base $\omega_{1}',\dots,\omega_{l}'$ du réseau des caractères:
\begin{center}
$\chi:T\rightarrow\mathbb{G}_{m}$
\end{center}
tels que
\begin{center}
$\forall~\alpha\in\Delta=\{\alpha_{1},\dots,\alpha_{r}\}, \left\langle \chi,\check{\alpha}\right\rangle=0.$
\end{center}
Chaque $\omega_{i}':T\rightarrow\mathbb{G}_{m}$, $1\leq i\leq l$, se prolonge de manière unique en un caractère:
\begin{center}
$\omega_{i}':G\rightarrow\mathbb{G}_{m}$.
\end{center}
Pour tout indice $i, 1\leq i\leq l$, notons $V_{\omega_{i}'}$ l'espace vectoriel de dimension un sur lequel $G$ agit par $\omega_{i}':G\rightarrow\mathbb{G}_{m}$.
L'ensemble des poids dominants est stable par translation par les éléments du réseau $\mathbb{Z}\omega_{1}'+\dots+\mathbb{Z}\omega_{l}'$. Le quotient par ce réseau est un cône saturé non dégénéré de $X^{*}(T)/(\mathbb{Z}\omega_{1}'+\dots+\mathbb{Z}\omega_{l}')$, engendré par $\bar{\omega}_{1},\dots,\bar{\omega}_{r}$ que l'on relève en une famille de caractères $\omega_{1},\dots,\omega_{r}$ de $T$.

Soit $(\rho_{\omega_{i}}, V_{\omega_{i}})$ la représentation irréductible de plus haut poids $\omega_{i}$.
Considérons $G_{+}:=(T\times G)/Z_{G}$ où $Z_{G}$ se plonge par $\la\rightarrow (\la,\la^{-1})$ et $T_{+}=(T\times T)/Z_{G}$.
On a une immersion :
$$\begin{array}{ccccc}
&  & G_{+} & \to & \prod\limits_{i=1}^{r}\End(V_{\omega_{i}})\times\prod\limits_{i=1}^{l}\Aut(V_{\omega_{i}'})\times\prod\limits_{i=1}^{r}\mathbb{A}_{\alpha_{i}} \\
& & (t,g) & \mapsto & (\omega_{i}(t)\rho_{\omega_{i}}(g),\omega_{i}'(tg),\alpha_{i}(t)).\\
\end{array}$$
Ici on pose, $H_{G}:=\prod\limits_{i=1}^{r}\End(V_{\omega_{i}})\times\prod\limits_{i=1}^{l}\Aut(V_{\omega_{i}'})\times\prod\limits_{i=1}^{r}\mathbb{A}^{1}_{\alpha_{i}}$
et $H_{G}^{0}$ sera la même chose où l'on enlève $\{0\}$ dans chaque $\End(V_{\omega_{i}})$.
On note alors $V_{G}$ (resp $V_{G}^{0}$) la normalisation de l'adhérence de $G_{+}$ dans $H_{G}$ (resp. $H_{G}^{0})$ et $V_{T}$ l'adhérence de $T_{+}$ dans $V_{G}$, lequel est normal par \cite[Cor. 6.2.14]{BK}.
On a un théorème analogue à celui de Chevalley \cite[Prop. 1.3]{Bt1}:
\begin{thm}\label{bouth}
L'application de restriction $\phi:k[V_{G}]^{G}\rightarrow k[V_{T}]^{W}$ est un isomorphisme de $k$-algèbres. De plus, $V_{T}/W$ est le produit d'un espace affine de dimension $2r$ avec un tore de dimension $l$, dont les coordonnées sont données par les $\omega_{i}'$, $(\alpha_{i},0)$, $\chi_{i}=(\omega_{i},\Tr(\rho_{\omega_{i}}))$.
\end{thm}
On en déduit alors un morphisme 
\begin{center}
$\chi_{+}: V_{G}\rightarrow\kC_{+}:=V_{T}/W$.
\end{center}
Nous avons également une flèche $\chi: G\rightarrow T/W=\mathbb{G}_{m}^{l}\times\mathbb{A}^{r}$, issue du théorème de Chevalley \cite[Th.6.1]{S}.
Steinberg a construit une section à cette flèche de la manière suivante; on commence par supposer que $G$ est semisimple simplement connexe, dans ce cas $T/W=\mathbb{A}^{r}$.
Pour un $r$-uplet $(a_{1},\dots, a_{r})\in T/W:=\mathbb{A}^{r}$, on définit:
\begin{center}
$\epsilon(a_{1},.., a_{r}):=\prod\limits_{i=1}^{r}x_{\alpha_{i}}(a_{i})n_{i}$,
\end{center}
où les $x_{\alpha_{i}}(a_{i})$ sont des éléments du groupe radiciel $U_{\alpha_{i}}$ et les $n_{i}$ sont des éléments du normalisateur $N_{G}(T)$ représentant les réflexions simples $s_{\alpha_{i}}$ de $W$.\\
Ainsi, $\epsilon(a)\in\prod\limits_{i=1}^{r}U_{i}n_{i}$ et en utilisant les relations de commutation on a que :
\begin{center}
$\prod\limits_{i=1}^{r}U_{i}n_{i}=U_{w}w$
\end{center}
où $w=s_{1}s_{2}...s_{r}$ et $U_{w}=U\cap wU^{-}w^{-1}$.
\medskip

Voyons comment on l'étend au cas réductif. Soit $T_{der}=T\cap G_{der}$.
Soit $S$ un sous-tore de $T$ de telle sorte que $T=S\times T_{der}$, alors $G=S\ltimes G_{der}$.
La flèche de Steinberg est donnée par:
\begin{center}
$\chi:G\rightarrow S\times\mathbb{A}^{r}$,
\end{center}
où $\mathbb{A}^{r}$ est la partie correspondant au quotient adjoint de $G_{der}$.
Soit $G^{reg}:=\{(g,\g)\in G\times G\vert~ g\g g^{-1}=\g\}$.
Nous avons alors le théorème suivant \cite[Th. 8.1]{S} et \cite[Prop. 2.5]{dC-M}:
\begin{prop}
Soit $\epsilon_{G_{der}}$ la section de Steinberg pour $G_{der}$.
On pose alors $\eps_{G}:S\times\mathbb{A}^{r}\rightarrow G$ donnée par
\begin{center}
$\eps_{G}(s,a)=s\epsilon_{G_{der}}(a)$.
\end{center}
Alors $\eps_{G}$ est une section à $\chi$ et tombe dans $G^{reg}$, à un automorphisme de $S\times\mathbb{A}^{r}$ près.
\end{prop}
On considère maintenant le schéma en groupes des centralisateurs:
\begin{center}
$I:=\{(g,\g)\in G\times V_{G}\vert~ g^{-1}\g g=\g\}$.
\end{center}
Voyons  comment on construit une section pour le semigroupe de Vinberg $V_{G}$.
Soit la flèche $\phi:T\rightarrow T_{+}$, $t\rightarrow (t,t^{-1})$. L'image $T_{\Delta}$, le tore antidiagonal est isomorphe à $T^{ad}:=T/Z_{G}$, et on a un isomorphisme canonique donné par:
\begin{center}
$\alpha_{\bullet}:T_{\Delta}\rightarrow\mathbb{G}_{m}^{r}$.
\end{center}
Soit alors $\psi$ l'isomorphisme inverse, nous opterons pour la notation indiciaire. Ainsi pour $b\in \mathbb{G}_{m}^{r}$, on a: 
\begin{center}
$\forall~ 1\leq i\leq r, \alpha_{i}(\psi_{b})=b_{i}$ et $\psi_{b}\in T_{\Delta}$.
\end{center}
On définit $\epsilon_{+}:\mathbb{G}_{m}^{r}\times (\mathbb{G}_{m}^{l}\times\mathbb{A}^{r})\rightarrow G_{+}$ par: 
\begin{center}
$\epsilon_{+}(b,a)=\epsilon_{G}(a)\psi_{b}$.
\end{center}
Dans la suite, on pose $\cQ_{+}:=\cQ T_{\Delta}$ où $\cQ:=\eps_{G}(T/W)$.
On obtient alors le théorème de structure suivant tiré de \cite[Thm.0.2-0.3]{Bt1}:
\begin{thm}\label{bouth2}
Soit $V_{G}^{reg}:=\{\g\in V_{G}~\vert~\dim I_{g}=r\}$.
La section  $\eps_{+}$ se prolonge en un morphisme $\eps_{+}:\kC_{+}\rightarrow V_{G}^{reg}$.
De plus, le morphisme $\chi_{+}^{reg}$ est lisse et ses fibres géométriques sont une union disjointe de $G$-orbites.
Enfin, il existe un unique schéma en groupes commutatifs $J$, lisse sur $\kC_{+}$ et muni d'un morphisme $\chi_{+}^{*}J\rightarrow I$, qui est un isomorphisme sur $V_{G}^{reg}$.
\end{thm}

$\rmq$
\begin{enumerate}
\item
Il est à noter que la section $\eps_{+}$ prolongée au semigroupe de Vinberg dépend fortement de l'ordre des facteurs $U_{i}n_{i}$. En effet, un ordre différent amène à des sections qui ne sont plus nécessairement conjuguées, comme on peut le voir au point au point $(0,0)\in\kC_{+}$. Ceci explique le fait que les fibres géométriques de $\chi_{+}$ ne sont pas connexes en général.
\item
Dans \cite{Bt1}, la preuve est faite dans le cas semisimple mais s'étend telle quelle au cas réductif.
\end{enumerate}
On dispose également d'un morphisme, dit d'abélianisation:
\begin{center}
$\alpha: V_{G}\rightarrow A_{G}:=\mathbb{G}_{m}^{l}\times\mathbb{A}^{r}$
\end{center}
donné par $(t,g)\mapsto((\omega_{i'}(tg))_{1\leq i'\leq l}, (\alpha_{i}(t))_{1\leq i\leq r})$.
\bigskip

A la suite de Donagi-Gaitsgory \cite{DG} et Ngô \cite[sect. 2.4]{N}, nous avons une description galoisienne du centralisateur régulier.
On suppose de plus la caractéristique du corps $k$ est première à l'ordre de $W$.
On a un morphisme fini plat $W$-équivariant:
\begin{center}
$\theta:V_{T}\rightarrow\kc$
\end{center}
ramifié le long du diviseur $\mathfrak{D}=\bigcup\limits_{\alpha\in R}\overline{\Kern(\alpha)}$, où
$\overline{\Kern(\alpha)}$ désigne l'adhérence dans $V_{T}$ de $\Kern(\alpha)$ avec $\alpha$ une racine.
On note $\kc^{rs}$ le complémentaire de ce diviseur, que l'on appelle le lieu régulier semisimple et $V_{T}^{rs}$ l'ouvert de $V_{T}$ correspondant. Considérons le schéma:
\begin{center}
$\Omega:=\prod\limits_{V_{T}/\kc}(T\times V_{T})$.
\end{center}
Pour tout $S$-schéma sur $\kc$, les $S$-points de $\Omega$ sont donnés par:
\begin{center}
$\Omega(S)=\Hom_{V_{T}}(S\times_{\kc}V_{T}, T\times V_{T})=\theta_{*}(T\times V_{T})$.
\end{center}
Le morphisme $\theta:V_{T}\rightarrow\kc$ étant fini plat, nous obtenons que $\Omega$ est représentable. C'est un schéma en groupes commutatifs, lisse de dimension $r\left|W\right|$ et au-dessus de l'ouvert régulier semi-simple, $\theta$ étant fini étale, $\Omega$ restreint à cet ouvert est un tore. L'action diagonale de $W$ sur $T\times V_{T}$ induit une action sur $\Omega$.
On considère alors:
\begin{equation}
J^{1}=\Omega^{W}=(\prod\limits_{V_{T}/\kc}(T\times V_{T}))^{W}
\label{jun}
\end{equation}
Comme la caractéristique du corps est première avec l'ordre de $W$, $J^{1}$ est un schéma en groupes lisse sur $\kc$.
Nous allons avoir besoin de considérer un sous-schéma ouvert de $J^{1}$.
\begin{defi}
Soit $\tilde{J}$ le sous-foncteur de $J^{1}$ qui à tout $\kc$-schéma $S$, associe le sous-ensemble $\tilde{J}(S)$ de $J^{1}(S)$ des morphismes $W$-équivariants:
\begin{center}
$f:S\times_{\kc}V_{T}\rightarrow T$
\end{center}
tel que pour tout point géométrique $x$ de $S\times_{\kc}V_{T}$ stable sous une involution $s_{\alpha}(x)=x$ attachée à une racine $\alpha$, on a $\alpha(f(x))\neq -1$.
\end{defi}
Les deux résultats suivants sont tirés de \cite[Lem. 11, Prop.12]{Bt2}
\begin{lem}
Le foncteur $\tilde{J}$ est représentable par un sous-schéma ouvert affine de $J^{1}$. De plus, on a les inclusions $J^{0}\subset\tilde{J}\subset J^{1}$ et au-dessus de $\kc^{rs}$ ce sont des isomorphismes.
\end{lem}

\begin{prop}\label{galois}
Le morphisme $J\rightarrow J^{1}$ se factorise via $\tilde{J}$ et induit un isomorphisme entre $J$ et $\tilde{J}$.
\end{prop}

\subsection{Le cas d'un groupe connexe réductif}
Nous étendons dans cette section les propriétés établies pour un groupe $G$ tel que $G_{der}$ est simplement connexe à un groupe connexe réductif quelconque.
La différence notable est que dans ce cas, il n'y a ni section, ni centralisateur régulier. Enfin, la flèche $G_{reg}\rightarrow T/W$ n'est même pas plate.
\medskip

Soit  $G$ connexe réductif  déployé sur $k$. On suppose que la caractéristique est première à l'ordre de  $\pi_{1}(G)$, ce qui est toujours le cas si $\car(k)\wedge\left|W\right|=1$.
Soit $\tilde{G}$ un groupe connexe réductif tel que $\tilde{G}_{der}$ est simplement connexe, muni d'une flèche $\la:\tilde{G}\rightarrow G$, finie étale de groupe de Galois abélien $\Gamma$, par hypothèse sur la caractéristique. 
On peut prendre par exemple, le revêtement simplement connexe de $G_{der}$ que l'on multiplie par le centre.
On note pour cette raison le morphisme de Steinberg pour $\tilde{G}$,
\begin{center}
$\chi^{sc}:\tilde{G}\rightarrow\kc^{sc}$
\end{center}
On définit le semi-groupe de Vinberg de $G$ par $V_{G}:=V_{\tilde{G}}//\Gamma$. Ici, la double barre désigne le quotient au sens des invariants. On a de plus une action de $\Gamma$ sur l'espace affine $\kc^{sc}$, on considère alors $\kc:=\Spec(k[\kc^{sc}]^{\Gamma})$ et la flèche $\chi_{+}^{sc}$ passe au quotient. On désigne toujours par $G_{+}:=(T\times G)/Z_{G}$ le groupe des unités de $V_{G}$.
Nous avons une flèche finie et surjective $V_{\tilde{G}}\rightarrow V_{G}$, on note alors  $V_{G}^{reg}$ l'ouvert image de $V_{\tilde{G}}^{reg}$ par cette application. Nous obtenons le diagramme commutatif :
$$\xymatrix{V_{\tilde{G}}\ar[r]\ar[d]_{\chi_{+}^{sc}}&V_{G}\ar[d]^{\chi_{+}}\\\kc^{sc}\ar[r]&\kc}.$$
\begin{defi}
Si $G_{der}$ n'est pas simplement connexe, on considère $V_{G}^{f-rs}$ le lieu fortement régulier semisimple, i.e. le lieu où le centralisateur est un tore maximal. 
\end{defi}
Posons $\kc^{f-rs}:=\chi_{+}(V_{G}^{f-rs})$. On a le lemme suivant:
\begin{lem}\label{Pop}
Le lieu fortement régulier semisimple $\kc^{f-rs}$ est un ouvert non vide lisse de $\kc$.
\end{lem}
\begin{proof}
C'est la même preuve que \cite[Lem. 5.1]{P}.
\end{proof}
Si $G_{der}$ est simplement connexe, on pose 
\begin{center}
$V_{G}^{f-rs}:=p^{-1}(V_{G_{ad}}^{f-rs})$,
\end{center}
avec $p: V_{G}\rightarrow V_{G_{ad}}$.
\medskip

Nous terminons par une description du morphisme d'abélianisation dans le cas $G$ semisimple et une étude du discriminant.
\begin{center}
$\alpha:V_{G}\rightarrow A_{G}$.
\end{center}
 Nous avons le théorème suivant d'après  \cite[Thm. 3-5]{Vi} et  \cite[Thm. 17]{Ri}:
\begin{prop}\label{abelianise}
Soit $G$ semisimple, alors l'abélianisé $A_{G}=\overline{Z}_{+}/Z_{G}$ est isomorphe à $\mathbb{A}^{r}=\prod\limits_{i=1}^{r}\ab^{1}_{\alpha_{i}}$ comme $T_{ad}$-variété torique. En particulier, nous avons $A_{G}=A_{G_{ad}}$.
\end{prop}
$\rmq$ Si $G$ est réductif, on ajoute un tore central.

\begin{cor}\label{produit}
L'espace caractéristique $\kc$ s'identifie au produit $A_{G}\times\kC$ où $\kC=T/W$.
\end{cor}
\begin{proof}
Si $G$ est tel que $G_{der}$ est simplement connexe, cela résulte de la proposition \ref{bouth}. Pour  un groupe $G$ général, soit un groupe $G'$ tel que $G_{der}'$ est simplement connexe et un sous-groupe fini central groupe $\Gamma$ de $G'$ tel que $G'/\Gamma=G$. Nous avons en particulier une action de $\Gamma$ sur $\mathfrak{C}_{+,G'}$ tel que: 
\begin{center}
$\mathfrak{C}_{+,G'}//\Gamma=\mathfrak{C}_{+,G}$
\end{center}
De plus, il résulte de la proposition \ref{abelianise} que  $A_{G}=A_{G'}$ et que $\Gamma$ agit trivialement sur $A_{G}$, nous déduisons donc:
\begin{center}
$\mathfrak{C}_{+,G'}//\Gamma= (A_{G'}\times\kC_{G'})//\Gamma=A_{G}\times\kC_{G}=\mathfrak{C}_{+,G}$,
\end{center}
ce qui conclut.
\end{proof}
On considère la fonction $\mathfrak{D}_{+}=(2\rho,\mathfrak{D})\in T_{+}=(T\times T)/Z_{G}$ sur $k[T_{+}]$ où $\mathfrak{D}=\prod\limits_{\alpha\in R}(1-\alpha(t))$ est la fonction discriminant sur $k[T]$.
Comme $W$ agit trivialement sur le premier facteur, on a que $\mathfrak{D}_{+}\in k[T_{+}]^{W}$ et de plus on a la propriété suivante:
\begin{center}
$t\in T_{+}^{rs}\Longleftrightarrow \mathfrak{D}_{+}(t)\neq 0$,
\end{center}
où $T_{+}^{rs}$ est le lieu régulier semisimple du tore.
D'après \cite[Lem. 2.18]{Bt1}, cette fonction s'étend en une fonction de $k[V_{T}]^{W}$.

$\rmq$ Sur $\overline{T}_{\Delta}$, la fonction étendue s'écrit pour $t_{+}=(t,t^{-1})$:
\begin{center}
$\mathfrak{D}_{+}(t_{+})=2\rho(t)\prod\limits_{\alpha>0}(1-\alpha(t^{-1}))(1-\alpha(t))$
\end{center}
Comme $2\rho=\sum\limits_{\alpha>0}\alpha$, nous obtenons:
\begin{equation}
\mathfrak{D}_{+}(t_{+})=(-1)^{\left|R^{+}\right|}\prod\limits_{\alpha>0}(1-\alpha(t))^{2}
\label{extdisc}
\end{equation}
où $R^{+}$ est l'ensemble des racines positives.
Enfin, nous avons  le critère de régularité tiré de \cite[Prop. 2.19]{Bt1}; soit $t_{+}\in V_{T}$, alors nous avons l'équivalence:
\begin{center}
$t_{+}\in V_{T}^{rs}\Longleftrightarrow \mathfrak{D}_{+}(t_{+})\neq 0$,
\end{center}
où $V_{T}^{rs}$ est le lieu régulier semisimple de $V_{T}$.

\section{Constructions globales}
Soit $X$ une courbe projective lisse connexe sur un corps algébriquement clos $k$. On note $F$ son corps de fonctions, $\left|X\right|$ l'ensemble des points fermés.
Pour $x\in\left|X\right|$, soit $D_{x}=\Spec(\mathcal{O}_{x})$ le disque formel en $x$ et $D_{x}^{\bullet}=\Spec(\fx)$, le disque formel épointé, d'uniformisante $\pi_{x}$.

Soit $G$ un groupe connexe réductif sur $k$. 
On considère une paire de Borel $(B,T)$ de $G$ et on note $W$ le groupe de Weyl de $G$. On suppose désormais que l'ordre de $W$ est premier avec la caractéristique de $k$.

On note $X_{*}(T)^{+}$ l'ensemble des cocaractères dominants de $T$.
On désigne par $\Bun$ le champ des  $G$-torseurs sur $X$. Le champ $\Bun$ est algébrique au sens d'Artin (\cite{LM} et \cite[Prop. 1]{H}).

\subsection{La fibration de Hitchin-Frenkel-Ngô}
Dans cette construction, on construit la fibration qui sera l'objet d'étude principal, elle a été introduite par Frenkel-Ngô dans \cite[sect. 4.1]{FN}.

Pour tout point fermé $x\in X$, considérons le $k$-schéma en groupes $K_{x}=G(\co_{x})$ et le ind-schéma $G_{x}:=G(F_{x})$, on dispose de la grassmannienne affine en $x$, $\Gr_{x}=G_{x}/K_{x}$. Elle admet d'après Lusztig une structure d'ind-schéma ainsi qu'une stratification en $K_{x}$-orbites localement fermées:
\begin{center}
$\Gr_{x}=\coprod\limits_{\la\in X_{*}(T)^{+}}\Gr_{\la,x}$,
\end{center}
dite de Cartan, indexée par les cocaractères dominants $\la\in X_{*}(T)^{+}$. Pour $\la\in X_{*}(T)^{+}$, soit $\overline{\Gr}_{\la,x}$ l'adhérence de $\Gr_{\la,x}$ dans $\Gr_{x}$.
Plus généralement, étant donné un ensemble fini de points fermés $S\subset\left|X\right|$, nous notons 
\begin{center}
$\Gr_{S}:=\prod\limits_{s\in S}\Gr_{s}$.
\end{center}
Soit $\Div(X)$ le groupe abélien des diviseurs sur $X$. On considère alors le groupe $\Div(X,T)=\Div(X)\otimes_{\mathbb{Z}}X_{*}(T)$ des diviseurs à coefficients dans les cocaractères de $T$. On note $\Div^{+}(X,T)\subset\Div(X,T)$ le cône formé par les diviseurs dont les coefficients sont dans $X_{*}(T)^{+}$. On a alors une action de $W$ sur $\Div(X,T)$ induite par l'action de $W$ sur $X_{*}(T)$.
Pour $\la=\sum\limits_{x\in \left|X\right|}\la_{x}[x]\in \Div(X,T)$, on pose $S=\supp(\la):=\{x\in \left|X\right|\vert ~\la_{x}\neq 0\}$ et nous notons $\Gr_{\la}$ (resp. $\overline{\Gr}_{\la}$) le produit 
\begin{center}
$\prod\limits_{s\in S}\Gr_{\la_{s},s}$
\end{center}
(resp. $\prod\limits_{s\in S}\overline{\Gr}_{\la_{s},s}$).

\subsubsection{Construction d'un $T$-torseur}
Pour tout $\omega\in X^{*}(T)^{+}$ et $\mu\in\Div^{+}(X,T)$, l'accouplement
\begin{center}
$\left\langle\omega, \mu\right\rangle=\sum\limits_{x\in\left|X\right|}\left\langle\omega, \mu_{x}\right\rangle [x]$
\end{center}
définit un diviseur sur la courbe.

Pour $x\in \left|X\right|$, on rappelle que $X_{*}(T)=T(F_{x})/T(\co_{x})$, en particulier pour tout $\mu\in X_{*}(T)^{+}$, on peut lui associer un $T$-torseur sur $D_{x}$ avec une trivialisation générique.
Pour $\mu=\sum\limits_{x\in X}\mu_{x}[x]\in\Div^{+}(X,T)$, notons $S=\supp(\mu)$, considérons alors le $T$-torseur $E_{T}(\mu)$ sur $X$, qui consiste à recoller le $T$-torseur trivial en dehors de $S$ avec le $T$-torseur sur le voisinage formel des points de $S$, donné par $\mu$, lequel est muni canoniquement d'une trivialisation générique.
Ce recollement nous est donné par une généralisation de Beauville-Laszlo par Beilinson-Drinfeld \cite[sect. 2.3.7]{BD}.
Pour chaque racine positive $\alpha\in R^{+}$,  nous avons
\begin{center}
$E_{T}(\mu)\times^{T,\alpha}\mathbb{G}_{m}=\mathcal{O}_{X}(\left\langle\alpha, \mu\right\rangle)$.
\end{center}
Comme $\left\langle \alpha, \mu\right\rangle\geq 0$, on a une section canonique définie sur la courbe que l'on note $1_{\left\langle \alpha, \mu\right\rangle}$.
En se limitant aux racines simples $(\alpha_{i})_{1\leq i \leq r}$, nous obtenons  $r$-fibrés en droites munis de $r$-sections, d'où l'on déduit un morphisme 
\begin{center}
$\mu:X\rightarrow[A_{G}/Z_{+}]$,
\end{center}
où $Z_{+}$ est le centre de $G_{+}$ et où l'on rappelle que $A_{G}$ désigne la base du morphisme d'abélianisation 
\begin{center}
$\alpha:V_{G}\rightarrow A_{G}:=\mathbb{G}_{m}^{l}\times\mathbb{A}^{r}$
\end{center}
donné par les racines simples de $G$ et les caractères $\omega_{i'}$. On remarque que comme l'on divise par $Z_{+}$, la partie centrale disparaît.
Fixons $\la=\sum\limits_{x\in X}\la_{x}[x]\in \Div^{+}(X,T)$. Nous avons vu que $\Div(X,T)$ est muni d'une action du groupe de Weyl $W$, on considère alors l'élément $-w_{0}\la\in \Div^{+}(X,T)$ où $w_{0}$ est l'élément long du groupe de Weyl.
D'après ci-dessus, nous obtenons une flèche:
\begin{center}
$-w_{0}\la:X\rightarrow[A_{G}/Z_{+}]$.
\end{center}
Nous avons le morphisme d'abélianisation $\alpha:V_{G}\rightarrow A_{G}$ qui  est équivariant par rapport à l'action du centre $Z_{+}$ de $G_{+}$, d'où l'on obtient une flèche :
\begin{center}
$\alpha: [V_{G}/Z_{+}]\rightarrow [A_{G}/Z_{+}]$.
\end{center}
Posons  alors $V_{G}^{\la}:=(-w_{0}\la)^{*}[V_{G}/Z_{+}]$. Nous obtenons un espace fibré sur $X$.
Le semi-groupe de Vinberg admet un ouvert lisse $V_{G}^{0}$ (cf. sect.\ref{introsemi}) et on note de la même manière $V_{G}^{\la,0}$ l'espace tiré sur $X$ par la flèche $-w_{0}\la$.
\begin{defi}\label{hitchin}
L'espace de Hitchin $\cmd$ est le champ des sections 
\begin{center}
$\Hom(X, [V_{G}^{\la}/G])$
\end{center}
où $G$ agit sur $V_{G}^{\la}$ par conjugaison. 
Il classifie les couples $(E,\phi)$ avec $E$ un $G$-torseur sur $X$ et $\phi$ une section de l'espace fibré au-dessus de $X$
\begin{center}
$V_{G}^{\la}\times^{G}E$.
\end{center}
Nous avons également l'ouvert $\cmdo$ qui classifie les sections
\begin{center}
$h_{(E,\phi)}:X\rightarrow [V_{G}^{\la,0}/G]$.
\end{center}
\end{defi}
On définit également l'ouvert régulier $\cmdo^{reg}$ de $\cmdo$ par le champ des sections:
\begin{center}
$\cmdo^{reg}:=\Hom_{X}(X,[V_{G}^{\la,reg}/G])$.
\end{center}
avec $V_{G}^{\la,reg}=(-w_{0}\la)^{*}[V_{G}^{reg}/Z_{+}]$ et où $V_{G}^{reg}\subset V_{G}$ désigne le lieu où la dimension du centralisateur est minimale.
\medskip

Passons à la définition de la base de cette fibration.
Nous avons un morphisme de Steinberg 
\begin{center}
$\chi_{+}: V_{G}\rightarrow\kc$
\end{center}
ainsi qu'un morphisme de projection $p_{1}:[\kc/Z_{+}]\rightarrow[A_{G}/Z_{+}]$, qui est lisse de dimension relative $r$ si $G_{der}$ est simplement connexe d'après le théorème \ref{bouth}. On note $\kc^{\lambda}$ le changement de base à $X$. 
On définit alors la base de Hitchin $\abd$ comme le champ des sections
\begin{center}
$h_{a}:X\rightarrow\kc^{\la}$.
\end{center}
Grâce au morphisme de Steinberg $\chi_{+}$, nous avons une fibration
\begin{center}
$f:\cmd\rightarrow\abd$
\end{center}
donnée par $f(E,\phi)=\chi_{+}(\phi)$, dite de Hitchin-Frenkel-Ngô.
Si l'on suppose de plus, $G$ semisimple simplement connexe, alors $\kcd$ est un fibré vectoriel de rang $r$ et l'on a :
\begin{center}
$\abd=\bigoplus\limits_{i=1}^{r}H^{0}(X,\mathcal{O}_{X}(\left\langle \omega_{i},-w_{0}\la\right\rangle)$.
\end{center}
$\rmq$ A ce stade, il est important de remarquer que $\cmd$ peut être vide si nous n'imposons pas de conditions sur $\la$.
Nous avons une suite exacte de la forme:
$$\xymatrix{1\ar[r]&G_{der}\ar[r]&G\ar[r]^{\det_{G}}&\mathbb{G}_{m}^{l}\ar[r]&1}$$

En particulier, étant donné un point $(E,\phi)\in\cmd(k)$ en considérant $\det_{G}(\phi)$, cela nous fournit $l$-fonctions sur la courbe $X$ dont le degré est $\deg\left\langle \omega_{i}',-w_{0}\la\right\rangle$.
Cela impose donc que:
\begin{center}
$\forall~ 1\leq i\leq l,\deg\left\langle \omega_{i}',-w_{0}\la\right\rangle=0$,
\end{center}
hypothèse que nous ferons par la suite systématiquement.
Dans ce cas, la base de Hitchin pour le groupe $G$ est donné par :
\begin{center}
$\abd=\bigoplus\limits_{j=1}^{l}(H^{0}(X,\co_{X}(\left\langle \omega_{j}',-w_{0}\la\right\rangle))-\{0\})\oplus\bigoplus\limits_{i=1}^{r}H^{0}(X,\co_{X}(\left\langle \omega_{i},-w_{0}\la\right\rangle))$.
\end{center}

En particulier, une condition nécessaire pour que $\cmd$ soit non vide est que de plus:
\begin{equation}
\forall 1\leq j\leq l, H^{0}(X,\co_{X}(\left\langle \omega_{j}',-w_{0}\la\right\rangle))\neq \{0\}.
\label{glob1}
\end{equation}
et dans ce cas $\co_{X}(\left\langle \omega_{j}',-w_{0}\la\right\rangle)\cong\mathcal{O}_{X}$.

$\rmq$ Nous verrons également par la suite (cf. sect. \ref{ssimple}) que pour que $\cmd$ soit non vide pour un groupe semisimple quelconque, il faut ajouter d'autres conditions topologiques.
\subsection{Le champ de Hecke}
\begin{defi}
Pour un ensemble fini de points fermés $S$. On définit le champ de Hecke $\mathcal{H}_{S}$, dont le groupoïde des $R$-points $\mathcal{H}_{S}(R)$, pour une $k$-algèbre $R$, est constitué des triplets $(E, E', \beta)$ où $E, E'$ sont des $G$-torseurs sur $X_{R}:=X\times_{k}R$ et $\beta$ un isomorphisme:
\begin{center}
$E_{\scriptscriptstyle{\vert X_{R}-\Gamma_{S,R}}}\stackrel{\cong}{\rightarrow}E'_{\scriptscriptstyle{\vert X_{R}-\Gamma_{S,R}}}$.
\end{center}
où $\Gamma_{S, R}$ désigne le graphe des points $x$ de $S$ dans $X_{R}$.
\end{defi}

Nous avons une flèche 
\begin{center}
$\inv:\cH_{S}\rightarrow[K_{S}\backslash\Gr_{S}]$
\end{center}
qui associe à un triplet $(E,E',\beta)$ la position relative du triplet local $(E_{S},E_{S'},\beta_{\vert F_{S}})$, où $E_{S}$ est la restriction de $E$ à $D_{S}$. Pour $\la=\sum\limits_{s\in S}\la_{s}[s]\in \Div^{+}(X,T)$, on note $\inv_{S}(E,E',\beta)=\la$ si la paire $(E,E')$ est en position relative $\la$.

\begin{defi}
Pour $\la=\sum\limits_{s\in S}\la_{s}[s]\in \Div^{+}(X,T)$, on considère alors le sous-champ fermé $\overline{\cH}_{\la}$, qui est l'image réciproque du fermé $\overline{\Gr}_{\la}$ par la flèche $\inv$.
\end{defi} 
$\rmq$ D'après Varshavsky \cite[Lem. 3.1]{Va}, $\overline{\cH}_{\la}$ est un champ algébrique localement de type fini sur $k$.
Nous en déduisons donc que $\cH$ a une structure de ind-champ algébrique.
Nous avons deux projections $p_{1}$ et $p_{2}$:
$$\xymatrix{&\mathcal{H}_{S}\ar[dr]^{p_{2}}\ar[dl]_{p_{1}}\\\Bun&&\Bun}$$
où $p_{1}(E, E', \beta)=E$ et  $p_{2}(E, E', \beta)=E'$.
Les fibres de $p_{1}$ et $p_{2}$ sont des formes tordues de la grassmannienne affine $\Gr_{S}$.
De plus, le sous-champ fermé $\overline{\cH}_{\la}$ est fibré au-dessus de $\Bun$ en $\overline{\Gr}_{\la}$.
Nous avons la proposition suivante due à Varshavsky \cite[A.8c]{Va}:
\begin{prop}\label{va}
Il existe un morphisme $V\rightarrow\Bun$ lisse à fibres géométriquement connexes tel que $\overline{\cH}_{\la}\times_{\Bun}V$ est isomorphe à $V\times_{k}\overline{\Gr}_{\la}$, le produit fibré se faisant indifféremment pour $p_{1}$ ou $p_{2}$. 
En particulier, les projections sont plates, projectives et algébriques au-dessus de $\Bun$.
\end{prop}
\subsection{Le champ de Picard}
On suppose  que $G_{der}$ est simplement connexe pour disposer d'un centralisateur régulier.
Le centralisateur régulier va nous permettre d'obtenir une action d'un champ de Picard sur les fibres de Hitchin.
Nous avons construit une section de Steinberg $\eps_{+}$ et un centralisateur régulier $J$ qui est un schéma en groupes commutatifs et lisses muni d'un morphisme :
\begin{center}
$\chi_{+}^{*}J\rightarrow I$
\end{center}
qui est un isomorphisme sur l'ouvert $V_{G}^{reg}$ et $I$ le schéma des centralisateurs des éléments de $V_{G}$ dans $G$.
On tire alors le centralisateur régulier $J$ en un schéma $J^{\la}:=(-w_{0}\la)^{*}J$ sur $\kcd$.
Pour tout $S$-point de $\abd$, on a une flèche $h_{a}:X\times S\rightarrow\kcd$. Posons $J_{a}:=h_{a}^{*}J^{\la}$ l'image réciproque de $J^{\la}$ sur $\kcd$.
\begin{defi}
On considère le groupoïde de Picard $\mathcal{P}_{a}(S)$ des $J_{a}$-torseurs sur $X\times S$. Quand $a$ varie, cela définit un groupoïde de Picard $\mathcal{P}$ au-dessus de $\abd$.
\end{defi}
La flèche $\chi_{+}^{*}J\rightarrow I$ induit pour tout $S$-point $(E,\phi)$ au-dessus de $a$ une flèche
\begin{center}
$J_{a}\rightarrow \Aut_{X\times S}(E,\phi)=h_{E,\phi}^{*}I$
\end{center}
avec $I$ le schéma des centralisateurs des éléments de $V_{G}$ dans $G$.
On en déduit alors une action du groupoïde de Picard $\mathcal{P}_{a}(S)$ sur le groupoïde $\overline{\mathcal{M}}_{\la}(a)(S)$, et donc une action de $\mathcal{P}$ sur $\cmd$.
Nous allons voir que cette action est simplement transitive sur l'ouvert régulier.
Nous avons besoin pour cela d'extraire des racines, nous faisons la définition suivante:
\begin{defi}\label{div1}
Soit $n\in\mathbb{N}$ et $\la\in\Div^{+}(X,T)$. On dit que \og $n$ divise $\la$\fg, noté $n\vert\la$, s'il existe $\la'\in\Div^{+}(X,T)$ tel que pour tout $\omega\in X^{*}(T)^{+}$, nous avons :
\begin{center}
$\deg(\left\langle \omega,-w_{0}\la\right\rangle)=n\deg(\left\langle \omega,-w_{0}\la'\right\rangle)$.
\end{center}
\end{defi}
De même que pour les fibres de Springer affines, on a la proposition tirée de \cite[Prop. 3.6]{Bt1}:
\begin{prop}\label{picardtorseur}
Si $c\vert\la$, avec $c=\left|\pi_{1}(G)\right|$, alors l'ouvert $\cmdo^{reg}$  est une gerbe sous l'action de $\cP$, neutralisée par la section de Steinberg.
\end{prop}
Dans la suite, nous supposerons que l'hypothèse $c\vert\la$ est vérifiée.

\subsection{Les $z$-extensions}
Dans la suite pour analyser la fibration de Hitchin d'un groupe connexe réductif $G$, il sera commode de pouvoir se ramener au cas d'un groupe $G'$ tel que $G'_{der}$ est simplement connexe, pour lequel la géométrie de la fibration est plus agréable.
Soit $G$ connexe réductif déployé sur $\bF_{q}$. Soit $v$ une place de $x$ et $F_{v}$ le complété à la place $v$ du corps de fonctions $F$ de $X$. 
La proposition suivante est due à Kottwitz \cite[sect.4]{Kot2}:
\begin{prop}
Il existe  un groupe connexe réductif déployé $G'$ sur $X$ de $G$ qui est une extension centrale par un tore $Z$ tel que :
\begin{itemize}
\item
$G'_{der}$ est simplement connexe.
\item
 $G'(F_{v})$ se surjecte sur $G(F_{v})$.
\end{itemize}
\end{prop}
En particulier, étant donné un élément $\la\in\Div(X,T)$, on peut le relever en un élément $\tilde{\la}\in \Div(X,T)$ en relevant localement chaque $\la_{x}$. Néanmoins, pour s'assurer que $\overline{\cm}_{\tilde{\la}}$ soit non-vide, à un certain nombre de points où $\la_{x}=0$, on le relève en $\tilde{\la}\in Z(F_{x})$ non nul de telle sorte que la condition \eqref{glob1} soit vérifiée.
Nous obtenons alors un carré commutatif:
$$\xymatrix{\overline{\cm}_{\tilde{\la}}\ar[r]\ar[d]&\cmd\ar[d]\\\mathcal{A}_{\tilde{\la}}\ar[r]&\abd}.$$

\subsection{Les ouverts $\abdh$ et $\abdd$}\label{cour1}
Nous avons besoin d'introduire des ouverts de $\abd$ pour lesquels nous avons plus de prise sur la fibration de Hitchin. Ils seront étudiés de manière approfondie ultérieurement.
On rappelle que nous avons un morphisme fini plat $W$-équivariant:
\begin{center}
$\theta:V_{T}\rightarrow\kc$.
\end{center}
ramifié le long du diviseur discriminant $\mathfrak{D}_{+}=\bigcup\limits_{\alpha\in R}\overline{\Kern(\alpha)}$. On note toujours $\mathfrak{D}_{+}\subset\kc$  son image dans $\kc$ que l'on tire ensuite sur $\kcd$ en un diviseur $\mathfrak{D}_{\la}:=(-w_{0}\la)^{*}\mathfrak{D}$ (on tire de même $V_{T}$ en $V_{T}^{\la}$).
On rappelle que nous avons introduit dans \ref{Pop}, un ouvert fortement régulier semisimple $\kc^{f-rs}\subset\kc$. On note $\kD^{f-rs}$ son complémentaire.
\begin{defi}
On définit l'ouvert génériquement fortement régulier semisimple $\abdh\subset\abd$ constitué des $a\in\abd$ tels que $a(X)\not\subset\mathfrak{D}^{f-rs}_{\la}$ ainsi que 
l'ouvert transversal $\abd^{\diamondsuit}$ constitué des $a\in\abdh$ qui intersectent transversalement le diviseur discriminant $\mathfrak{D}_{\la}$.
\end{defi}
$\rmq$ En particulier, on a l'inclusion:
\begin{center}
$\abd^{\diamondsuit}\subset\abdh$. 
\end{center}
Pour $a\in\abdh$, on a une section $h_{a}:X\rightarrow\kcd$. On pose $\tilde{X}_{a}$ le revêtement fini plat obtenu en tirant par $h_{a}$ le revêtement 
\begin{center}
$V_{T}^{\la}\rightarrow\kcd$.
\end{center}
L'ouvert $\abdh$ a la propriété agréable que le champ de Picard est lisse au-dessus de celui-ci.
\begin{prop}\label{picardlisse}
Supposons $G_{der}$ simplement connexe. Pour tout $a\in\abdh$,
\begin{center}
$H^{0}(X,\Lie(J_{a}))=\Lie(Z_{G})$.
\end{center}
Le champ de Picard $\cP^{\heartsuit}:=\cP\times_{\abd}\abdh$ est lisse sur $\abdh$.
En particulier, si $G$ n'a pas de tore central, alors $\cP_{a}$ est un champ de Picard de Deligne-Mumford.
\end{prop}
\begin{proof}
De la description galoisienne \ref{galois}, on déduit que le groupe $H^{0}(X,\Lie(J_{a}))$ s'identifie aux sections $W$-équivariantes:
\begin{center}
$s:\tilde{X}_{a}\rightarrow\mathfrak{t}$.
\end{center}
Comme on a vu que $\tilde{X}_{a}$ est une courbe propre géométriquement connexe et réduite, 
\begin{center}
$H^{0}(\tilde{X}_{a},\mathfrak{t})=\mathfrak{t}$.
\end{center}
et donc en prenant les $W'$-invariants, on déduit de l'annulation $\mathfrak{t}_{der}^{W}$, que 
\begin{center}
$H^{0}(X,\Lie(J_{a}))=\Lie(Z_{G})$.
\end{center}
Le schéma $J_{a}$ est un schéma en groupes commutatif et lisse, l'obstruction à déformer un $J_{a}$-torseur gît dans $H^{2}(X,\Lie(J_{a}))$, qui est nul comme nous sommes sur une courbe. En particulier, $\cP_{a}$ est lisse.
Enfin, comme la dimension du $H^{0}$ est constante, on en déduit que la dimension de l'espace tangent reste constante, d'où la lissité de $\cP\times_{\abd}\abdh$ sur $\abdh$.
De plus, nous avons:
\begin{center}
$H^{0}(X,J_{a})=T^{W}$
\end{center}
lequel est fini non-ramifié si $G$ n'a pas de tore central et $\Gamma$ premier à la caractéristique. Ainsi, sous cette hypothèse, $\cP_{a}$ est bien de Deligne-Mumford.
\end{proof}
\begin{cor}\label{lissitereg}
L'ouvert régulier $\cmdo^{reg,\heartsuit}$ est lisse au-dessus de $\abdh$. 
\end{cor}
En général, la fibration de Hitchin n'a aucune raison d'être lisse. Néanmoins, la situation est plus agréable au-dessus de l'ouvert transversal $\abdd$ d'après \cite[Prop. 25]{Bt2}:
\begin{prop}\label{fact}
Soit $\cmdo^{reg,\diamond}$ la restriction à $\abdd$ de $\cmdo^{reg}$ et $f^{\diamond}:\cmdo^{reg,\diamond}\rightarrow\abdd$, alors nous avons un carré cartésien :
$$\xymatrix{\cmdo^{reg,\diamond}\ar[d]_{f^{\diamond}}\ar[r]&\cmd\ar[d]^{f}\\\abdd\ar[r]&\abd}$$
en particulier, pour tout $a\in\abdd$, nous avons $\cmd(a)=\cmdo^{reg}(a)$.
\end{prop}

\begin{defi}\label{plusgrand}
Pour un entier $N$ et $\la\in\Div^{+}(X,T)$, on dit que $N\prec\la$ si, pour tout $\omega\in X^{*}(T)^{+}$ non nul, nous avons :
\begin{center}
$\deg(\left\langle \omega,-w_{0}\la\right\rangle)\geq N$.
\end{center}
\end{defi} 

Maintenant, il nous faut s'assurer que l'ouvert $\abdd$ est bien non vide, ce qui fait l'objet de la proposition suivante:
\begin{prop}
Supposons $2g\prec\la$, alors l'ouvert $\abdd$ est non vide.
\end{prop}
\begin{proof}
Si $G_{der}$ est simplement connexe, la preuve est identique à celle de Ngô \cite[Prop. 4.7.1]{N}.
Dans le cas général, on considère $G'$ une $z$-extension de $G$ et on relève $\tilde{\la}$ en $\la$.
On a alors une application 
\begin{center}
$\mathcal{A}_{\tilde{\la}}\rightarrow\abd$,
\end{center}
qui induit une application $\mathcal{A}_{\tilde{\la}}^{\diamondsuit}\rightarrow\abdd$, comme on vient de voir que $\mathcal{A}_{\tilde{\la}}^{\diamondsuit}$ est non vide, il en est de même de $\abdd$.
\end{proof}

\section{Le complexe d'intersection de $\cmd$}
Nous voulons dans cette section calculer le complexe d'intersection de l'espace total de Hitchin $\cmd$. Nous avons vu qu'au-dessus de l'ouvert $\abdd$, la fibration de Hitchin était lisse, nous allons maintenant l'étudier sur un ouvert plus gros que $\abdd$ qui prendra en compte les singularités de l'espace de Hitchin. Cet ouvert sera suffisamment gros pour les applications locales que nous avons en vue.
\subsection{Le théorème de transversalité}
On note $F$ le corps de fonctions de notre courbe projective lisse géométriquement connexe $X$ de genre $g$ définie sur un corps algébriquement clos $k$.
Soit $G$ un groupe connexe réductif déployé avec $G_{der}$ simplement connexe.
On renvoie à \cite[sect. 4.2]{Bt2} pour les preuves.

On considère un diviseur  $\la\in \Div^{+}(X,T)$  et on note $S=\supp(\la)$.
On dispose du diviseur discriminant $\mathfrak{D}_{\la}$ sur $\kcd$.
Considérons le sous-schéma
\begin{center}
$\abdD:=\{a\in\abd\vert~\forall~ x\in X, d_{x}(a)\leq d\}$,
\end{center}
lequel est ouvert.
Pour tout $a\in\abd$, nous avons une décomposition en somme de diviseurs:
\begin{center}
$\Delta(a)=\Delta_{tr}(a)+\Delta_{sing}(a)$,
\end{center}
où $\Delta_{sing}(a)=\sum\limits_{x\vert~d_{x}(a)\geq 2}d_{x}(a)[x]$.
On considère  la fonction :
$$\begin{array}{ccccc}
d_{sing} & : & \abd & \to & \mathbb{N} \\
 & & a & \mapsto & \deg(\Delta_{sing}(a)) \\
\end{array}.$$
Nous avons alors le lemme suivant:
\begin{lem}\label{scsupsing}
La fonction $d_{sing}$ est semi-continue supérieurement, i.e. pour tout $d\in\mathbb{N}$, le sous-schéma de $\abd$ constitué des $a\in\abd$ tels que $d_{sing}(a)\leq d$ est ouvert.
\end{lem}

Nous allons avoir besoin d'un autre invariant pour définir le bon ouvert.
On rappelle que nous avons $\la=\sum\limits_{x\in X}\la_{x}[x]$ et $S=\supp(\la)$. On fixe un point fermé $t\in X$, tel que $\la_{t}\neq 0$, il va jouer le rôle de point auxiliaire.
On note alors $S_{0}=S-\{t\}$.
Pour chaque point fermé $x\in S_{0}$, nous avons un schéma $V_{G}^{\la_{x}}$ au-dessus de $D_{x}=\Spec(\co_{x})$ le voisinage formel autour de $x$. Ce schéma est lisse en fibre générique.
Soit alors l'entier 
\begin{center}
$e'_{x}:=e^{Elk}_{V_{G}^{\la_{x}}/\co_{x}}$,
\end{center}
où l'on renvoie à \cite[Déf.39]{Bt2} pour la définition.
Cet entier mesure la singularité du schéma $V_{G}^{\la_{x}}$, il est à noter que si $\la_{x}=0$, alors $e'_{x}=0$, puisqu'à ce moment-là, nous sommes dans le groupe.
On note alors $e_{x}=\max(\left\langle 2\rho,\la_{x}\right\rangle,e'_{x})$ et  $e=\sum\limits_{x\in S_{0}}e_{x}$.
On considère alors l'ouvert suivant:
\begin{defi}\label{bemol}
On rappelle que nous avons fixé un point fermé $t\in X$, tel que $\la_{t}\neq 0$.
Pour un entier $d\in\mathbb{N}$, on définit le sous-schéma $\abdbD\subset\abdD$ constitué des $a\in\abdD$ tels que:
\begin{itemize}
\item
$d_{t}(a)=0$.
\item
$3d_{sing}(a) +(3e+2d+1)\left|S_{0}\right|+2g-2\prec\la$.
\end{itemize}
\end{defi}
\begin{prop}
Le sous-schéma $\abdbD\subset\abdD$ est ouvert.
\end{prop}
$\rmq$ Cet ouvert $\abdbD$ peut sembler artificiel, mais pour les applications au lemme fondamental il est suffisant. En effet, nous aurons en un point $x$ de la courbe un cocaractère dominant $\la_{x}$ et un discriminant local $d_{x}$; pour globaliser le problème nous aurons juste à prendre le $\la_{t}$ aussi grand que l'on veut, de telle sorte que l'inégalité 
\begin{center}
$3d_{sing}(a) +(3e+2d_{x}+1)\left|S_{0}\right|+2g-2\prec\la$,
\end{center}
puisse être remplie.

\medskip
On forme le carré cartésien :
$$\xymatrix{\cmdbD\ar[r]\ar[d]&\cmd\ar[d]\\\abdbD\ar[r]&\abd}.$$
On pose alors 
\begin{center}
$\cmdb=\bigcup\limits_{d\in\mathbb{N}}\cmdbD$,
\end{center}
le théorème principal est alors le suivant \cite[Thm. 30]{Bt2}:
\begin{thm}\label{transverse}
Soit $G$ un groupe connexe réductif déployé tel que $G_{der}$ est simplement connexe, alors le champ $\cmdb$ est équidimensionnel,  soit $m$ sa codimension, alors on a l'égalité suivante entre les complexes d'intersections:
\begin{center}
$(\Delta^{\flat})^{*}[-m]IC_{\overline{\cH}_{\la}}=IC_{\cmdb}$.
\end{center}
\end{thm}
Nous allons avoir besoin d'un résultat plus général duquel le théorème ci-dessus se déduit.
On rappelle que nous avons $S=\supp(\la)$. D'après \cite[Prop. 1.10-1.14]{VLaf}, si nous avons un diviseur $N=\sum\limits_{i\in S}n_{i}[x_{i}]$avec des $n_{i}$ suffisamment grands par rapport à $\la$, nous disposons d'une flèche lisse:
\begin{center}
$g:\overline{\cH}_{\la}\rightarrow[\overline{\Gr}_{\la}/G_{N}]$.
\end{center}
où $G_{N}:=\Res_{N/k}G$ est la restriction à la Weil de $G$ à $N$.
Nous obtenons donc une flèche composée:
\begin{center}
$g^{\flat}:\cmdb:\rightarrow[\overline{\Gr}_{\la}/G_{N}]$
\end{center}
Nous avons fixé un point auxiliaire $t\in X$ tel que $\la_{t}\neq 0$. Comme en ce point, on impose au polynôme caractéristique d'être régulier semisimple, l'image de $f^{\flat}$ va tomber dans l'ouvert 
\begin{center}
$U:=\Gr_{\la_{t}}\times\prod\limits_{s\neq t}\overline{\Gr}_{\la_{s}}$.
\end{center}
Posons $\overline{\cH}'_{\la}:=g^{-1}(U)$, $\overline{\Gr}'_{\la}:=\prod\limits_{s\neq t}\overline{\Gr}_{\la_{s}}$ et $p:U\rightarrow \overline{\Gr}'_{\la}$.
On considère alors la flèche $g^{\flat}$ composée avec $p$:
\begin{center}
$\theta^{\flat}:\cmdb\rightarrow[\overline{\Gr}'_{\la}/G'_{N}]$
\end{center}
avec $G_{N}':=\Res_{N-n_{t}[t]/k}G$.
La proposition est la suivante \cite[Prop. 32]{Bt2}:
\begin{prop}\label{bemolisse}
La flèche $\theta^{\flat}$ est lisse.
\end{prop}
\subsection{Le cas général du théorème de transversalité}
Nous avons besoin d'étendre ce résultat à un groupe connexe réductif $G$ quelconque. 

On considère  $p:G'\rightarrow G$ une $z$-extension de $G$ par un tore déployé $Z$. Soit $\tilde{\la}$ qui relève $\la$, on forme le carré cartésien suivant:
$$\xymatrix{\widetilde{\cm}_{\tilde{\la}}\ar[d]\ar[r]&\overline{\cH}_{\tilde{\la}}\ar[d]\\\cmd\ar[r]&\chm}$$
On a un isomorphisme naturel $\overline{\Gr}_{\tilde{\la}}\rightarrow\overline{\Gr}_{\la}$, en particulier pour la topologie lisse, $\overline{\cH}_{\tilde{\la}}$ est isomorphe à $\Bung\times\overline{\Gr}_{\la}$ et comme la flèche $\Bung\rightarrow\Bun$ est lisse surjective, il en est donc de même de la flèche $\widetilde{\cm}_{\tilde{\la}}\rightarrow\cmd$ et il nous suffit donc d'étudier le complexe d'intersection de $\widetilde{\cm}_{\tilde{\la}}$.

Pour un $G'$-torseur $E$ sur $X$, on désigne par $p_{*}E$ le $G$-torseur sur $X$ qui s'obtient par poussé en avant.
Le champ $\widetilde{\cm}_{\tilde{\la}}$ classifie les paires $(E,E',\phi)$ telles que $p_{*}E=p_{*}E'$, $p(\phi)\in H^{0}(X,V_{G}^{\la}\wedge^{G} p_{*}E)$ et $\phi$ une modification bornée par $\tilde{\la}$.
En particulier, la condition $p_{*}E=p_{*}E'$ implique que $E$ et $E'$ ne diffèrent que d'un $Z$-torseur, donc on obtient que 
$\widetilde{\cm}_{\tilde{\la}}$ classifie les paires $(E,\cL,\phi)$ avec $\cL$ un $Z$-torseur et $E'= E\wedge^{G'}L$, où l'on a poussé $\cL$ en un $G'$-torseur.
Nous obtenons donc une flèche canonique $\widetilde{\cm}_{\tilde{\la}}\rightarrow\Bunz$ et la fibre au-dessus du fibré trivial est isomorphe à $\overline{\cm}_{\tilde{\la}}$.
On a alors la proposition suivante:
\begin{prop}
Localement pour la topologie lisse sur $\Bunz$, la fibration $\widetilde{\cm}_{\tilde{\la}}\rightarrow\Bunz$ est isomorphe à $\overline{\cm}_{\tilde{\la}}\times_{k}\Bunz$.
En particulier, nous obtenons que: 
\begin{center}
$(\Delta^{\flat})^{*}[-m]IC_{\chl}=IC_{\cmdb}$.
\end{center}
\end{prop}
\begin{proof}
Soit $\pi:S\rightarrow\Bunz$ un atlas lisse surjectif avec $S$ un schéma, soit $\cL$ le $Z$-torseur universel sur $\Bunz$ alors $\pi^{*}\cL$ est localement trivial pour la topologie de Zariski. 
On obtient alors de la même manière que \cite[Lem. 4.1]{Va} que $\widetilde{\cm}_{\tilde{\la}}\times_{\Bunz}S$ est isomorphe à $\cmd\times_{k}\Bunz$.
L'assertion sur les complexes d'intersections provient alors du théorème de transversalité \ref{transverse} ainsi que du diagramme commutatif:
$$\xymatrix{\overline{\cm}_{\tilde{\la}}^{\flat}\ar[r]\ar[dr]&\widetilde{\cm}^{\flat}_{\tilde{\la}}\ar[d]\ar[r]&\overline{\cH}_{\tilde{\la}}\ar[d]\\&\cmdb\ar[r]&\chl}$$
où le carré du diagramme est cartésien avec des flèches verticales lisses.
\end{proof}
On a de manière analogue une flèche :
\begin{center}
$\theta^{\flat}:\cmdb\rightarrow[\overline{\Gr}'_{\la}/G'_{N}]$
\end{center}
avec $G_{N}':=\Res_{N-n_{t}[t]/k}G$ et on obtient pour un groupe connexe réductif général que:
\begin{prop}\label{bemolisse2}
La flèche $\theta^{\flat}$ est lisse.
\end{prop}

\section{Les fibres de Springer affines}
On s'intéresse dans ce paragraphe à la variante locale des fibres de Hitchin, les fibres de Springer affines. On rappelle certains résultats de \cite{Bt1} que l'on complète.
\subsection{La formule de dimension}
Soit $F=k((\pi))$ un corps local, d'anneau d'entiers $\mathcal{O}$ et de corps résiduel $k$ algébriquement clos. Soit $G$  connexe semisimple simplement connexe. Soit $K=G(\co)$.
On se donne $\la\in X_{*}(T)^{+}$. On considère alors $a\in\mathfrak{C}_{+}^{\la}(\co)\cap\kc^{\la,rs}(F)$.
Nous faisons alors la définition suivante:
\begin{defi}\label{sophist}
On définit la fibre de Springer affine $\mathcal{M}_{\lambda}(a)$ (resp. $\mathcal{M}_{\lambda}^{reg}(a)$) comme le foncteur dont le groupoïde des $R$-points  pour une $k$-algèbre $R$ est:
$$\xymatrix{X\hat{\times}R\ar[r]^{h_{E,\phi_{+}}}\ar[dr]^{h_{a}}&[V_{G}^{\la,0}/G]\ar[d]^{\chi_{+}}\\&\kcd}$$ 
muni d'un isomorphisme entre la restriction de $h_{E,\phi_{+}}$ à $X^{\bullet}\hat{\times} R$ et la section de Steinberg $[\epsilon_{+}]^{\la'}(a)\in[V_{G}^{reg}/G]$, (resp. les morphismes $h_{E,\phi_{+}}$ qui se factorisent par $[V_{G}^{\la,reg}/G\times Z_{+}])$.
\end{defi}
\begin{lem}
Si l'on regarde les $k$-points de $\mathcal{M}_{\lambda}(a)$, en négligeant les nilpotents, il s'agit de l'ensemble suivant:
\begin{center}
$\{g\in G(F)/K~\vert ~g^{-1}\gamma_{0} g \in V_{G}^{\la,0}(\mathcal{O})\}$.
\end{center}
où $\gamma_{0}=\epsilon_{+}(a)$. Celui-ci étant non vide par définition.
\end{lem}
$\rmq$ La condition que $a\in\kcd(\co)$, nous donne que $\gamma_{0}=(\pi^{-w_{0}\la},\gamma)$, pour un certain $\gamma\in K\pi^{\lambda}K$. 

\subsection{Symétries d'une fibre de Springer affine}
On rappelle que $I:=\{(g,\g)\in G\times V_{G}\vert~ g\g g^{-1}=\g\}$.
Nous avons défini le centralisateur régulier $J$,  dont nous avons vu qu'il était muni d'un morphisme:
\begin{center}
$\chi_{+}^{*}J\rightarrow I$, 
\end{center}
qui est un isomorphisme au-dessus de $V_{G}^{reg}$. 
On se donne alors une section $h_{a}:X\rightarrow\kcd$, nous avons l'image réciproque $J_{a}=h_{a}^{*}J$.
\begin{defi}\label{centr}
Considérons  le groupoïde de Picard $P(J_{a})$ au-dessus de $\Spec(k)$ qui associe à toute $k$-algèbre $R$, le groupoïde des $J_{a}$-torseurs sur $R[[\pi]]$, munis d'une trivialisation sur $R((\pi))$.
\end{defi}
On définit une action du champ $P(J_{a})$ sur $\mathcal{M}_{\lambda}(a)$. En effet si $(E,\phi_{+})\in\mathcal{M}_{\lambda}(a)(R)$, on a un morphisme de faisceaux:
\begin{center}
$J_{a}\rightarrow\underline{\Aut}(E,\phi_{+})$ 
\end{center}
qui se déduit de la flèche $\chi_{+}^{*}J\rightarrow I$.
Celui-ci permet de tordre $(E,\phi_{+})$ par un $J_{a}$-torseur sur $X\hat{\times}R$ trivialisé sur $X^{Ullet}\hat{\times} R$.
On a alors la proposition suivante \cite[Prop. 3.6]{Bt1}: 
\begin{prop}\label{picardloc}
$[\chi_{+}]:[V_{G}^{reg}/(G\times Z_{+})]\rightarrow[\mathfrak{C}_{+}/Z_{+}]$ est une gerbe liée par le centralisateur $J$ et neutre.

En particulier, la fibre de Springer $\mathcal{M}_{\lambda}^{reg}(a)$ est un espace principal homogène sous $\mathcal{P}(J_{a})$. Dans la suite, c'est ce qu'on appellera l'orbite régulière.
\end{prop}
On considère alors la flèche finie plate surjective, génériquement étale:
\begin{center}
$\theta:V_{T}^{\la}\rightarrow\kcd$.
\end{center}
Nous avons  le diviseur discriminant $\mathfrak{D}_{+}=(2\rho,\mathfrak{D})\subset\kcd$ dont le lieu de non-annulation s'identifie au lieu régulier semisimple d'après \eqref{extdisc}.
On le tire alors sur la base $\kcd$ en un diviseur noté $\mathfrak{D}_{\la}$.
Pour $a\in\kcd(\co)\cap\kc^{rs}(F)$, on pose $d(a):=\val(a^{*}\mathfrak{D}_{\la})$.
Nous allons donner une formule pour $d(a)$. Soit $t_{+}=(\pi^{-w_{0}\la},t)\in V_{T}^{\la}(\overline{F})$ tel que $\theta(t_{+})=a$ et où $\overline{F}$ est la clôture algébrique de $F$ et $\overline{\co}$ son anneau d'entiers.
Par le critère valuatif, on a que $t_{+}\in V_{T}^{\la}(\overline{\co})$ et on déduit:
\begin{equation}
d(a)=\left\langle 2\rho,\la\right\rangle+\val(\det(\Id-\ad(t):\mathfrak{g}(F)/\mathfrak{g}_{t}(F)\rightarrow\mathfrak{g}(F)/\mathfrak{g}_{t}(F))),
\label{calculvin}
\end{equation}
avec $\kg_{t}$ l'algèbre de lie de son centralisateur. Nous avons alors le théorème suivant \cite[Thm. 3.7- Cor. 3.9]{Bt1}:
\begin{thm}\label{dim}
Soit $a\in\kcd(\co)\cap\kc^{rs}(F)$, on considère la fibre de Springer affine $\cm_{\la}(a)$ introduite dans la définition \ref{sophist}, alors nous avons:
\begin{enumerate}
\item
La fibre de Springer affine $\cm_{\la}(a)$ est un $k$-schéma localement de type fini.
\item
On a la formule de dimension:
\begin{center}
$\dim\cP(J_{a})=\dim\cm_{\la}(a)=\frac{d(a)-c(a)}{2}$.
\end{center}
où $c(a)=r-\rg_{F}(J_{a}(F))$ et $\rg_{F}$ désigne le rang du plus grand sous- tore $F$-déployé de $J_{a}(F)$.
\end{enumerate}
\end{thm}
On note $\delta(a)=\dim\cP(J_{a})$, l'invariant $\delta$ local.
\subsection{Miscellanées sur les fibres de Springer affines}
Nous avons un résultat de locale constance des intégrales orbitales :
\begin{prop}\label{Harish-Chandra}
Soit $a\in \kc^{\la}(\co)\cap\kc^{\la,rs}(F)$. Il existe un entier $N$ tel que pour toute extension finie $k'$ de $k$, pour tout $a'\in\kc^{\la}(\co\otimes_{k}k')$ tel que 
\begin{center}
$a\equiv a' \mod \pi^{N},$
\end{center}
la fibre de Springer $\overline{\mathcal{M}}_{\la}(a')$ muni de l'action de $\mathcal{P}(J_{a'})$ est isomorphe à $\overline{\mathcal{M}}_{\la}(a)\otimes_{k}k'$ munie de l'action de
$\mathcal{P}(J_{a'})\otimes_{k}k'$.
\end{prop}
\begin{proof}
La preuve est la même que \cite[Prop 3.5.1]{N}.
\end{proof}
Nous terminons ce paragraphe sur les fibres de Springer affine en montrant qu'elles admettent un quotient projectif.
Nous avons  besoin d'écrire la grassmannienne affine $G(F)/K$ comme une union de variétés projectives $X_{i}$ pour $i\in\mathbb{N}$ où
\begin{center}
$X_{i}=\{x\in G(F)/K~\vert~\ad(x)V_{G}(\co)\subset \pi^{-i}V_{G}(\co)\}$.
\end{center}
Cette variété est bien projective car elle est fermée dans le produit
$\prod\limits_{k=1}^{r}X_{i}^{\omega_{k}}$
avec
\begin{center}
$X_{i}^{\omega_{k}}=\{x\in G(F)/K~\vert~\ad(x)\End (V_{\omega_{i}})(\co)\subset \pi^{-i})\End (V_{\omega_{i}})(\co)\}$.
\end{center}
lesquelles sont projectives par Kazhdan-Lusztig \cite[p. 133]{KL}.

\begin{prop}\label{varproj}
Soit $a\in \kc^{\la}(\co)\cap\kc^{\la,rs}(F)$, il existe un entier $n\in\mathbb{N}$ tel que:
\begin{center}
$\overline{\mathcal{M}}_{a}^{\la,red}=\La_{a}(\overline{\mathcal{M}}_{a}^{\la,red}\cap X_{n})$.
\end{center}
avec $\La_{a}=T_{a}(F)/T_{a}(\co)$ où $T_{a}$ désigne le plus grand tore déployé de $J_{a}$. De plus, $\La_{a}$ agit librement et donc le quotient $\overline{\mathcal{M}}_{a}^{\la,red}/\La_{a}$ est projectif.
\end{prop}

\begin{proof}
Nous reprenons les arguments de Kazhdan-Lusztig \cite[Prop. 1]{KL} auquel nous renvoyons pour les détails. 
On commence par prouver le résultat dans le cas où $\g_{0}=\epsilon_{+}(a)$ est déployé. En utilisant alors la décomposition d'Iwasawa, on prouve que pour tout
$x\in\overline{\mathcal{M}}_{a}^{\la,red}$, il existe un élément $\la\in\La_{a}$ tel que $\la x\in Y_{a}$ où
\begin{center}
$Y_{a}:=\{g\in U(F)/U(\mathcal{O})\vert~ g^{-1}\gamma_{0} g\in V_{G}^{\la}(\co)\}$.
\end{center}

Maintenant que nous nous sommes ramenés à $U(F)/U(\co)$, pour $\delta\in\mathbb{N}$ on considère  $V_{G}(\co)^{\la}_{\delta}$ le sous-ensemble des éléments $\g\in V_{G}^{\la}(\co)$ tel que $\val\Delta(\g)=\delta$ et on montre alors qu'il existe un entier $n(\delta)$ tel que pour tout $\g\in V_{G}^{\la}(\co)_{\delta}$, on ait
\begin{center}
$Y_{\g}\subset X_{n(\delta)}$.
\end{center}
Cela s'obtient en passant par les groupes radiciels et en identifiant $U(F)$ avec $F^{r}$. On l'applique alors à $\epsilon_{+}(a)$ pour obtenir la proposition dans le cas déployé.
Pour le cas général, on passe à une extension finie pour déployer $\g_{0}$ et on redescend ensuite par un argument de cohomologie galoisienne \cite[p 135]{KL}.
\end{proof}

\subsection{Modèle de Néron et groupes des composantes connexes}
L'étude du groupe des composantes connexes de $\mathcal{P}(J_{a})$ se justifie par la proposition suivante :
\begin{prop}
L'ensemble des composantes irréductibles de la fibre de Springer affine est en bijection canonique avec le groupe des composantes connexes de $\mathcal{P}(J_{a})$.
\end{prop}
$\rmq$ Cette proposition est un corollaire de la densité de l'orbite régulière, dont on repousse la preuve au corollaire \ref{equidim}.
\medskip

Dans cette section, nous passons en revue les  résultats établis par Ngô dans \cite[sect. 3.8, 3.9]{N}. Maintenant que nous avons introduit les bons objets, les preuves s'étendent telles quelles dans ce nouveau contexte.
On suppose ici que $k$ est un corps de clôture algébrique $\bar{k}$.
On rappelle que $X=\Spec(\co)$ et on surmontera d'une barre tous les objets que l'on étend sur $\bar{k}$.
D'après Bosch-Lütkebohmert-Raynaud \cite[ch 10]{BLR}, il existe un unique schéma en groupes lisses de type fini $J_{a}^{\flat}$ sur $X$ de même fibre générique que $J_{a}$ et tel que pour tout schéma en groupes lisses de type fini $J'$ sur $\bar{X}$, de même fibre générique, on a un morphisme de groupes $J'\rightarrow J_{a}^{\flat}$ induisant l'identité sur les fibres génériques.
En remplaçant dans la définition $\ref{centr}$, $J_{a}$ par $J_{a}^{\flat}$, on obtient un ind-groupe  $\mathcal{P}(J_{a}^{\flat})$ sur $\bar{k}$.
Le morphisme canonique $J_{a}\rightarrow J_{a}^{\flat}$ induit un morphisme sur les ind-groupes
$\mathcal{P}(J_{a})\rightarrow \mathcal{P}(J_{a}^{\flat})$.

\begin{lem}
Le groupe $\mathcal{P}(J_{a}^{\flat})$ est homéomorphe à un groupe abélien libre de type fini. Le morphisme $p:\mathcal{P}(J_{a})\rightarrow \mathcal{P}(J_{a}^{\flat})$ est surjectif et son noyau est un schéma en groupes affines de type fini sur $\bar{k}$.
\end{lem}
Introduisons le revêtement caméral, il s'insère dans le diagramme cartésien suivant:
$$\xymatrix{\tilde{X}_{a}\ar[d]_{\theta_{a}}\ar[r]&V_{T}^{\la}\ar[d]^{\theta}\\\bar{X}\ar[r]^{a}&\kc^{\la}}.$$
Nous allons donner une interprétation galoisienne de $J_{a}^{\flat}$ en fonction de la normalisation de ce revêtement caméral.

\begin{prop}
Soit $\tilde{X}_{a}^{\flat}$ la normalisation de $\tilde{X}_{a}$. Alors le modèle de Néron $J_{a}^{\flat}$ admet la description galoisienne  suivante:
\begin{center}
$J_{a}^{\flat}=\prod\limits_{\tilde{X}_{a}^{\flat}/\bar{X}}(T\times\tilde{X}_{a}^{\flat})^{W}$.
\end{center}
\end{prop}
On écrit $\tilde{X}_{a}=\Spec(\tilde{\co}_{a})$ et $\tilde{X}^{\flat}_{a}=\Spec(\tilde{\co}^{\flat}_{a})$.
\begin{cor}\label{autdim}
On en déduit une autre formule pour la dimension de $\mathcal{P}(J_{a})$:
\begin{center}
$\dime\mathcal{P}(J_{a})=\dime_{\bar{k}}(\mathfrak{t}_{+}\otimes_{\bar{\co}}\tilde{\co}_{a}^{\flat}/\tilde{\co}_{a})$.
\end{center}
\end{cor}
Passons à la description du groupe des composantes connexes.
Pour $a:\bar{X}\rightarrow\kc^{\la}$ génériquement régulier semisimple, nous avons un schéma en groupes $J_{a}$ lisse abélien sur $\bar{X}$.
Nous allons déterminer la structure de $\pi_{0}(\mathcal{P}(J_{a}))$.
Pour un groupe abélien de type fini $\Lambda$, posons
\begin{center}
$\Lambda^{*}:=\Spec(\bar{\mathbb{Q}}_{l}[\Lambda])$.
\end{center}
Comme $a$ est génériquement régulier semisimple, on obtient que le revêtement caméral est génériquement un $W$-torseur. En choisissant un point géométrique de $\tilde{X}_{a}$	au-dessus du point générique de $X$, on obtient un morphisme de groupes:
\begin{center}
$\pi_{a}^{\bullet}:I:=\Gal(\bar{F}/F)\rightarrow W$.
\end{center}
Ce morphisme va nous permettre de décrire la structure du $\pi_{0}$.
Soit $J_{a}^{0}$ le sous-schéma ouvert des composantes neutres de $J_{a}$.
Nous avons une suite exacte:
$$\xymatrix{1\ar[r]&\pi_{0}(J_{a})\ar[r]&J_{a}(\bar{F})/J_{a}^{0}(\bar{\co})\ar[r]&J_{a}(\bar{F})/J_{a}(\bar{\co})\ar[r]&1}$$
d'où l'on obtient la suite analogue pour les $\pi_{0}$:
$$\xymatrix{\pi_{0}(J_{a})\ar[r]&\pi_{0}(\cP(J_{a}^{0}))\ar[r]&\pi_{0}(\cP(J_{a}))\ar[r]&1}$$

\begin{prop}
Le choix d'un point géométrique de $\tilde{X}_{a}$	au-dessus du point générique de $\bar{X}$ nous donne un isomorphisme canonique entre
\begin{center}
$\pi_{0}(\mathcal{P}(J_{a}^{0}))^{*}=\hat{T}^{\pi_{a}^{\bullet}(I)}$
\end{center}
et
\begin{center}
$\pi_{0}(\mathcal{P}(J_{a}))^{*}=\hat{T}(\pi_{a}^{\bullet}(I))$,
\end{center}
où $\hat{T}(\pi_{a}^{\bullet}(I))$  désigne le sous-groupe de $\hat{T}^{\pi_{a}^{\bullet}(I)}$ des éléments $\kappa\in \hat{T}$ tels que $\pi_{0}(\mathcal{P}(J_{a}))$ est dans le groupe de Weyl $W_{H}$ de la composante neutre $\hat{H}$ du centralisateur $\kappa$ dans $\hat{G}$.
\end{prop}
\section{Le champ de Picard}
\subsection{Morphisme degré}
A la suite de Chaudouard-Laumon \cite[5.3-6.1]{CL}, nous avons besoin de définir un morphisme degré.
Soit $a\in\abd$, de la flèche $\chi_{+}^{*}J\rightarrow I^{1}$, on déduit un morphisme de schémas en groupes
\begin{center}
$J_{a}\rightarrow S$
\end{center}
où l'on rappelle que $S:=G/G_{der}$ et que $I$ désigne le centralisateur dans $G$ des éléments du Vinberg.
Nous déduisons alors une flèche entre les classifiants sur $X$:
\begin{center}
$\cP_{a}\rightarrow \Buns$
\end{center}
et du morphisme degré $B(S)\rightarrow X_{*}(S)$, nous obtenons un morphisme entre champs de Picard:
\begin{center}
$\deg:\cP_{a}\rightarrow X_{*}(S)$
\end{center}
avec $a\in\abd$, où l'on voit $X_{*}(S)$ comme un $k$-champ de Picard constant.
On pose alors $\cP_{a}^{1}:=\Ker(\deg:\cP_{a}\rightarrow X_{*}(S))$. On obtient un sous-champ de Picard ouvert de $\cP_{a}$.
En faisant varier $a\in\abd$, on obtient un champ de Picard $\cP^{1}$, ouvert dans $\cP$, qui contient fibre à fibre $\cP^{0}$ et  qui agit sur $\cmd$.
On considère $\cmd^{1}$ le sous-champ ouvert des fibres de degré zéro. On a bien une action de $\cP^{1}$ sur $\cmd^{1}$.

\subsection{La courbe camérale}
Soit $X$ une courbe projective lisse géométriquement connexe sur un corps fini $k$. On note $\bar{X}=X\otimes_{k}\bar{k}$. Soit un groupe $G$ connexe réductif déployé, tel que $G_{der}$ est simplement connexe. On suppose $\car(k)\wedge\left|W\right|=1$.
Soit $\la\in\Div^{+}(X,T)$ tel que $c\vert\la$ avec $c=\left|\pi_{1}(G)\right|$ au sens de la définition \ref{div1}.
La base de Hitchin se décrit alors de la façon suivante:
\begin{center}
$\abd=(k^{*})^{l}\oplus\bigoplus\limits_{i=1}^{r}H^{0}(X,\mathcal{O}_{X}(\left\langle\omega_{i} ,-w_{0}\la\right\rangle))$,
\end{center}
où $l$ est la dimension du centre de $G$.
Soit le carré cartésien:
$$\xymatrix{\tilde{X}\ar[r]\ar[d]&V_{T}^{\la}\ar[d]^{\theta}\\ X\times\abd\ar[r]&\kcd}$$
où le morphisme horizontal du bas associe à $(x,a)$ l'élément $a(x)$. Le morphisme vertical de droite est fini plat muni d'une action de $W$.

Soit $U\subset X\times\abd$ l'image réciproque de $\kc^{\la, rs}\subset\kcd$ dans $X\times\abd$. Le morphisme de projection $p:X\times\abd\rightarrow\abd$ étant lisse, en prenant l'image de $U$ par $p$, nous obtenons l'ouvert $\abd^{\heartsuit}\subset\abd$, introduit précédemment, dont les $\bar{k}$-points sont:
\begin{center}
$\abd^{\heartsuit}(\bar{k}):=\{a\in\abd\vert ~a(\bar{X})\not\subset\mathfrak{D}_{\la}\}$,
\end{center}
où $\mathfrak{D}_{\la}$ désigne le diviseur discriminant.

$\rmq$ Comme $G_{der}$ est simplement connexe, il n'y a pas de différence entre régulier semisimple et fortement régulier semisimple.
\medskip

En prenant la fibre en chaque point $a\in\abd(\bar{k})$ de la flèche $\tilde{X}\rightarrow X\times\abd$ on obtient le revêtement caméral:
\begin{center}
$\theta_{a}:\tilde{X}_{a}\rightarrow\bar{X}$. 
\end{center}
Nous avons le lemme suivant \cite[Lem. 4.5.1]{N}:

\begin{lem}\label{reduite}
Pour tout $a\in\abd^{\heartsuit}(\bar{k})$ le revêtement caméral $\theta_{a}:\tilde{X}_{a}\rightarrow\bar{X}$ est génériquement un $W$-torseur et la courbe camérale $\tilde{X}_{a}$ est réduite.

\end{lem}
Pour $a\in\abdk$, on construit alors le revêtement $\tilde{X}_{\rho,a}$ qui s'insère dans le diagramme cartésien suivant:
$$\xymatrix{\tilde{X}_{a}\ar[r]\ar[d]_{\theta_{a}}& V_{T}^{\la}\ar[d]^{\theta}\\ \bar{X}\ar[r]^{a}&\kcd}$$

$\rmq$ En reprenant les notations de \eqref{jun}, on a la formule:
\begin{center}
$J_{a}^{1}=\pi_{a,*}(T)^{W}$.
\end{center}

\begin{prop}\label{conn}
Supposons $\lambda\succ 2g$ (cf.\ref{plusgrand}). Alors pour tout $a\in\abd^{\heartsuit}(\bar{k})$, les courbes camérales $\tilde{X}_{a}$ et $\tilde{X}_{\rho,a}$ sont connexes.
\end{prop}
\begin{proof}
On commence par rappeler un théorème de Debarre \cite[Th. 1.4]{Deb} :
\begin{thm}
Soit $M$ une variété irréductible et $m:M\rightarrow \mathbb{P}$ où $\mathbb{P}=\mathbb{P}^{n_{1}}\times...\times\mathbb{P}^{n_{r}}$ est un produit d'espaces projectifs. On suppose que $m$ est propre au-dessus d'un ouvert $V$ de $\mathbb{P}$. Dans chaque $\mathbb{P}^{n_{i}}$, on se donne une droite $L_{i}$ telle que pour tout sous ensemble $I$ de $\{1,...,n\}$, on a
\begin{center}
$\dime (p_{I}(m(M)))>\sum\limits_{i\in I}\codim(L_{i})$
\end{center}
où $p_{I}$ est la projection de $\mathbb{P}$ sur $\prod\limits_{i\in I}\mathbb{P}^{n_{i}}$. On suppose que $L=L_{1}\times...\times L_{r}$ est contenu dans $V$.
Alors $m^{-1}(L)$ est connexe.
\end{thm}
On rappelle que nous avons une flèche de projection lisse:
\begin{center}
$p_{1}:\abd\rightarrow (k^{*})^{l}$
\end{center}
Soit $\xi=p_{1}(a)$ et $\abd^{\xi}$ la fibre au-dessus de $\xi$.
Pour $1\leq i\leq r$, considérons le diviseur $D_{i}:=\left\langle \omega_{i}, -w_{0}\lambda\right\rangle$.
Soit le fibré en droites projectives $\mathbb{P}(D_{i})$ au-dessus de $\bar{X}$. L'hypothèse sur le degré nous dit que le fibré $\mathcal{O}(1)$ est très ample ce qui nous fournit un plongement projectif $k:\mathbb{P}(D_{i})\rightarrow\mathbb{P}^{n_{i}}$.
Ensuite, il ne nous reste qu'à reprendre l'argument de \cite[sect. 4.6.2]{N} que l'on applique à $\abd^{\xi}$ au lieu de $\abd$.
\end{proof}

\subsection{Le lieu lisse de la courbe camérale}
Dans la section \ref{cour1}, nous avions introduit l'ouvert transversal $\abd^{\diamondsuit}$ constitué des $a\in\abdk$  tels que $h_{a}:\bar{X}\rightarrow\kcd$ intersecte transversalement le diviseur discriminant $\mathfrak{D}_{\la}$.
Nous obtenons une caractérisation alternative de l'ouvert transversal à partir de la courbe camérale :
\begin{prop}
Un point $a\in\abdk$  est dans l'ouvert $\abddk$ si et seulement si $\tilde{X}_{a}$ est lisse.
\end{prop}

\begin{proof}
Supposons que $a\in\abddk$, montrons que l'image inverse de cette section sur $X\times V_{T}^{\la}$ est lisse. En dehors du discriminant, le diviseur est étale, donc il n'y a rien prouver.
Comme $a(\bar{X})$ coupe transversalement $\mathfrak{D}_{\la}$, il ne rencontre pas le lieu singulier.
Un couple $(v,x)\in V_{T}^{\la}(\bar{k})$ est tel que $v\in \bar{X}$ et $x$ est dans la fibre de $V_{T}^{\la}$ au-dessus de $v$.
En particulier, pour $(v,x)\in V_{T}^{\la}(\bar{k})$ un point d'intersection de $a(\bar{X})$ avec $\mathfrak{D}_{\la}-\mathfrak{D}_{\la}^{sing}$, $x$ appartient à un unique hyperplan de racine, on peut donc se ramener à un groupe de rang semisimple un.
On calcule alors le complété formel de la courbe $\tilde{X}_{a}$ en $(v,x)$ qui est isomorphe $\bar{k}[[\eps_{v}]][[t]]/(t^2-at+\eps_{v}^{2n_{v}})$.
La condition de transversalité impose que la valuation de $\Delta=a^{2}-4\eps_{v}^{2n}$ soit égale à un.
Cela force donc $n_{v}=0$ et $\val_{v}(a)=0$ et dans ce cas la courbe $\tilde{X}_{a}$ est lisse en $(v,x)$.

Réciproquement, si $a\notin\abddk$ et que $a(\bar{X})$ coupe le lieu lisse du diviseur discriminant avec une multiplicité au moins deux, la courbe n'est pas lisse d'après ci-dessus.
Enfin, si $a(\bar{X})$ coupe $\mathfrak{D}_{\la}^{sing}$, il est dans au moins deux hyperplans de racines, donc le groupe de monodromie locale $\pi_{a}(I_{v})$ contient deux involutions distinctes, ce qui est impossible puisque qu'il est cyclique sous l'hypothèse que l'ordre de $W$ est premier à la caractéristique.
\end{proof}

\begin{cor}
 Supposons $\la\succ 2g$, alors pour tout $a\in\abddk$, la courbe $\tilde{X}_{a}$ est irréductible.
\end{cor}
\begin{proof}
On sait déjà que $\tilde{X}_{a}$ est connexe par la proposition \ref{conn}, comme de plus $\tilde{X}_{a}$ est lisse, on obtient que $\tilde{X}_{a}$ est également lisse et comme elle est connexe, elle est irréductible.
\end{proof}

\subsection{Modèle de Néron global}
De même que pour le cas local, nous pouvons dévisser le champ de Picard $\mathcal{P}_{a}$ en utilisant le modèle de Néron global. Sa construction est due à Bosch-Raynaud-Lütkebohmert \cite[ch. 10, par. 6]{BLR}.
On reprend les résultats établis par Ngô dans \cite[sect. 4.8]{N}, dont les preuves sont identiques.
\medskip

Rappelons comment s'obtient le modèle de Néron global.
Soit $a\in\abdhk$, on considère $U$ l'ouvert de $\bar{X}$, l'image réciproque $a^{-1}(\kc^{\la,rs})$.
Pour tout $x\in\bar{X}-U$, on note $\bar{D}_{x}^{Ullet}$ le disque formel épointé en $x\in\bar{X}$, on a les modèles de Néron locaux de $J_{a\vert{\bar{D}_{x}^{Ullet}}}$ et on les recolle avec le tore $J_{a,\vert{U}}$, pour obtenir un schéma en groupes lisse de type fini sur $\bar{X}$, $J_{a}^{\flat}$, muni d'une flèche canonique $J_{a}\rightarrow J_{a}^{\flat}$.
En considérant $\tilde{X}_{a}^{\flat}$ la normalisation de la courbe camérale, l'action de $W$ sur $\tilde{X}_{a}$  induit une action de $W$ sur $\tilde{X}_{a}^{\flat}$.

\begin{prop}\label{neron}

\begin{enumerate}
\item
Le schéma en groupes $J_{a}^{\flat}$ s'identifie aux points fixes sous $W$ de la restriction à la Weil $\prod\limits_{\tilde{X}_{a}^{\flat}/\bar{X}}(T\times_{\bar{X}} \tilde{X}_{a}^{\flat})$.
\item
Le morphisme $\mathcal{P}_{a}(\bar{k})\rightarrow\mathcal{P}_{a}^{\flat}(\bar{k})$ est essentiellement surjectif.
\item
La composante neutre $(\mathcal{P}_{a}^{\flat})^{0}$ est un champ abélien, i.e. un produit de variétés abéliennes quotientées par des $\mathbb{G}_{m}$ agissant trivialement.
\item
Le noyau $\mathcal{R}_{a}$ de $\mathcal{P}_{a}\rightarrow\mathcal{P}_{a}^{\flat}$ est un produit de groupes algébriques affines de type fini $\mathcal{R}_{x}(a)$, triviaux sauf en les points $x\in\bar{X}-U$.
\end{enumerate}
\end{prop}
Nous avons une courbe camérale $\pi_{a}:\tilde{X}_{a}\rightarrow\bar{X}$ génériquement étale galoisien de groupe $W$. Soit $\tilde{X}_{a}^{\flat}$ la normalisation de $\tilde{X}_{a}$ munie d'une action de 
$W$. On a alors la description galoisienne pour $J_{a}^{\flat}$:
\begin{equation}
J_{a}^{\flat}=\prod\limits_{\tilde{X}_{a}^{\flat}/\bar{X}}(T\times\tilde{X}_{a}^{\flat})^{W}.
\label{nergal}
\end{equation}
Pour $a\in\abdhk$, on définit alors $\delta_{a}:=\dim(\mathcal{R}_{a})$.
Cet invariant s'écrit comme une somme d'invariants locaux:
\begin{equation}
\delta_{a}:=\sum\limits_{x\in \bar{X}-U}\delta_{v}(a),
\label{locglob}
\end{equation}
où $\delta_{v}(a)$ est donné par la formule \ref{dim}. De la description du groupe des symétries locales et de la proposition \ref{neron}, nous obtenons:
\begin{cor}
Pour tout $a\in\abdhk$, on a la formule:
\begin{center}
$\delta_{a}=\dim H^{0}(\bar{X},\kt\otimes_{\mathcal{O}_{\bar{X}}}(\theta_{a*}^{\flat}\mathcal{O}_{\tilde{X}_{a}^{\flat}}/\theta_{a*}\mathcal{O}_{\tilde{X}_{a}^{\flat}}))^{W}.$
\end{center}
\end{cor}
\begin{defi}
On définit le sous-ensemble $\abd^{ani}(k)$ de $\abdhk$ constitué des $a$ tels que $\pi_{0}(\cP_{a}^{1})$ est fini.
\end{defi}
$\rmq$ On montre dans la proposition \ref{ouvan} que ce sous-ensemble est ouvert.
\subsection{Automorphismes}
Soit $(E,\phi)\in\cmdhk$. Soit $\Aut(E,\phi)$ le faisceau des automorphismes, il est représentable par le schéma en groupes $I_{(E,\phi)}=h_{(E,\phi)}^{*}I$, l'image réciproque de $I$ par $h_{E,\phi}:\bar{X}\rightarrow[V_{G}^{\la}/G]$.
Sur  l'ouvert régulier semisimple $U_{a}$ de $\bar{X}$, ce schéma en groupes est un tore, mais sur $\bar{X}$, il n'est ni lisse ni plat.
D'après [BLR], il existe néanmoins  un unique  schéma en groupes lisse  $I_{(E,\phi)}^{lis}$ sur $\bar{X}$ tel que pour tout $\bar{X}$-schéma lisse $S$, on a:
\begin{center}
$\Aut(E,\phi)=\Hom_{X}(S,I_{(E,\phi)}^{lis})$.
\end{center}
La flèche canonique $I_{(E,\phi)}^{lis}\rightarrow I_{(E,\phi)}^{lis}$ est un isomorphisme au-dessus de $U_{a}$ et par définition:
\begin{center}
$\Aut(E,\phi)=H^{0}(\bar{X},I_{(E,\phi)}^{lis})$.
\end{center}
Comme $J_{a}$ est lisse, nous avons alors une flèche naturelle:
\begin{center}
$J_{a}\rightarrow I_{(E,\phi)}^{lis}$
\end{center}
qui est un isomorphisme au-dessus de $U_{a}$, ainsi qu'une flèche:
\begin{center}
$I_{(E,\phi)}^{lis}\rightarrow J_{a}^{\flat}$
\end{center}
également un isomorphisme au-dessus de $U_{a}$.
\begin{prop}\cite[Prop. 4.11.2]{N}\label{sgrpe}
Pour tout $(E,\phi)\in\cmdhk$, on a la suite d'inclusions:
\begin{center}
$H^{0}(\bar{X},J_{a})\subset \Aut(E,\phi)\subset H^{0}(\bar{X},J_{a}^{\flat})$.
\end{center}
\end{prop}
\subsection{Un calcul de dimension}\label{dimhitchin2}
On commence par le cas où $G$ est simplement connexe.
\begin{lem}
Supposons $\la\succ2g-2$, alors $\abd$ est un espace affine de dimension:
\begin{center}
$\dime\abd=\left\langle \rho, -w_{0}\lambda\right\rangle+r(1-g)$,
\end{center}
avec $\rho=\sum\limits_{i=1}^{r}\omega_{i}$.
\end{lem}

\begin{proof}
On a :
\begin{center}
$\deg \kcd=\sum\limits_{i=1}^{r}\deg (D_{i})=\left\langle \rho, -w_{0}\la\right\rangle$.
\end{center}
Par Riemann-Roch:
\begin{center}
$\dime H^{0}(X,\kcd)-\dime H^{1}(X,\kcd)=\left\langle \rho, -w_{0}\la\right\rangle+r(1-g)$
\end{center}
Comme  $\la \succ 2g-2$, nous avons $H^{1}(X,\kcd)=0$.
\end{proof}
Calculons maintenant la dimension du champ de Picard. 
\begin{prop}\label{dimpic}
Supposons $\la\succ2g-2$, alors pour tout $a\in\abdhk$, $\dime(\cP_{a})=\left\langle \rho, -w_{0}\lambda\right\rangle+r(g-1)$.
\end{prop}
\begin{proof}
D'après la proposition \ref{picardlisse}, on sait déjà que le champ de Picard est lisse au-dessus de $\abdh$, donc il suffit de calculer la dimension en un point quelconque $a\in\abdhk$.
Nous avons une flèche:
\begin{center}
$a:\bar{X}\rightarrow\kcd$.
\end{center}
Considérons le carré suivant:
$$\xymatrix{V_{T}^{\la,0}\ar[d]\ar[r]&V_{T}^{\la}\ar[d]\\\kc^{\la,0}\ar[r]&\kcd}$$
où $V_{T}^{\la, 0}$ est l'ouvert $V_{T}^{\la}\cap V_{G}^{\la,0}$ et $\kc^{\la,0}$ est l'image de $V_{T}^{\la,0}$ par $\theta$, qui est plat donc ouvert.
On choisit alors un point $a\in\abdhk$ tel que $a(\bar{X})\in\kc^{\la,0}$. Par exemple, il nous suffit de prendre un point $a\in\abddk$ tel que pour tout $x\in\supp(\la)$, $\val(\Delta_{x}(a))=0$ où $\Delta_{x}(a)$ est le discriminant local en $x$. Nous avons alors le diagramme cartésien:
$$\xymatrix{\tilde{X}_{a}\ar[r]\ar[d]_{\theta_{a}}&V_{T}^{\la,0}\ar[d]^{\theta}\\\bar{X}\ar[r]^{a}&\kc^{\la,0}}$$

De la description galoisienne de $J_{a}$ établi dans la remarque suivant le lemme \ref{reduite}, on obtient que $\Lie J_{a}=(\theta_{a,*}\kt)^{W}$ où $\theta_{a}:\tilde{X}_{a}\rightarrow \bar{X}$ est le revêtement caméral. 
Par changement de base plat, nous avons $\theta_{a,*}\kt=a^{*}\theta_{*}\kt$, il nous suffit donc de calculer $\theta_{*}\kt$.
Comme nous avons un isomorphisme entre $V_{T}^{0}$ et $Z_{+}\times\overline{T}_{\Delta}$ et que la flèche d'abélianisation:
\begin{center}
$Z_{+}\times\overline{T}_{\Delta}\rightarrow\ab^{r}$
\end{center}
induit un isomorphisme entre $\overline{T}_{\Delta}$ et $\ab^{r}$, la flèche d'abélianisation est donc un $T$-torseur.
Et de l'égalité $\theta_{*}(\Omega_{V_{T}^{\la,0}/X})^{W}=\Omega_{\kc^{\la,0}/X}$, nous obtenons:
\begin{center}
$(\theta_{*}\kt)^{W}=(\kcd)^{*}$.
\end{center}
Le théorème de Riemann-Roch nous donne alors:
\begin{center}
$\dime \cP_{a}=\dime H^{1}(\bar{X},\Lie(J_{a}))-\dime H^{0}(\bar{X},\Lie(J_{a}))=\left\langle \rho, -w_{0}\lambda\right\rangle+r(g-1)$.
\end{center}
\end{proof}
Pour un groupe connexe réductif tel que $G_{der}$ simplement connexe, il nous suffit alors d'ajouter la dimension du centre pour obtenir la dimension de $\abd$ et pour la dimension du champ de Picard, on ajoute $(g-1)\dim(Z_{G})$.
\subsection{Densité de l'ouvert régulier}
On suppose de plus que $G$ est semisimple. On établit formule du produit analogue à celle de Ngô \cite[sect 4.15]{N} entre les fibres de Hitchin et les fibres de Springer affines. Soit $a\in\abdhk$. Soit $U$ l'image réciproque de l'ouvert $\kc^{\la, rs}$ par la flèche $a:\bar{X}\rightarrow\kcd$. Le recollement avec la section de Steinberg définit un morphisme
\begin{center}
$\prod\limits_{x\in \bar{X}-U}\overline{\mathcal{M}}_{\la,x}(a)\rightarrow\overline{\mathcal{M}}_{\la}(a)$.
\end{center}
Nous avons aussi un morphisme pour le champ de Picard
\begin{center}
$\prod\limits_{x\in \bar{X}-U}\mathcal{P}_{x}(J_{a})\rightarrow\mathcal{P}_{a}$,
\end{center}
qui induit une flèche 
$$\xymatrix{\psi:\prod\limits_{x\in \bar{X}-U}\overline{\mathcal{M}}_{\la,x}(a)\wedge^{\prod\limits_{x\in \bar{X}-U}\mathcal{P}_{x}(J_{a})}\mathcal{P}_{a}\ar[r]&\overline{\mathcal{M}}_{\la}(a)}.$$

\begin{thm}\label{produithitchin}
Pour tout $a\in\abdhk$, la flèche $\psi$ induit une équivalence de catégories sur les $\bar{k}$-points.

De plus pour tout $a\in\abdank$, le quotient de $\prod\limits_{x\in \bar{X}-U}\overline{\mathcal{M}}_{\la,x}^{red}(a)\times\mathcal{P}_{a}$
par l'action diagonale de $\prod\limits_{x\in \bar{X}-U}\mathcal{P}_{x}^{red}(J_{a})$ est un champ de Deligne Mumford propre et 
\begin{center}
$\psi_{red}:\prod\limits_{x\in \bar{X}-U}\overline{\mathcal{M}}_{\la,x}^{red}(a)\wedge^{\prod\limits_{x\in \bar{X}-U}\mathcal{P}_{x}^{red}(J_{a})}\mathcal{P}_{a}\rightarrow\overline{\mathcal{M}}_{\la}(a)$
\end{center}
est un homéomorphisme.
\end{thm}
\begin{proof}
La preuve est la même que celle de \cite[Prop. 4.15.2]{N} une fois que nous avons fait appel à la proposition \ref{varproj}.
\end{proof}
\begin{cor}\label{stabaffine}
Pour tout $a\in\abdank$, $\overline{\mathcal{M}}_{\la}(a)$ est homéomorphe à un schéma projectif. De plus, pour tout $m\in\overline{\mathcal{M}}_{\la}(a)$, le stabilisateur de $m$ dans $\cP_{a}$ est un groupe affine.
\end{cor}
De la formule du produit et du résultat local \cite[Prop. 4.24]{Bt1}, on déduit le corollaire suivant:
\begin{cor}\label{comple}
Pour tout $a\in\abdhk$, le complémentaire de $\mathcal{M}_{\la}^{reg}(a)$ dans $\mathcal{M}_{\la}(a)$ est de dimension strictement plus petite que 
$\mathcal{M}_{\la}(a)$.\\
En particulier, $\dime\mathcal{M}_{\la}^{reg}(a)=\dime\mathcal{M}_{\la}(a)=\dime\overline{\mathcal{M}}_{\la}(a)$.
\end{cor} 
Pour montrer la densité de l'ouvert régulier, nous allons maintenant avoir besoin du théorème de transversalité \ref{transverse}, on rappelle que $S=\supp(\la)$.

\begin{cor}\label{equidim}
La flèche $f^{\flat}:\overline{\mathcal{M}}_{\la}^{\flat}\rightarrow\abdb$ est plate et ses fibres sont géométriquement réduites. On en déduit alors que pour tout $a\in\abdb(k)$, l'ouvert régulier $\mathcal{M}_{\la}^{reg}(a)$ est dense dans $\overline{\mathcal{M}}_{\la}(a)$. 
\end{cor}

\begin{proof}
Le morphisme $f^{\flat}$ a des fibres de dimension constante, la source est Cohen-Macaulay d'après le théorème  \ref{transverse} et le but est lisse, donc $f^{\flat}$ est plat.
Comme de plus, les fibres $\overline{\mathcal{M}}_{\la}(a)$ sont équidimensionnelles, d'après le corollaire $\ref{comple}$, on en déduit que $\mathcal{M}_{\la}^{reg}(a)$ est dense dans la fibre $\overline{\mathcal{M}}_{\la}(a)$.
\end{proof}

$\rmq$ Par la formule du produit \ref{produithitchin}, nous obtenons également la densité de l'orbite régulière locale.

\section{Stratification}
\subsection{Normalisation en famille}
De même que dans Ngô, nous allons avoir besoin d'obtenir une stratification de $\abd$ en fonction de l'invariant $\delta$, correspondant à la dimension de la partie affine du champ de Picard.
Nous renvoyons à Ngô \cite[sect. 5]{N} pour les preuves qui sont les mêmes dans le cas qui nous occupe.

\begin{defi}
Soit $f:Y\rightarrow S$ un morphisme projectif plat à fibres réduites de dimension un. Une normalisation en famille est un $S$-schéma  $Y^{\flat}$ avec une flèche $\xi:Y^{\flat}\rightarrow Y$ propre et birationnelle telle qu'au-dessus de chaque ouvert $U$ de $Y$, dense dans chaque fibre de $Y$ au-dessus de $S$, le composé $f\circ\xi$ soit propre et lisse.
\end{defi}

Considérons le foncteur $\mathcal{B}$ qui à tout $k$-schéma $S$ associe le groupoïde des triplets $(a,\tilde{X}_{a}^{\flat},\xi)$ où:
\begin{itemize}
\item
$a\in\abdh(S)$
\item
$\tilde{X}_{a}^{\flat}$ une $S$-courbe propre et lisse munie d'une action de $W$.
\item
$\xi:\tilde{X}_{a}^{\flat}\rightarrow\tilde{X}_{a}$ une normalisation en famille $W$-équivariante.

\end{itemize}

Soit $b=(a,\tilde{X}_{a}^{\flat},\xi)\in\mathcal{B}(S)$ un point dans un schéma connexe $S$.
En notant $\pi:X\times S\rightarrow S$ la projection, sous l'hypothèse que l'ordre de $W$ soit première à la caractéristique et comme $S$ est connexe, le faisceau
\begin{center}
$(\pi_{*}(\xi_{*}\mathcal{O}_{\tilde{X}_{a}^{\flat}}/\mathcal{O}_{X_{a}^{\flat}})\otimes_{\mathcal{O_{X}}}\mathfrak{t})^{W}$
\end{center}
est un $\mathcal{O}_{S}$-module localement libre de rang $\delta(b)$.
Pour $a\in\abdhk$, on a vu que la courbe camérale était connexe réduite, donc admet une unique normalisation et la flèche canonique $\cB\rightarrow\abdh$ induit une bijection au niveau des $\bar{k}$-points.
\begin{prop}
Le foncteur $\cB$ est représentable par un $k$-schéma de type fini.
\end{prop}

Fixons un point $\infty\in X(\bar{k})$.
On rappelle que nous avons un ouvert $\abd^{\infty}$ constitué des $a\in\abd$ qui sont réguliers semisimples en $\infty$.
Soit  le diagramme cartésien
$$\xymatrix{\tabd\ar[d]\ar[r]&V_{T}^{\la,\infty}\ar[d]\\\abd^{\infty}\ar[r]&\kc^{\la,\infty}}$$
où la flèche horizontale du bas associe à $a\in\abd^{\infty}$ le point $a(\infty)$ et $\kc^{\la,\infty}$ la fibre au-dessus de $\infty$.
Comme pour $a\in\abd^{\infty}$, $a(\infty)$ est régulier semisimple, cela fait de $\tabd$ un $W_{\infty}$-torseur au-dessus de $\abd^{\infty}$, où $W_{\infty}$ est la fibre du $X$-schéma groupes fini étale $W$ en $\infty$.

De plus, le morphisme de $\tabd\rightarrow V_{T}^{\la,\infty}$ se factorise par le lieu régulier semisimple de $V_{T}^{\la,\infty}$ lequel est lisse et géométriquement irréductible.
\begin{lem}
Supposons $\la\succ 2g$, alors $\tabd$ est lisse et géométriquement irréductible.
\end{lem}
\begin{proof}
L'hypothèse sur le degré assure que l'application linéaire $\abd\rightarrow\kcd$ est surjective, donc la flèche $\tabd\rightarrow V_{T}^{\la,\infty}$ est lisse à fibres connexes. Comme de plus, elle se factorise par le lieu régulier-semisimple de $V_{T}^{\la,\infty}$, on obtient donc le résultat.
\end{proof}

On fait alors le changement de base $\tilde{\cB}$, de $\cB$ à $\tabd$. En considérant la flèche $\tilde{\cB}\rightarrow\tabd$, on obtient par \cite[sect. 5.3.3]{N} une stratification de $\tabd\otimes_{k}\bar{k}$ en localement fermés irréductibles:
\begin{center}
$\tabd\otimes_{k}\bar{k}=\coprod\limits_{\psi\in\Psi}\tabdp$ 
\end{center}
où l'on peut supposer, si $\tilde{\cB}_{\psi}$ désigne l'image réciproque de $\tabdp$, que le morphisme induit est fini radiciel.
On peut également supposer que l'adhérence d'une strate est la réunion de strates, ce qui nous permet de définir un ordre partiel sur les strates et donc comme $\tabd$ est géométriquement irréductible, l'ensemble $\Psi$ admet un élément maximal $\psi_{G}$.

\subsection{Invariants monodromiques}
Soit $\tilde{a}=(a,\tilde{\infty})\in\tabdk$ avec $\tilde{\infty}$ un point géométrique de $\tilde{X}_{a}$ au-dessus de $\infty$. Soit $U_{a}$ le plus grand ouvert de $\bar{X}$ au-dessus duquel le revêtement $\tilde{X}_{a}\rightarrow\bar{X}$ est étale.
On a alors un morphisme
\begin{center}
$\pi_{1}(U_{a},\infty)\rightarrow W$
\end{center}
On note $W_{a}$ le sous-groupe de $W$ image de $\pi_{a}^{\bullet}$ et $I_{a}$ l'image par $\pi_{a}^{\bullet}$ du noyau de la flèche
\begin{center}
$\pi_{1}(U_{a},\infty)\rightarrow\pi_{1}(\bar{X},\infty)$.
\end{center}
Nous avons par construction, $W_{a}\subset W$ et $I_{a}$ est un sous-groupe normal de $W_{a}$, contenu dans $W_{a}$.
Nous avons une autre interprétation des groupes $W_{a}$ et $I_{a}$ tirée de Ngô \cite[sect. 5.4.2]{N}. Nous avons la courbe camérale $\tilde{X}_{a}$ qui est munie d'une action de $W$. Soit $C_{a}$ la composante connexe contenant $\infty$, alors $W_{a}$ est le sous-groupe de $W$ qui laisse stable la composante connexe. $I_{a}$ est alors le sous-groupe engendré par les éléments de $W_{a}$ admettant au moins un point fixe sur $C_{a}$.
Comme la projection de $C_{a}$ sur $ X$ admet une action libre de $\Theta$, $I_{a}$ est contenu dans le noyau de $W_{a}$.

\begin{prop}
L'application $\tilde{a}\rightarrow (I_{a}, W_{a})$ est constante sur chacune des strates $\tabd$ définie dans la section précédente.
On en déduit donc une application 
\begin{center}
$\psi\rightarrow (I_{\psi}, W_{\psi})$,
\end{center}
définie sur l'ensemble des strates $\Psi$ compatible avec l'application $\tilde{a}\rightarrow (I_{a}, W_{a})$ définie au niveau des points géométriques.

\end{prop}

\begin{lem}\label{croiss}
Considérons l'ordre partiel sur les couples $(I, W)$, $(I_{1},W_{1})\leq (I_{2},W_{2})$ si et seulement si $W_{1}\subset W_{2}$ et $I_{1}\subset I_{2}$, avec $I_{1}$ normal dans $W_{1}$ et $W_{1}\subset W\rtimes\Theta$. L'application 
\begin{center}
$\psi\rightarrow (I_{\psi}, W_{\psi})$
\end{center}
est croissante.
De plus, si $\la\succ2g$, $(I_{\psi_{G}},W_{\psi_{G}})=(W,W)$.
\end{lem}
\begin{proof}
L'assertion sur la croissance vient de \cite[Lem. 5.4.5]{N}. L'assertion sur $W_{\psi_{G}}$, vient du fait que pour $\la\succ2g$, l'ouvert $\abdd$ est non vide et que $W_{a}=W$, pour $a\in\abddk$.
L'assertion sur $I_{a}$ vient du fait que c'est un sous-groupe normal de $W$. La courbe camérale coupe transversalement chaque hyperplan de racines, donc $I_{a}$ contient toutes les réflexions $s_{\alpha}$ associées aux murs $h_{\alpha}$.
\end{proof}

De ce lemme, on peut en déduire une stratification de  $\tabd$ par les couples $(I_{-},W_{-})$ où $W_{-}$ un sous-groupe de $W$ et $I_{-}$ est un sous-groupe normal de $W_{-}$ inclus dans $W$.
On obtient que la réunion des strates $\tabdp$, avec $(I_{\psi}, W_{\psi})$ avec $I_{\psi}\subset I_{-}$ et $W_{\psi}\subset W_{-}$, est un fermé de $\tabd$ et que la réunion des strates $\tabdp$, avec $I_{\psi}=I_{-}$ et $W_{\psi}= W_{-}$ est ouvert dans ce fermé. On note $\tabdi$ cette strate.
Nous obtenons alors
\begin{center}
$\tabd=\coprod\limits_{(I_{-},W_{-})}\tabdi$.
\end{center}

\subsection{Calcul de $\pi_{0}(\mathcal{P}_{a})$}
Fixons un point $\infty\in X(\bar{k})$. On considère l'ouvert $\abd^{\infty}$ de $\abd\otimes\bar{k}$ défini sur $\bar{k}$, qui consiste en les points $a\in\abdk$ tels que $a(\infty)\in\kc^{\la,rs}(\bar{k})$. 
On considère alors en chaque point $a\in\abd^{\infty}$, le champ de Picard $\mathcal{P}_{a}^{\infty}$ qui classifie les $J_{a}$-torseurs avec une trivialisation en $\infty$.

\begin{prop}\cite[Prop. 4.5.7]{N}\label{infini}
Le foncteur $\mathcal{P}^{\infty}$ est représentable par un schéma en groupes lisse localement de type fini au-dessus de $\abd^{\infty}$. Pour tout $a\in\abd^{\infty}$, le tore $J_{a,\infty}$ agit sur $\mathcal{P}_{a}^{\infty}$ en modifiant la rigidification et induit un isomorphisme entre le champ quotient $[\mathcal{P}_{a}^{\infty}/J_{a,\infty}]$ et $\mathcal{P}_{a}$.
\end{prop}
\begin{proof}
La représentabilité vient de la représentabilité du schéma de Picard de la courbe camérale. La deuxième assertion est immédiate, elle consiste en l'oubli de la trivialisation.
\end{proof}
Soit $a\in\abd^{\infty}(\bar{k})$. Soit $U$ l'image réciproque dans $\bar{X}$ du lieu régulier semisimple. Au-dessus de $U$, le revêtement caméral est fini étale de groupe de Galois $W$. En se donnant $\tilde{\infty}\in\tilde{X}_{a}$ au-dessus de $\infty$, on a un morphisme de groupes:
\begin{center}
$\pi_{a}^{\bullet}:\pi_{1}(U,\infty)\rightarrow W$.
\end{center}
On note $W_{a}$ l'image de $\pi_{a}^{\bullet}$ et $I_{a}$ l'image du noyau de $\pi_{1}(U,\infty)\rightarrow\pi_{1}(\bar{X},\infty)$.
Soit $\mathcal{P}_{a}'$ le champ de Picard associé à $J_{a}^{0}$. A nouveau, pour un groupe abélien de type fini $\Lambda$, on note:
\begin{center}
$\Lambda^{*}=\Spec(\overline{\mathbb{Q}}_{l}[\Lambda])$
\end{center}
le  groupe affine diagonalisable associé.
De la suite exacte:
$$\xymatrix{0\ar[r]&J_{a}^{0}\ar[r]&J_{a}\ar[r]&\pi_{0}(J_{a})\ar[r]&0}$$
nous obtenons une suite exacte longue
$$\xymatrix{H^{0}(\bar{X},\pi_{0}(J_{a}))\ar[r]&H^{1}(\bar{X},J_{a}^{0})\ar[r]&H^{1}(\bar{X},J_{a})\ar[r]&H^{1}(\bar{X},\pi_{0}(J_{a}))=0},$$
le dernier terme étant nul comme $\pi_{0}(J_{a})$ est à support fini. On en déduit la proposition suivante:

\begin{lem}
On a une flèche surjective:
\begin{center}
$\mathcal{P}_{a}'\rightarrow\mathcal{P}_{a}$,
\end{center}
de noyau fini et une suite exacte
$$\xymatrix{H^{0}(\bar{X},\pi_{0}(J_{a}))\ar[r]&\pi_{0}(\mathcal{P}_{a}')\ar[r]&\pi_{0}(\mathcal{P}_{a})\ar[r]&0}.$$
\end{lem}
Nous pouvons maintenant élucider la structure de $\pi_{0}(\cP_{a})$ grâce à \cite[Lem. 8.3.1-8.4.1]{CL}:
\begin{prop}\label{ncl}
Pour $\tilde{a}=(a,\infty)\in\tabdk$, on a un morphisme surjectif bien défini
\begin{center}
$X_{*}(T\cap G_{der})\rightarrow\pi_{0}(\cP_{a}^{1})$
\end{center}
défini et $\pi_{0}(\cP_{a}')$ s'identifie à $X_{*}(T\cap G_{der})_{W_{a}}$.
De plus, ces flèches se mettent en famille en un morphisme surjectif de faisceaux entre le faisceau constant $X_{*}(T\cap G_{der})$ et $\pi_{0}(\cP^{1})$.
\end{prop}
\begin{defi}
On définit le lieu anisotrope $\abdan\subset\abdh$ constitué des $a\in\abdh$ tels que $\pi_{0}(\cP^{1}_{a})$ est fini.
\end{defi}
\begin{prop}\label{ouvan}
Le lieu anisotrope $\abdan\subset\abdh$ est un ouvert de $\abdh$.
\end{prop}
\begin{proof}
Du lemme $\ref{croiss}$, nous obtenons que la réunion des strates $\tabdp$ tel que $(T\cap G_{der})^{W_{\psi}}$ est fini, est un ouvert de $\tabd$.
Et de la définition du lieu anisotrope, il résulte qu'un point $\tilde{a}=(a,\infty)\in\abd^{\infty}(\bar{k})$ est dans cet ouvert si et seulement si $a\in\abdank$.
\end{proof}
\begin{cor}\label{corlisse}
Au-dessus de $\abdan$, le champ $\cmd^{1,ani}$ est de Deligne-Mumford ainsi que $\mathcal{P}^{1,ani}$.
\end{cor}
\begin{proof}
De la description galoisienne de $J_{a}^{\flat}$ \eqref{nergal}, nous obtenons que:
\begin{center}
$H^{0}(\bar{X},J_{a}^{\flat})=T^{W_{a}}$
\end{center}
Par définition de $\abdan$, nous avons $(T\cap G_{der})^{W_{a}}$ qui est fini non-ramifié comme la caractéristique est première à l'ordre du groupe de Weyl. Comme $(E,\phi)\in\cmd^{1,ani}$, il résulte de la proposition \ref{sgrpe} que $\Aut(E,\phi)$ est contenu dans $(T\cap G_{der})^{W_{a}}$, donc est également fini non-ramifié. Enfin, comme un ouvert de $\cmd^{1,ani}$ est un torseur sous $\mathcal{P}^{1,ani}$, nous obtenons la propriété analogue pour le champ de Picard.
\end{proof}
Enfin, il résulte de la proposition \ref{ncl} que nous avons immédiatement:
\begin{cor}
Au-dessus de $\tabdan$, $\pi_{0}(\cP^{1})_{\vert\tabdan}$ est un faisceau en groupes abéliens finis.
\end{cor}
\subsubsection{Module de Tate}
Plaçons-nous au-dessus de l'ouvert anisotrope. Au-dessus de cet ouvert, d'après le corollaire \ref{corlisse}, $\cP^{1}$ est un champ de Deligne-Mumford lisse.
Notons $\cP^{0}$ le sous-champ de Picard des composantes neutres de $\cP^{1}$.
Posons $g:\mathcal{P}^{0}\rightarrow\abdan$ le morphisme structural. Il est lisse de dimension relative $d$ et nous posons:
\begin{center}
$T_{\overline{\mathbb{Q}}_{l}}(\mathcal{P}^{0})=H^{2d-1}(g_{!}\overline{\mathbb{Q}}_{l})$
\end{center}
qui est un $\overline{\mathbb{Q}}_{l}$-faisceau sur $\abdan$.
On tire la proposition suivante de \cite[Prop. 4.12.1]{N}:
\begin{prop}\label{tate}
Il existe une forme alternée
\begin{center}
$\psi:T_{\overline{\mathbb{Q}}_{l}}(\mathcal{P}^{0})\times T_{\overline{\mathbb{Q}}_{l}}(\mathcal{P}^{0})\rightarrow \overline{\mathbb{Q}}_{l}(-1)$
\end{center}
telle qu'en chaque point géométrique $a\in\abdank$, la forme $\psi_{a}$ est nulle sur la partie affine $T_{\overline{\mathbb{Q}}_{l}}(R_{a})$ et induit une accouplement parfait sur la partie abélienne $T_{\overline{\mathbb{Q}}_{l}}(A_{a})$.
\end{prop}

D'après Ngô \cite[sect. 4.12]{N}, on a une suite exacte de modules de Tate:
$$\xymatrix{0\ar[r]&T_{\overline{\mathbb{Q}}_{l}}(R_{a})\ar[r]&T_{\overline{\mathbb{Q}}_{l}}(\mathcal{P}_{a}^{0})\ar[r]&T_{\overline{\mathbb{Q}}_{l}}(A_{a})\ar[r]&0}$$
où $R_{a}$ est un groupe affine et $A_{a}$ un groupe abélien. Ils sont obtenus de la manière suivante.
En introduisant un rigidificateur, nous avons vu d'après la proposition $\ref{infini}$, qu'au-dessus de l'ouvert $\abd^{\infty}$, nous avons un schéma en groupes lisse de type fini $\cP_{\infty}$ qui se surjecte sur $\mathcal{P}^{0}$, et l'on fait ensuite pour ce groupe la décomposition de Chevalley avec  $A_{a}$ sa partie abélienne et $R_{a}$ sa partie affine.
Cela revient alors à un dévissage de Chevalley pour le module de Tate $T_{\overline{\mathbb{Q}}_{l}}(\mathcal{P}^{0})$.
\subsection{Stratification à $\delta$ constant}
Pour $a\in\abdhk$, on rappelle que nous avions un entier $\delta(a)$ égal à la dimension de :
\begin{center}
$H^{0}(\bar{X},(\theta_{a*}^{\flat}\mathcal{O}_{\tilde{X}_{a}^{\flat}}/\theta_{a*}\mathcal{O}_{X_{a}}\otimes_{\mathcal{O}_{X}}\mathfrak{t})^{W})$.
\end{center}

\begin{lem}
La fonction $\tilde{a}\rightarrow\delta(\tilde{a})$ est constante sur chaque strate $\tabdp$ pour tout $\psi\in\Psi$.
\end{lem}

Nous obtenons alors une application:
\begin{center}
$\delta:\Psi\rightarrow\mathbb{N}$
\end{center}
laquelle est décroissante \cite[5.6.2]{N}.
Elle vérifie de plus la proposition suivante
\begin{prop}\label{scsup}
Soit $P\rightarrow S$ un schéma en groupes commutatifs lisses de type fini. La fonction $s\rightarrow\delta_{s}$ où $\delta_{s}$ est la dimension de sa partie affine, est une fonction semi-continue supérieurement.
Enfin, si $\la\succ 2g$, $\delta(\psi_{G})=0$.

\end{prop}

De cette proposition, nous en déduisons que la réunion des strates $\tabdp$ avec $\psi\in\Psi$ tel que $\delta(\psi)\leq\delta$ est un fermé de $\tabd$ et donc
\begin{center}
$\tabdde=\coprod\limits_{\delta(\psi)=\delta}\tabdp$
\end{center}
est un ouvert dans le fermé défini ci-dessus. D'où nous déduisons la stratification suivante:
\begin{center}
$\tabd=\coprod\limits_{\delta\in\mathbb{N}}\tabdde$.
\end{center}
\section{Codimension des strates $\abdde$}

Nous voulons prouver que $\codim(\abdde)\geq\delta$, ce qui en caractéristique $p$, pose déjà problème dans le cas des algèbres de Lie, nous montrons donc un énoncé analogue à \cite[Prop. 5.7.2]{N}.

\begin{thm}\label{bad}
Pour un groupe $G$ fixé et pour tout $\delta\in\mathbb{N}$, il existe un entier $N$ dépendant de $G$ et $\delta$, tel que si $\la\succ N$ alors la strate $\abdde$ est de codimension plus grande ou égale à $\delta$.
\end{thm}

Cet énoncé va se déduire d'un calcul local.

\subsection{Stratifications radicielles}
Soit $\co=k[[\pi]]$ avec $k$ algébriquement clos de corps de fractions $F$. Soit $F^{sep}$ une clôture séparable de $F$. On suppose que la caractéristique du corps est première à l'ordre de $W$. Dans ce paragraphe, on peut supposer, sans restreindre la généralité, que $G$ est simplement connexe. On se donne un cocaractère dominant $\la\in X_{*}(T)^{+}$.
On dira qu'un sous-ensemble $Y$ de $\kcd(\co)$ est admissible s'il existe $m\in\mathbb{N}$ tel que $Y$ soit l'image réciproque d'une sous-variété de $\kcd(\co /\pi^{m}\co)$.
Pour $a\in\kcd(\co)':=\kcd(\co)\cap\kcd(F)^{rs}$, soit $x_{+}=(\pi^{-w_{0}\la},x)\in T_{+}(F^{sep})$ d'image $a$, le choix de ce point définit un morphisme :
\begin{center}
$\pi_{a}:I:=\Gal(F^{sep}/F)\rightarrow W$,
\end{center}
comme la caractéristique du corps est première à $W$, ce morphisme se factorise par le quotient modéré $I^{mod}$ de $\Gal(F^{sep}/F)$. On considère alors un générateur topologique de $I^{mod}$ et on note $w_{a}$ son image par $\pi_{a}$.$\\$
Pour toute racine positive $\alpha\in R_{+}$, on définit une fonction $r: R_{+}\rightarrow\mathbb{Q}_{+}$ par:
\begin{center}
$r(\alpha)=\left\langle \alpha,\la\right\rangle+\val(1-\alpha(x))+\val(1-\alpha^{-1}(x))$.
\end{center}
Dans la suite pour un élément $t_{+}:=(\pi^{-w_{0}\la},t)\in V_{T}^{\la}(\co)$ et $\alpha\in R_{+}$, on pose :
\begin{center}
$\phi_{\alpha}(t)=\left\langle \alpha,\la\right\rangle+\val(1-\alpha(t))+\val(1-\alpha^{-1}(t))$.
\end{center}
Ainsi, nous avons $d(a)=\val(a^{*}\kD_{\la})=\sum\limits_{\alpha\in R^{+}}r(\alpha)$.
On note $c(a)=\dim T_{+}-\dim (T_{+}^{w_{a}})$, la chute du rang torique et d'après la formule de dimension \ref{dim}, on a
\begin{center}
$\delta(a)=\frac{d(a)-c(a)}{2}$.
\end{center}
L'orbite sous $W$ de la paire $(w_{a},r)$ ne dépend pas du choix de $x$. On considère alors le sous-ensemble $\kcd(\co)_{[w,r]}$ de $\kcd(\co)'$ constitué des éléments $a$ d'invariant $[w,r]$ fixé.
\medskip

Nous allons avoir besoin d'obtenir les strates $\kcd(\co)_{[w,r]}$ d'une autre manière dans le même esprit que Goresky-Kottwitz-McPherson \cite[sect.4.4]{GKM2}.
Soit $w\in W$ d'ordre $l$ premier à la caractéristique du corps.
Soit $E=F(\pi^{1/l})$ d'anneau d'entiers $\co_{E}$. On choisit $\xi_{l}$ une racine primitive de l'unité et on note $\tau_{E}$ l'unique élément de $\Gal(E/F)$ qui envoie $\pi^{1/l}$ sur $\xi_{l}\pi^{1/l}$. On considère alors:
\begin{center}
$V_{T,w}^{\la}(\co):=\{u\in V_{T}^{\la}(\co_{E})'\vert~ w\tau_{E}(u)=u\}$,
\end{center}
et pour une $\co$-algèbre $A$, on pose:
\begin{center}
$V_{T,w}^{\la}(A)=V_{T,w}^{\la}(\co)\otimes_{\co}A$.
\end{center}
On peut définir un schéma $V_{T,w}^{\la}$ dont les $\co$-points sont donnés par l'ensemble ci-dessus.
On commence par considérer la restriction à la Weil $\Res_{\co_{E}/\co}V_{T}^{\la}$. Pour toute $\co$-algèbre $A$, nous avons:
\begin{center}
$\Res_{\co_{E}/\co}V_{T}^{\la}(A)=V_{T}^{\la}(\co_{E}\otimes_{\co}A)$,
\end{center}
L'automorphisme $\tau_{E}$ sur $\co_{E}/\co$ induit un automorphisme:
\begin{center}
$\tau_{E}:\Res_{\co_{E}/\co}V_{T}^{\la}\rightarrow\Res_{\co_{E}/\co}V_{T}^{\la}$.
\end{center}
L'action de $W$ sur $V_{T}^{\la}$ induit une action sur $\Res_{\co_{E}/\co}V_{T}^{\la}$ qui commute avec l'action de $\tau_{E}$ d'où nous obtenons une action de $\mathbb{Z}/l\mathbb{Z}$ sur $\Res_{\co_{E}/\co}V_{T}^{\la}$ qui agit par $w\tau_{E}$.
On considère alors:
\begin{center}
$V_{T,w}^{\la}:=(\Res_{\co_{E}/\co}V_{T}^{\la})^{\mathbb{Z}/l\mathbb{Z}}$.
\end{center}
De plus, la flèche $V_{T}^{\la}\rightarrow\kcd$, induit un morphisme:
\begin{center}
$\theta_{w}:V_{T,w}^{\la}\rightarrow\kcd$.
\end{center}
Enfin, pour une fonction $r: R_{+}\rightarrow\mathbb{Q}_{+}$ fixée,  on introduit le sous-ensemble de $V_{T,w}^{\la}(\co)$ :
\begin{center}
$V_{T,w}(\co)_{r}:=\{t\in V_{T,w}^{\la}(\co)\vert~\forall~\alpha\in R_{+}, r(\alpha)=\phi_{\alpha}(t)\}$,
\end{center}
dont on vérifie immédiatement qu'il se surjecte par $\theta_{w}$ sur $\kcd(\co)_{[w,r]}$.
Nous avons alors la proposition suivante :
\begin{lem}
Pour toute $W$-orbite $[w,r]$, le sous-ensemble $\kcd(\co)_{[w,r]}$ est admissible.
\end{lem}
\begin{proof}
Posons $Y:=V^{\la}_{T,w}(\co)_{r}$, il nous suffit de voir que $Y$ est admissible comme il s'envoie surjectivement sur $\kcd(\co)_{[w,r]}$.
On choisit $N$ tel que pour tout $\alpha\in R_{+}$, $r(\alpha)<N$.
Soient $u'\in Y$ et  $u\in V_{T,w}^{\la}(\co)$, tels que:
\begin{center}
$u=u'~[\pi^{N}]$.
\end{center}
Nous voulons montrer que $u'\in Y$, ce qui revient à voir que:
\begin{center}
$\phi_{\alpha}(u')=r(\alpha)$,
\end{center}
or cette condition peut se vérifier modulo $N$ compte tenu du choix de $N$,  l'égalité $u=u'~[\pi^{N}]$ nous permet de conclure.
\end{proof}

\subsection{Une inégalité de codimension}
\begin{prop}\label{inegal}
Si $\delta(a)>0$, alors on a $\codim [w,r]\geq \delta(a)+1$.
\end{prop}

On commence par montrer l'assertion dans le cas déployé, i.e. $w=1$.
La preuve suit celle de \cite[Thm. 8.2.2]{GKM2}. On considère le sous-groupe de Lévi $M$ de $G$ dont le système de racines $R_{M}$ est donné par le sous-ensemble $R_{1}=\{\alpha\in R\vert~r(\left|\alpha\right|)\geq 1\}$ de $R$. Les objets correspondant au Lévi seront notés d'un indice \og M\fg.
\begin{lem}\label{levi}
Supposons que $\kC_{+,M}^{\la}(\co)_{r_{M}}$ soit de codimension $a$ dans $\kC_{+,M}^{\la}(\co)$, alors il en est de même pour $\kcd(\co)_{r}$ dans $\kcd(\co)$.
\end{lem}

\begin{proof}
Comme nous avons, $\kC_{+,M}:=V_{T}/W_{M}$, on obtient un morphisme surjectif:
\begin{center}
$g:\kC_{+,M}\rightarrow\kc$.
\end{center}
On définit alors le polynôme $\kD_{\la, G/M}$ sur $V_{T}$ par:
\begin{center}
$\kD_{+, G/M}=2(\rho-\rho_{M})\prod\limits_{\alpha\in R-R_{M}}(1-\alpha(t))$.
\end{center}
Soit $U_{M}$ le complémentaire du lieu d'annulation de $\kD_{+,G/M}$.
L'application $g$ est étale au-dessus de l'ouvert $U_{M}$. Le reste de la preuve est maintenant identique à \cite[Lem. 11.1.2]{GKM2}.
\end{proof}

\begin{lem}\label{mulcar}
Supposons que $\mathfrak{C}_{+,r}^{\mu}(\co)$ soit de codimension $a$ dans $\kc^{\mu}(\co)$, alors $\mathfrak{C}_{+,r+m}^{\mu+\mu'}(\co)$ est de codimension $a+rm$ dans $\kc^{\mu+\mu'}(\co)$, où $\mu'$ est tel que :
\begin{center}
$\forall ~i, \left\langle \omega_{i},\mu'\right\rangle=m$.
\end{center}
\end{lem}
\begin{proof}
Cela vient du fait que $\mathfrak{C}_{+,r+m}(\co)^{\mu+\mu'}$ est l'image de $\mathfrak{C}_{+,r}^{\mu}(\co)$ par l'application:
\begin{center}
$(\omega_{\bullet}(\pi^{\mu'}):(a_{1},\dots,a_{r})\mapsto(\pi^{\mu+\mu'}),\pi^{m}a_{1},\dots,\pi^{m}a_{r})$.
\end{center}
et du lemme \cite[10.2.1]{GKM2}.
\end{proof}
Nous pouvons maintenant prouver la proposition \ref{inegal}:
\begin{proof}
Passons à la preuve dans le cas déployé d'après Goresky-Kottwitz-McPherson.
Nous allons en fait démontrer un résultat plus précis sur la codimension, à savoir qu'elle est égale à:
\begin{center}
$D_{G}(r):=d_{G}(r)+\frac{1}{2}\sum\limits_{\alpha\in R_{+}}r(\alpha)=d_{G}(r)+\delta(a)$,
\end{center}
où $d_{G}(r)$ est la codimension de $V_{T}^{\la}(\co)_{r}$ dans $V_{T}^{\la}(\co)$, en particulier, on a $d_{G}(r)\geq 1$ comme $\delta(a)>0$.
\medskip

On raisonne sur le cardinal du nombre de racines. Le cas où $\left|R\right|=0$ est trivial.
On suppose d'abord qu'il existe $\alpha\in R_{+}$ tel que $r(\alpha)=0$, alors $R_{1}$ est un sous-ensemble strict de racines. En particulier, l'énoncé est vrai pour le Lévi $M$ du lemme \ref{levi}. En particulier, on obtient que $\kcd(\co)_{r}$ est de codimension $D_{M}(r_{M})$. Il nous faut donc voir que $D_{M}(r_{M})=D_{G}(r)$, ce qui est clair car $r$ s'annule sur les racines qui ne sont pas dans $M$ et que $V_{T}^{\la}(\co)_{r}$ est ouvert dans $V_{T}^{\la}(\co)_{r_{M}}$ d'où $d_{G}(r)=d_{M}(r_{M})$.
Si $0$ n'apparaît pas comme valeur de $r$, soit $m$ la plus petite valeur que prend $r$, alors nous avons $r=r'+m$ de telle sorte que le théorème vaut pour $r'$ et il ne nous reste plus qu'à utiliser le lemme \ref{mulcar} pour conclure.

Passons au cas général; dans ce cas il nous suffit de voir que:
\begin{center}
$\kcd(\co)_{[w,r]}\otimes_{\co}\co_{E}\subset\kcd(\co_{E})_{r}$,
\end{center}
d'où $l\codim[w,r]\geq \codim\kcd(\co_{E})_{r}\geq l(\delta(a)+1)$, ce qui conclut.
\end{proof}

\begin{prop}
La proposition \ref{inegal} implique le théorème \ref{bad}.
\end{prop}
\begin{proof}
Il suffit de reprendre \cite[Prop. 5.7.2]{N}.
\end{proof}
\section{La fibration de Hitchin anisotrope}
\subsection{La propreté}
On rappelle que l'on considère toujours un groupe connexe réductif déployé $G$.
Dans cette section, on démontre le théorème suivant:
\begin{thm}\label{proprete}
Le morphisme de Hitchin:
\begin{center}
$f^{ani}:\cmd^{1,ani}\rightarrow\abd^{ani}$
\end{center}
est propre. 
\end{thm}

On commence par montrer la partie existence du critère valuatif. Il vaut au-dessus de $\abdh$ et pour tout $\cmdh$. 
\begin{prop}
Soit $R$ un anneau de valuation discrète de corps résiduel $\kappa$ et de corps de fractions $K$. Soit $a\in\abdh(R)$ et $\phi_{K}\in\cmdh(K)$ qui relève $a_{K}$. Alors, il existe une extension finie étale $K'$ de $K$ telle que pour la normalisation $R'$ de $R$ dans $K'$, il existe un morphisme $\phi'\in\cmdh(R')$ au-dessus de $a$ dont la restriction à $\Spec(K')$ soit au-dessus de $\phi_{K}$.
\end{prop}

\begin{proof}
Le point $\phi_{K}$ correspond à un point $(E_{K},\phi_{K})$, $\eta$ le point générique de $X_{K}$. Quitte à faire une extension finie de $K$, on peut supposer que $E_{K}$ est trivial au point $\eta$.
En fixant une trivialisation de $E_{K}$, on obtient un élément $\g\in V_{G}^{\la}(L)$
où $L$ est le corps de fonctions de $X_{K}$.

Soit $B$ l'anneau local au point fermé de $X_{R}$ défini par l'idéal engendré par une uniformisante $\pi$ de $R$. C'est un anneau de valuation discrète de corps résiduel le corps de fonctions $X_{\kappa}$ et de corps de fractions $L$.
Soit $a_{\vert\Spec(B)}\in\kcd(B)$, comme $a\in\abdh(R)$, on en déduit que $a_{\vert\Spec(B)}\in\kc^{\la,f-rs}(B)$.
\begin{lem}
Soit $a_{\vert\Spec(B)}\in\kc^{\la,f-rs}(B)$, alors on peut trouver $\g_{0}\in V_{G}^{\la}(B)$ tel que $\chi_{+}(\g_{0})=a_{\vert\Spec(B)}$.
\end{lem}
\begin{proof}
Si $G$ vérifie que $G_{der}$ est simplement connexe, il suffit de prendre la section de Steinberg, dans le cas général, on a un carré cartésien:
$$\xymatrix{V_{G'}^{\tilde{\la},f-rs}\ar[d]^{\chi_{+}}\ar[r]&V_{G}^{\la,f-rs}\ar[d]^{\chi_{+}}\\\kc^{\tilde{\la},f-rs}\ar[r]&\kc^{\la,f-rs}}$$
où $G'$ est une $z$-extension de $G$ par un tore induit $Z$, $\tilde{\la}$ relève $\la$ et les flèches horizontales sont des $Z$-torseurs.
En particulier, tout $Z$-torseur sur un anneau de valuation discrète est trivial, on peut donc relever $a_{\vert\Spec B}$ en $\tilde{a}\in\kc^{\tilde{\la},f-rs}(B)$, on a alors un élément $\g_{1}\in V_{G'}^{\tilde{\la},f-rs}$ qui relève $\tilde{a}$ que l'on projette en un élément de $V_{G}^{\la,f-rs}(B)$, ce qu'on voulait.
\end{proof}
Maintenant, comme $\chi_{+}(\g)=\chi_{+}(\g_{0})$, quitte à faire une extension, on peut supposer qu'il existe $g\in G(L)$ tel que $\Ad(g)\g=\g_{0}$.
On recolle alors $(E_{K},\phi_{K})$ avec le couple $(E_{0},\g_{0})$ sur $\Spec(B)$, où $E_{0}$ est le torseur trivial.
On obtient donc un couple $(E,\phi)$ défini sur le complémentaire des points fermés de codimension deux de $X_{R}$.

Soit $x\in X\times_{k}S$ un tel point et considérons l'anneau local $\mathcal{O}_{x}$ en $x$. C'est un anneau local régulier de dimension 2.
Nous avons un couple $(E,\phi)$ sur $\Spec(\mathcal{O}_{x})-\{x\}$ que nous voulons prolonger.
D'après un théorème de Horrocks \cite{Ho}, il y a une équivalence de catégories entre les $G$-torseurs sur
$\Spec(\mathcal{O}_{x})-\{x\}$ et les $G$-torseurs sur $\Spec(\mathcal{O}_{x})$, donnée par la restriction.
En particulier, cela nous fournit un prolongement de $E$.
Quant à la section, on a une flèche $\Spec(\mathcal{O}_{x})-\{x\}\rightarrow V_{G}^{\la}$ qui se prolonge comme $V_{G}^{\la}$ est affine.
\end{proof}

Il nous faut maintenant montrer la séparation, nous allons nous inspirer de la démonstration de Chaudouard-Laumon \cite[sect. 9.2]{CL}. L'énoncé du critère valuatif est le suivant:
\begin{prop}
Soit $R$ un anneau de valuation discrète contenant $k$. Soient $S=\Spec(R)$ et $\eta=\Spec(L)$ son point générique. Soient $m,m'\in\cmd^{1,ani}(R)$ et $m_{K},m_{K'}\in\cmd^{1,ani}(K)$ les points associés tels que:
\begin{itemize}
\item
$f^{ani}(m)=f^{ani}(m')$,
\item
on a un isomorphisme $\theta_{K}:m_{K}\rightarrow m_{K'}$.
\end{itemize}
Alors, il existe une extension finie étale $L'$ de $L$ telle que pour la normalisation $S'$ de $S$ dans $L'$, il existe un unique isomorphisme $\theta:m\rightarrow m'$ qui prolonge $\theta_{K}$.
\end{prop}
\begin{proof}
Nous avons $m=(E,\phi)$ et $m'=(E',\phi')$ dans $\cmd^{ani}(R)$. Nous noterons d'un indice $K$ les objets correspondants de $\cmd^{ani}(K)$.
On a un isomorphisme $G$-équivariant
\begin{center}
$\theta_{K}:E_{K}\rightarrow E'_{K}$
\end{center}
compatible aux sections $\f$ et $\f'$ une fois que l'on a tordu par le $V_{G}^{\la}$.
Nous voulons le prolonger en un isomorphisme $\theta$ au-dessus de $\Spec(R)$.
Soit $\pi$ une uniformisante de $R$. Il définit un point $x$ de codimension un de $X_{R}:=X\times_{k}R$, on note $B$ l'anneau local en ce point.

\begin{lem} \cite[Lem. 9.3 et sect. 9.3]{CL}
Pour que $\theta_{K}$ se prolonge à $X_{R}$, il faut et il suffit qu'il se prolonge à $\Spec(B)$. De plus, de tels prolongements, s'ils existent, sont uniques.
Enfin, pour $K'$ par une extension étale de $K$ et $R'$ la normalisation de $R$ dans $K'$, il est équivalent de montrer le prolongement de $\theta_{K'}$ ou de $\theta_{K}$.
\end{lem}

$\rmq$ Si $\theta_{K}$ se prolonge en $\theta$ la compatibilité avec les sections est automatique puisque l'égalité sera vraie sur un ouvert de $X_{R}$.

Posons $a=f(m)\in\abdan(R)$.
Par Drinfeld-Simpson \cite[Thm. 2]{DS}, en agrandissant suffisamment $K$ on peut supposer que $E$ et $E'$ sont triviaux et d'après le lemme suivant on peut les trivialiser de telle sorte que les trivialisations de $E\times_{G}V_{G}^{\la}$ et $E'\times_{G}V_{G}^{\la}$  qui s'en déduisent envoient $\phi$ sur $\phi'$.

\begin{lem}
Soient $\g$ et $\g'$ des éléments de $V_{G}(B)$ tels que 
$\chi_{+}(\g)=\chi_{+}(\g')=a\in\kc^{f-rs}(B)$,
alors il existe une extension finie $K'$ de $K$ telle que si l'on désigne par $B'$ la normalisation de $B$ dans $K'$, $\g$ et $\g'$ deviennent conjugués sous un élément de $G(B')$.
\end{lem}
\begin{proof}
On renvoie à \cite[Lem. 9.5]{CL} pour les détails. L'idée est que le transporteur de $\g$ à $\g'$ est un schéma en tores sur $\Spec(B)$ qui est donc isotrivial.
\end{proof}

Maintenant on fixe  une fois pour toutes ces trivialisations et on obtient alors une trivialisation de $\theta_{K}$ qui s'identifie à la translation à gauche d'un élément $g\in G(F)$ et $\g\in V_{G}(B)$ tel que 
\begin{center}
$g^{-1}\g g=\g$.
\end{center}
De plus, comme nos deux torseurs $E_{K}$ et $E'_{K}$ sont de degré zéro, cela force $g\in G_{der}(F)$.
L'assertion se réduit alors au lemme suivant:
\begin{lem}
L'élément $g$ est dans $G_{der}(B)$.
\end{lem}
Comme $a\in\kc^{f-rs}(B)$ la réduction en fibre spéciale est encore régulière semi-simple. Ainsi le centralisateur de $\g$ est un schéma en tores $T_{a}$ sur $\Spec(B)$ et est donc isotrivial. 
En particulier, le tore $T_{a,F}$ se déploie sur une extension finie étale.
On obtient qu'il existe $\la\in X_{*}(T_{a})$ et $h\in T_{a}(F)\cap G_{der}(B)$ tel que
\begin{center}
$g=\pi^{\la}h$.
\end{center}
Comme $h$ est dans le centralisateur  de $\g$, il n'est pas gênant de modifier la trivialisation de $E$ par $h$. Ainsi, quitte à changer de trivialisation, on peut supposer $h=1$.
Enfin, on a que $\la$ est nécessairement fixe sous $\Gal(F_{s}/F)$. Or, $a$ est anisotrope donc $X_{*}(T_{a}\cap G_{der}(F))^{\Gal(F_{s}/F)}=0$et $\la=0$.
C'est ici que notre chemin diverge d'avec celui de Chaudouard-Laumon, cela nous suffit pour conclure.
\end{proof}
\subsection{La surjectivité}
A ce stade, si $G_{der}$ n'est pas simplement connexe, nous ne savons pas si la flèche $f^{ani}:\cmdan\rightarrow\abdan$ est surjective et nous aurons évidemment besoin de cette assertion lorsque nous voudrons démontrer un théorème du support.
\begin{prop}\label{surj1}
La fibration $f^{ani}:\cmdan\rightarrow\abdan$ est surjective sur les $\bar{k}$-points.
\end{prop}
Cette proposition n'est intéressante que si $G_{der}$ n'est pas simplement connexe, puisque dans ce cas nous ne disposons pas de section de Steinberg, nous nous plaçons donc dans ce cadre.
On note $\abdhp$ l'ouvert de $\abd$ constitué des polynômes génériquement réguliers semisimples et réguliers semisimples également en les points de $\supp(\la)$.
On démontre alors le lemme suivant:
\begin{prop}\label{surj2}
Le morphisme $f^{\heartsuit,+}:\cmdhp\rightarrow\abdhp$ est surjectif sur les $\bar{k}$-points.
\end{prop}
\begin{proof}
On a besoin d'un lemme:
\begin{lem}
Soit $a\in\abdhp(\bar{k})$. On note $U:=a^{-1}(\kcdr)$, en particulier, par hypothèse $\supp(\la)\subset U$. Alors il existe un ouvert non vide $V\subset U$ et une paire de Hitchin $(E_{V},\phi_{V})$ qui s'envoie sur $a_{V}\in H^{0}(V,\kcd)$, avec $E_{V}$ le $G$-torseur trivial.
\end{lem}
\begin{proof}
Si $G_{der}$ est simplement connexe, on utilise la section de Steinberg, on a donc naturellement une paire $(E_{V},\phi_{V})$ avec $E_{V}$ le torseur trivial.
On considère maintenant seulement $G$ connexe réductif. Soit $G'$ une $z$-extension, de tore induit $Z$ et $\tilde{\la}$ qui relève $\la$.
Nous avons donc le diagramme suivant:
$$\xymatrix{&\kc^{\tilde{\la},f-rs}\ar[d]\\U\ar[r]^-{a}&\kcdfr}$$
où la flèche verticale est un $Z$-torseur. En particulier, en tirant ce $Z$-torseur sur $U$ par $a$, il est localement trivial pour la topologie de Zariski.
On considère alors un ouvert $V$ qui contient $\supp(\la)$ sur lequel le $Z$-torseur est est trivial. On obtient alors un relèvement $\tilde{a}_{V}\in H^{0}(V,\kc^{\tilde{\la},f-rs})$ de $a_{\vert V}$ et on dispose d'une paire $(E'_{V},\phi'_{V})$, comme $G'_{der}$ est simplement connexe, que l'on pousse ensuite dans $V_{G}^{\la}$.
En particulier, nous obtenons une paire de Hitchin $(E_{V},\phi_{V})$ sur $V$ d'image $a_{V}$ avec $E_{V}$ le torseur trivial.
\end{proof}
Il nous faut maintenant étendre cette paire en les points de $X-V$. Soit $x\in X-V$, on note $\co_{x}$ l'anneau local complété en $x$, de corps de fractions $F_{x}$.

Nous avons une section $a_{x}\in\kcd(\co_{x})\cap\kcdr(F_{x})$, comme $\supp(\la)\subset V$, nous avons même que $a_{x}\in T/W(\co_{x})$.
En restreignant la paire $(E_{V},\phi_{V})$ à $F_{x}$, nous avons une section  $\g_{x}\in G(F_{x})^{rs}$ telle que $\chi(\g_{x})=a_{x}$.
En particulier, il existe une extension finie étale $E_{x}$ de $F_{x}$ telle qu'il existe $h_{x}\in G(E_{x})$ et $t_{x}\in T(E_{x})$ satisfaisant:
\begin{center}
$h_{x}\g_{x}h_{x}^{-1}=t_{x}$.
\end{center}
De plus, comme $T\rightarrow T/W$ est un morphisme fini et que $\chi(t_{x})=a_{x}\in\kcd(\co)$, on obtient par critère valuatif que $t\in T(\co_{E})$.
Maintenant, en appliquant \cite[sect. 8.1]{GKM2}, on sait qu'il existe un tore non-déployé $T_{w}$ défini sur $F_{x}$, déployé sur $E_{x}$ tel que $t_{x}\in T_{w}(\co_{x})$ et muni d'un plongement $T_{w}\rightarrow G$.
En particulier, on obtient un élément $\g'_{x}\in G(\co)$ tel que $\chi(\g'_{x})=a_{x}$.
En étudiant le transporteur de $\g_{x}$ à $\g'_{x}$, on obtient que c'est un $T$-torseur sur $F_{x}$, lequel est donc trivial, puisqu'il est au-dessus du complété en $x$ d'un corps de fonctions d'une courbe sur un corps algébriquement clos.
Il existe donc $g\in G(F_{x})$ tel que :
\begin{center}
$g\g_{x}g^{-1}=\g'_{x}$.
\end{center}
La donnée de $g$ nous permet alors de recoller la paire $(E_{V},\phi_{V})$ avec $(E_{0},\g'_{x})$ à la Beauville-Laszlo, où $E_{0}$ est le torseur trivial. En itérant, on obtient donc une paire $(E,\phi)\in\cmdhp(\bar{k})$ d'image $a\in\abdhp(\bar{k})$, ce qu'on voulait.
\end{proof}
On peut passer à la preuve de la proposition \ref{surj1}.
\begin{proof}
Le morphisme $f^{ani}$ est d'image fermée $F$ dans $\abdan$. De plus, d'après le lemme \ref{surj2} $F$ contient l'ouvert $\abdhp\cap\abdan$, lequel est dense dans $\abdan$. On en déduit alors que $F=\abdan$, ce qu'on voulait.
\end{proof}

\subsection{Correspondance endoscopique}\label{endocar}
Dans cette section, nous introduisons les groupes endoscopiques qui vont intervenir dans la cohomologie des fibres de Hitchin au-dessus de l'ouvert anisotrope.

Soit $G$ un groupe connexe réductif déployé tel que $G_{der}$ est simplement connexe. 
Soit $\hat{G}$ le dual de Langlands obtenu en échangeant racines et coracines dans la donnée radicielle. Il est  muni d'une paire de Borel $(\hat{T},\hat{B})$.

Soit $\kappa\in\hat{T}$, on considère $\hat{\bh}$ le centralisateur connexe de $\kappa$, qui est réductif comme $\kappa$ est semisimple. On note $\bh$ son dual. L'épinglage de $\hat{G}$ induit un épinglage pour $\hat{\bh}$ de même tore $\hat{T}$ et pour un Borel $\hat{B}_{\bh}$ de $\hat{\bh}$ contenant celui-ci. De plus, on a également une identification $\out(\hat{\bh})=\out(\bh)$. Considérons la suite exacte
$$\xymatrix{1\ar[r]&\hat{\bh}\ar[r]&\hat{G}_{\kappa}\ar[r]&\pi_{0}(\kappa)\ar[r]&1}$$
où $\pi_{\kappa}$ est le groupe des composantes connexes du centralisateur de $\kappa$ dans $\hat{G}$. On a une flèche canonique:
\begin{center}
$o_{H}(\kappa):\pi_{0}(\kappa)\rightarrow\out(\bh)$
\end{center}

\begin{defi}\label{endos1}
On appelle donnée endoscopique de $G$ sur $k$, un couple $(\kappa,\rho_{\kappa})$ avec $\kappa$ comme ci-dessus et $\rho_{\kappa}$ un morphisme
\begin{center}
$\rho_{\kappa}:\pi_{1}(\bar{X},\infty)\rightarrow\pi_{0}(\kappa)$.
\end{center}
On note $\rho_{H}$ le composé de $\rho_{\kappa}$ et $o_{\bh}(\kappa)$. Ce morphisme nous permet de tordre $\bh$ pour obtenir une forme quasi déployée $H$ de $\bh$, qui est le groupe endoscopique associé au couple $(\kappa,\rho_{\kappa})$.
De plus, si $\kappa\in\hat{T}$ est d'ordre fini, on dira que la donnée endoscopique est elliptique. Dans ce cas, si $G$ est semisimple, alors $H$ l'est également.
\end{defi}
\textit{Dans la suite, on suppose que la donnée endoscopique est déployée}. Une telle hypothèse est vérifiée si $\hat{G}$ a également son groupe dérivé simplement connexe. En particulier, la donnée de $\rho_{\kappa}$ est triviale. On réduit donc la donnée endoscopique au seul $\kappa$. Fixons alors $\kappa\in\hat{T}$, on a un morphisme canonique
\begin{center}
$W_{H}\rightarrow W$
\end{center}
compatible avec l'action sur $T$.
On en flèche $\nu:\kC_{H}:=T/W_{H}\rightarrow\kC:=T/W$, qui, au-dessus de $\kC^{rs}$, est un morphisme fini étale:
\begin{center}
$\nu^{rs}:\kCH^{G-rs}\rightarrow\kC^{rs}$
\end{center}
où $\kCH^{G-rs}:=\nu^{-1}(\kC^{rs})$.
Nous voulons essayer de passer au semi-groupe de Vinberg.

Naïvement, on voudrait définir une flèche $\nu:\kch\rightarrow\kc$; néanmoins, bien que $H$ et $G$ partagent le même tore maximal, lorsque l'on veut passer au Vinberg, il faut ajouter une partie abélienne, laquelle est une variété torique qui contient le tore adjoint comme groupe des inversibles qui eux sont distincts pour $H$ et $G$.
Il est donc vain d'essayer de définir la correspondance endoscopique au niveau de $\kc$.
En revanche, si nous fixons la partie abélienne, ce problème disparaîtra.

Nous avons $\kc=A_{G}\times\kC$ et de même pour $H$. Etant donné que nous avons un isomorphisme entre les tores maximaux $T$ et $T^{H}$ de $G$ et $H$ et que les racines de $H$ s'identifient à un sous-ensemble de $G$, tout cocaractère dominant de $G$, reste dominant pour $H$. On en déduit alors une flèche :
\begin{center}
$-w_{0}\la:X\rightarrow[A_{H}/T^{H}]$.
\end{center}
Les espaces caractéristiques $\mathfrak{C}_{+,H}^{\la}$ et $\kcd$ s'obtiennent par torsion par le $T$-torseur $E_{T}(-w_{0}\la)$ des espaces fibrés au-dessus de $X$, $\kC_{H}$ et $\kC$, d'où l'on obtient  une flèche 
\begin{center}
$\nu_{+}:\mathfrak{C}_{+,H}^{\la}\rightarrow\kcd$.
\end{center}
On en déduit alors un morphisme
\begin{center}
$\nu:\abde\rightarrow\abd$.
\end{center}
\begin{lem}\label{transI}
La flèche $\nu$ restreinte à $\abdh$ est finie non ramifiée.
\end{lem}
\begin{proof}
Il suffit de reprendre \cite[7.2]{N2}.
\end{proof}
Pour étudier la cohomologie de la fibration de Hitchin, il n'y a pas que la strate $\abdh$ qui va contribuer, mais également des strates plus \og petites\fg.
En effet, notons $V_{\la}$ la représentation irréductible de plus haut poids $\la$ de $\mathstrut^{L}G$. Pour une donnée endoscopique  $\kappa$, du morphisme $\eta:\hat{H}\rightarrow\hat{G}$, on en déduit par restriction une représentation:
\begin{equation}
\eta^{*}V_{\la}:=\bigoplus\limits_{\mu\leq\la}(V^{H}_{\mu})^{\oplus m_{\mu}},
\label{restri}
\end{equation}
avec $V^{H}_{\mu}$ la représentation irréductible de plus haut poids $\mu\in X_{*}(T)^{+}$ et $m_{\la}=1$.

$\rmq$ Il est à noter que l'on a $m_{\mu}\neq 0$ dès que $\la-\mu$ ne peut s'écrire comme une somme à coefficients positifs de racines simples de $H$. On note alors $\soc_{H}(\la)$ l'ensemble des $\mu$ tels que $\mu\neq 0$.
\medskip
\begin{defi}\label{changement}
Etant donnée une donnée endoscopique $\kappa$, avec les notations de ci-dessus, on considère le schéma:
\begin{center}
$\abdeth:=\bigcup\limits_{\mu\in\soc_{H}(\la)}\mathcal{A}_{\mu,H}$.
\end{center}
\end{defi}
Si $\mu\leq\la$, on a une immersion fermée canonique :
\begin{center}
$\mathcal{A}_{\mu}\rightarrow\mathcal{A}_{\la}$,
\end{center}
d'où l'on déduit à nouveau une flèche, notée de la même manière:
\begin{equation}
\nu:\abdeth\rightarrow\abd
\label{transII}
\end{equation}
qui est fini non-ramifiée au-dessus de $\abdh$ d'après le lemme \ref{transI}.
\subsection{Cohomologie au-dessus de l'ouvert anisotrope}
On suppose $G_{der}$ simplement connexe et que l'endoscopie est déployée. Soit $\tcmdan$ (resp. $\cPinf$) le changement de base de $\cmd^{1}$ (resp. $\cP^{1}$)  à $\tabdan$. On pose $IC_{\cmd}$ le complexe d'intersection sur $\tcmdan$.

Nous avons une flèche $\tilde{f}^{ani}:\tcmdan\rightarrow\tabdan$ qui est propre et $\tilde{g}^{ani}:\cPinf\rightarrow\tabdan$ est propre et lisse.
Les deux sources sont des champs de Deligne-Mumford.
D'après le théorème de Deligne \cite{De} $\tilde{f}_{*}^{ani}IC_{\cmd}$ est pur et par le théorème de Beilinson-Bernstein-Deligne-Gabber \cite{BBD}, il se décompose au-dessus de $\tabdan$ en 
\begin{center}
$\tilde{f}_{*}^{ani}IC_{\cmd}=\bigoplus\limits_{n\in\mathbb{Z}}\mathstrut^{p}H^{n}(\tilde{f}^{ani}_{*}IC_{\cmd})[-n]$
\end{center}
où $\mathstrut^{p}H^{n}(\tilde{f}^{ani}_{*}IC_{\cmd})[-n]$ est pervers pur de poids $n$.
Comme  $\cPinf$ agit sur $\tcmdan$ de manière compatible à la stratification par les $\mu\leq\la$, $\cPinf$ agit sur $\tilde{f}_{*}^{ani}IC_{\cmd}$.
Par le lemme d'homotopie \cite[3.2.3]{LN} établi par Laumon-Ngô, l'action de $\cPinf$ sur les faisceaux de cohomologie pervers $\mathstrut^{p}H^{n}(\tilde{f}^{ani}_{*}IC_{\cmd})$ se factorise par $\pi_{0}(\cPinf)$.
La flèche surjective établie dans la proposition \ref{ncl}:
\begin{center}
$X_{*}(T\cap G_{der})\rightarrow\pi_{0}(\cPinf)$,
\end{center}
 fait de $\pi_{0}(\cPinf)$ un quotient fini de $X_{*}(T\cap G_{der})$, nous permet de définir pour tout $\kappa\in T$, un facteur direct $\mathstrut^{p}H^{n}(\tilde{f}^{ani}_{*}IC_{\cmd})_{\kappa}$ sur lequel $X_{*}(T)$ agit à travers $\kappa:X_{*}(T\cap G_{der})\rightarrow\overline{\mathbb{Q}}_{l}^{*}$.

\subsection{Transfert endoscopique}
Soit $\kappa$ une donnée endoscopique de $G$. 
D'après \eqref{transII}, on a une flèche 
\begin{center}
$\nu:\abdeth\rightarrow\abd$.
\end{center}
Soit $\tilde{a}:=(a,\tilde{\infty})\in\tabdk$. On a une projection
\begin{center}
$\tilde{X}_{a}\rightarrow\bar{X}$
\end{center}
génériquement étale de  groupe $W\rtimes\pi_{0}(\kappa)$ vers $\bar{X}$. 
On définit une flèche
\begin{center}
$\tilde{\nu}:\tabdeth\rightarrow\tabd$ 
\end{center}
donnée par $(a_{H},\tilde{\infty})\mapsto(\nu(a_{H}),\tilde{\infty})$. Cette flèche est définie sur $k$ si $\tilde{\infty}$ est aussi défini sur $k$.
Nous avons la proposition due à Ngô \cite[Prop. 6.3.2-6.3.3]{N} et \cite[10.3, Lem. 10.1]{N2} 
\begin{prop}\label{reunion}
Le morphisme $\tilde{\nu}_{H}:\tabdeth\rightarrow\tabd$ est une immersion fermée.
\end{prop}
En particulier, cela nous permet de considérer $\tabdeth$ comme un sous-schéma fermé de $\tabd$. Nous notons  $\tabdanka$ pour désigner son image dans $\tabd$ (Si l'on considérait des groupes quasi-déployés, $\tabdanka$ serait la réunion des $\tabdeth$ associés aux morphismes $\rho_{\kappa}:\pi_{1}(\bar{X},\infty)\rightarrow\pi_{0}(\kappa)$).
Nous pouvons donc maintenant localiser le lieu dans lequel on rencontre les supports de $\mathstrut^{p}H^{n}(\tilde{f}^{ani}_{*}IC_{\cmd})_{\kappa}$ :
\begin{prop}\label{kappa}
Le support du faisceau pervers $\mathstrut^{p}H^{n}(\tilde{f}^{ani}_{*}IC_{\cmd})_{\kappa}$ est contenu dans $\tabdanka$.
\end{prop}
\begin{proof}
Il vient de la description de la proposition \ref{ncl} de $\pi_{0}(\tcPan)$, que la restriction de $\mathstrut^{p}H^{n}(\tilde{f}^{ani}_{*}IC_{\cmd})_{\kappa}$ à l'ouvert $\tabdan-\tabdanka$ est nulle.
\end{proof}
On pose alors 
\begin{center}
$(\tilde{f}^{ani}_{*}IC_{\cmd})_{\kappa}=\bigoplus\limits_{n\in\mathbb{Z}}\mathstrut^{p}H^{n}(\tilde{f}^{ani}_{*}IC_{\cmd})_{\kappa}[-n]$.
\end{center}
et nous notons $\supp((\tilde{f}^{ani}_{*}IC_{\cmd})_{\kappa})$ l'ensemble des supports des faisceaux pervers irréductibles qui interviennent dans la décomposition des faisceaux $\mathstrut^{p}H^{n}(\tilde{f}^{ani}_{*}IC_{\cmd})_{\kappa}$.

\subsection{Détermination des supports}
Dans le théorème \ref{transverse}, nous avions introduit l'ouvert $\abdb$ sur lequel on contrôle le complexe d'intersection.
On pose alors $\abdbf=\abdb\cap\abdan$.
On rappelle que nous avons une fonction $\delta$ définie sur $\abd$. Pour $a\in\abdk$, nous avons :
\begin{center}
$\delta(a)=\dim R_{a}$
\end{center}
où $R_{a}$ est le groupe défini dans la proposition \ref{neron}.
D'après \cite[Lem. 6.5.3]{N}, cette fonction est semi-continue supérieurement, nous faisons alors la définition suivante analogue à celle de Chaudouard-Laumon \cite{CL}:
\begin{defi}
Soit $\abdbon\subset\abd$ le plus grand ouvert de $\abd$ tel que pour tout point de Zariski $a\in\abdbon$, on ait l'égalité suivante :
\begin{center}
$\codim_{\abd}(a)\geq\delta(a)$.
\end{center}
\end{defi}
$\rmq$ La fonction $\delta$ étant semi-continue supérieurement, l'égalité $\delta(a)=\delta$ définit une partie localement fermée $\abdde$ de $\abd$. L'ouvert $\abdbon$ est le complémentaire des composantes irréductibles $\abdde'$ de $\abdde$ qui vérifient
\begin{center}
$\codim_{\abd}(\abdde')<\delta$.
\end{center}
Par noetherianité, il n'y a qu'un nombre fini de $\delta$ qui contribuent.
Notons $\tilde{f}^{ani,\flat}$ la fibration de Hitchin restreinte à $\tabdanf$.
\begin{thm}\label{detsupport}
Soit $\kappa:X_{*}(T\cap G_{der})\rightarrow\overline{\mathbb{Q}}_{l}^{*}$ et $Z\in\supp((\tilde{f}^{ani,\flat}_{*}IC_{\cmd})_{\kappa})$, alors $Z$ est inclus dans $\tabdeth$. Si de plus, $Z\cap\tabdethb\neq\emptyset$, alors $Z=\tabdethanf$.
\end{thm}
Nous reportons la preuve de ce théorème à la section \ref{final}. Nous pouvons en revanche d'ores et déjà étudier  la cohomologie ordinaire en degré maximal.
\begin{prop}\label{stabmax}
Soit $d$ la dimension relative de $\tcmdan\rightarrow\tabdan$.
On a un isomorphisme entre la partie stable (i.e. $\kappa=1$) $(R^{2d}\tilde{f}^{ani}_{*}IC_{\cmd})_{st}$ et le faisceau constant $\overline{\mathbb{Q}}_{l}$. On a également un isomorphisme entre $(R^{2d}\tilde{f}^{ani}_{*}IC_{\cmd})_{\kappa}$ et $\tilde{\nu}_{H,*}\overline{\mathbb{Q}}_{l}$.
\end{prop}
\begin{proof}
On reprend la preuve de \cite[Prop 6.5.1]{N}.
La section de Steinberg définit une immersion ouverte $\cPinf\rightarrow\tcmdan$ dont le complémentaire est de dimension relative inférieure ou égale à $d-1$. Nous obtenons donc un isomorphisme
\begin{center}
$R^{2d}g_{!}\overline{\mathbb{Q}}_{l}\rightarrow R^{2d}\tilde{f}^{ani}_{*}IC_{\cmd}$
\end{center}
compatible à l'action de $\pi_{0}(\tcPan)$, où $g$ le morphisme structural de $\tcPan$.
En utilisant alors le morphisme trace, on peut identifier $R^{2d}g_{!}\overline{\mathbb{Q}}_{l}$ avec le faisceau associé au préfaisceau
\begin{center}
$U\mapsto\overline{\mathbb{Q}}_{l}^{\pi_{0}(\cPinf)(U)}$
\end{center}
La fin de la proposition résulte de \cite[Prop 6.5.1]{N} auquel on renvoie pour les détails.
\end{proof}
\section{La fibration si $G$ non simplement connexe}\label{ssimple}
\subsection{Un calcul de dimension}
On se place dans le cas $G$ semisimple. Traiter le cas réductif n'apporte que des modifications mineures (cf. \cite[Rmq. 6.20]{Hu}). 
A la suite d'Hurtubise-Markman \cite[sect. 3.2]{Hu}, nous introduisons un invariant discret topologique lié à une paire $(E,\phi)\in\cmd$. On note $G_{sc}$ le revêtement simplement connexe de $G$ et $T_{sc}$ l'image réciproque de $T$ dans $G_{sc}$.
On note $U:=X-\supp(\la)$. On fixe un point fermé $x_{0}\in X$.

On considère alors le schéma en groupes $\Aut(E)_{U}$. La section unité de $\Aut(E)_{U}$ induit une application injective de $\pi_{1}(U,x_{0})\rightarrow\pi_{1}(\Aut(E)_{U})$.
Comme le groupe $\pi_{1}(\Aut(E)_{U})$ s'identifie au produit $\pi_{1}(U,x_{0})\times\pi_{1}(G)$. La section $\phi$ induit une flèche:
\begin{center}
$\tau_{\phi}:\pi_{1}(U,x_{0})\rightarrow\pi_{1}(G)$.
\end{center}
Comme $\pi_{1}(G)$ est abélien, $\tau_{\phi}$ induit une classe dans $H^{1}(U,\pi_{1}(G))$. Si l'on choisit une autre paire $(E_{1},\phi_{1})\in\cmd$ la différence $\tau_{\phi}-\tau_{\phi_{1}}$ est la restriction d'une classe de $H^{1}(X,\pi_{1}(G))$, lequel est fini.
On note $S_{G}$ l'ensemble fini des valeurs que peut prendre l'invariant $\tau$.
En particulier, on a une décomposition en composantes connexes de $\cmd$:
\begin{center}
$\cmd=\coprod\limits_{\tau\in S_{G}}\cmd(\tau)$
\end{center}
Nous allons maintenant voir comment relier les composantes $\cmd(\tau)$ à un espace de Hitchin pour $G_{sc}$.
On fixe un invariant $\tau_{\phi}:\pi_{1}(U,x_{0})\rightarrow\pi_{1}(G)$.
En vertu de \cite[Cond. 6.13]{Hu}, cela ne restreint pas la généralité de supposer $\tau$ surjectif. On note $\Gamma:=\pi_{1}(G)$.

On considère alors $f:X_{\tau}\rightarrow X$ le revêtement ramifié de $X$ qui compactifie le revêtement galoisien de groupe $\Gamma$, 
\begin{center}
$f^{0}:X_{\tau}^{0}\rightarrow U$
\end{center}
correspondant à $\tau$. Soit le diviseur 
\begin{center}
$\tau^{-1}(\la)=\sum\limits_{x\in X}\sum\limits_{\substack{y\in X_{\tau}\\f(y)=x}}d_{x}\la_{x}[y]\in\Div(X,T_{sc})$
\end{center}
où $d_{x}$ est l'ordre du stabilisateur de $\Gamma$ agissant sur $p^{-1}(x)$.
Posons 
\begin{center}
$\overline{\cm}_{\tau^{-1}(\la)}:=\Hom(X_{\tau},[V_{G_{sc}}^{\tau^{-1}(\la)}/G_{sc}])$
\end{center}
et $\mathcal{A}_{{\tau^{-1}(\la)}}$ la base de Hitchin associée. Nous avons alors par \cite[(50)]{Hu} que la paire $(f^{*}E,f^{*}\phi)$ admet un relèvement $(E',\phi')\in\overline{\cm}_{\tau^{-1}(\la)}$, qui est unique modulo l'action de $\Gamma$ par translation à gauche sur $V^{\tau^{-1}(\la)}_{G_{sc}}$.
Nous disposons d'une deuxième action  de $\Gamma$ sur $V^{\tau^{-1}(\la)}_{G_{sc}}$, en voyant un élément $g\in\Gamma$ comme un automorphisme de $X_{\tau}$:
\begin{center}
$g_{*}\phi'=(g^{-1})^{*}\phi'$.
\end{center}
Il résulte alors de \cite[(50)]{Hu}, qu'en combinant ces deux actions, on obtient que $\phi'$ est équivariante pour l'action diagonale:
\begin{equation}
g_{*}(g.\phi'):=\phi'.
\label{eq1}
\end{equation}
En particulier, nous obtenons que:
\begin{equation}
\chi_{+}(\phi')\in \mathcal{A}_{\tau^{-1}(\la)}^{\Gamma}.
\label{eq2}
\end{equation}
On considère alors le schéma $V_{T_{sc}}^{\tau^{-1}(\la)}$ au-dessus de $X_{\tau}$. Il est muni d'une action de $\Gamma\times\Gamma$, la première venant de l'action naturelle de $\Gamma$ sur $V_{G_{sc}}$ et la deuxième issue de l'action de $K$ sur $X_{\tau}$.
Nous avons que le quotient de $V_{T_{sc}}^{\tau^{-1}(\la)}$ par $\Gamma\times\Gamma$ s'identifie à $V_{T}^{\la}$.
On forme  alors le revêtement intermédiaire:
\begin{center}
$V_{T}(\la,\tau):=V_{T_{sc}}^{\tau^{-1}(\la)}/\Delta(\Gamma)$
\end{center}
où $\Delta(\Gamma)$ désigne l'action diagonale et le quotient étant au sens des invariants.
On remarque que nous avons naturellement :
\begin{equation}
H^{0}(X,V_{T}(\la,\tau)/W)=\mathcal{A}_{\tau^{-1}(\la)}^{\Gamma}.
\label{equi2}
\end{equation}
\begin{defi}
A la suite d'Hurtubise-Markman \cite{Hu}, on considère l'ouvert $V_{T}(\la,\tau)^{0}$ qui est le complémentaire  de l'intersection du lieu de ramification de $V_{T}(\la,\tau)\rightarrow V_{T}(\la,\tau)/W$ avec les fibres de $V_{T}(\la,\tau)$ au-dessus de $\supp(\la)$.
Enfin, on considère $V_{T_{sc}}^{\tau^{-1}(\la),0}$ l'ouvert correspondant de $V_{T_{sc}}^{\tau^{-1}(\la)}$
\end{defi} 
Nous avons la proposition suivante tirée de \cite[Lem. 6.14]{Hu}
\begin{prop}\label{carhu}
Les schémas $V_{T}(\la,\tau)^{0}/W$ et $V_{T_{sc}}^{\tau^{-1}(\la),0}/W$ sont lisses et on a un carré cartésien:
$$\xymatrix{V_{T_{sc}}^{\tau^{-1}(\la),0}\ar[r]\ar[d]&V_{T}(\la,\tau)^{0}\ar[d]\\V_{T_{sc}}^{\tau^{-1}(\la),0}/W\ar[r]\ar[d]&V_{T}(\la,\tau)^{0}/W\ar[d]\\X_{\tau}\ar[r]&X}$$
\end{prop}
Cela va nous permettre de pouvoir calculer la dimension de $\abd$ quand $G$ est semisimple.
Nous combinons les résultats de \cite[6.16, 6.17, 6.19]{Hu}:
\begin{prop}\label{hu1}
Si $2g_{X_{\tau}}-1\prec\tau^{-1}\la$, alors l'espace des sections $H^{0}(X,V_{T}(\la,\tau)^{0}/W)$ est soit vide, soit une variété quasi-projective lisse de dimension:
\begin{center}
$\left\langle \rho,-w_{0}\la\right\rangle +r(1-g)$.
\end{center}
En particulier, $\abd$ est de même dimension.
\end{prop}
$\rmq$\begin{itemize}
\item
Il est à noter que pour un seul $\la_{v}$, l'accouplement $\left\langle \rho,\la_{v}\right\rangle$ peut être un demi-entier, néanmoins globalement lorsque l'espace des sections est non vide la somme globale $\left\langle \rho,\la\right\rangle$ est bien un entier (cf.\cite[Rmq 3.1.]{Hu}).
\item
Dans Hurtubise-Markman, $X$ est une courbe elliptique, mais l'argument est le même une fois que l'on ajoute dans Riemann-Roch la contribution venant du genre de la courbe. L'assertion sur $\abd$ se déduit du lemme \ref{surj3}.
\end{itemize}
\medskip

Nous avons besoin d'un renforcement de la proposition \ref{carhu}:
\begin{lem}\label{aj1}
Notons $\kc(\la,\tau):=V_{T}(\la,\tau)/W$, alors $\kc(\la,\tau)$ est un fibré vectoriel et la flèche $V_{T}(\la,\tau)\rightarrow\kc(\la,\tau)$ est finie plate.
\end{lem} 
\begin{proof}
On peut également obtenir $\kc(\la,\tau)$ comme le quotient de $\kc^{\tau^{-1}(\la)}$ par l'action diagonale de $K$.
On a un diagramme commutatif:
$$\xymatrix{\kc^{\tau^{-1}(\la)}\ar[d]\ar[r]&\kc(\la,\tau)\ar[d]\\X_{\tau}\ar[r]&X},$$
montrons qu'il est cartésien. On a une flèche surjective entre schémas affines $\kc^{\tau^{-1}(\la)}\rightarrow \kc(\la,\tau)\times_{X}X_{\tau}$
qui est birationnelle d'après la proposition \ref{carhu}. Enfin, le schéma $\kc(\la,\tau)$ est normal sur $X$ comme il est quotient d'une  flèche normale par un groupe fini. Par changement de base, $\kc(\la,\tau)\times_{X}X_{\tau}$ est normal sur $X_{\tau}$, donc normal. On déduit alors par le Main Theorem de Zariski que la flèche est un isomorphisme.
De plus, par descente fidèlement plate, nous obtenons que $\kc(\la,\tau)$ est un fibré vectoriel sur $X$.

Enfin, la flèche $V_{T}(\la,\tau)\rightarrow\kc(\la,\tau)$ est finie plate car $V_{T}(\la,\tau)$ est Cohen-Macaulay, comme quotient d'un schéma Cohen-Macaulay et $\kc(\la,\tau)$ est lisse.
\end{proof}

\subsection{Un champ de Picard dans le cas semisimple}\label{pichu}
Posons 
\begin{center}
$\abda:=\coprod\limits_{\tau\in S_{G}}H^{0}(X,\kc(\la,\tau))$.
\end{center}
Pour $a\in\abda(\bar{k})$ d'invariant $\tau$, il résulte de l'égalité \eqref{equi2}, que nous disposons d'un schéma en groupes lisses $J_{a}^{\sharp}$ sur $X_{\tau}$, obtenu en tirant le centralisateur régulier pour $G_{sc}$ par $a$.
La description galoisienne de la proposition \ref{galois} nous dit que $J_{a}^{\sharp}$ s'obtient comme le complémentaire des hyperplans de racines dans le schéma $\pi^{\sharp}_{*}(T_{sc})^{W}$ où 
\begin{center}
$\pi^{\sharp}:V_{T}^{\tau^{-1}(\la)}\rightarrow\kc^{\tau^{-1}(\la)}$.
\end{center}
Comme $\pi^{aug}:V_{T}(\la,\tau)\rightarrow \kc(\la,\tau)$ est fini plat, $\pi^{aug}_{*}(T)^{W}$ est un schéma en groupes commutatifs lisses sur $ \kc(\la,\tau)$.
En particulier, on peut considérer le schéma en groupes commutatifs lisses  $J_{a}^{aug}$ sur $X$ qui s'obtient également comme complémentaire des hyperplans de racines dans $\pi^{aug}_{*}(T)^{W}$.
On  construit de cette manière un champ de Picard $\cP_{a}^{aug}$, en considérant les $J_{a}^{aug}$-torseurs sur $X$. En faisant varier $a$, on obtient un champ $\cP^{aug}$ sur $\abdae$.
Nous avons la proposition suivante:
\begin{prop}\label{hu2}
Le champ $\cP^{aug,\heartsuit}$ est lisse au-dessus de $\abdah$. De plus, pour $a\in\abdah(\bar{k})$, on a la formule de dimension:
\begin{center}
$\dim\cP_{a}^{aug}=\left\langle \rho,-w_{0}\la\right\rangle +r(g-1)$.
\end{center}
\end{prop}
\begin{proof}
La lissité s'obtient de la même manière que la proposition \ref{picardlisse}. Comme $\cP^{aug,\heartsuit}$ est lisse au-dessus de $\abdah$, il suffit de calculer la dimension de l'ouvert $H^{0}(X,V_{T}(\la,\tau)^{0}/W)$.
La preuve est alors la même que pour la proposition \ref{dimpic} en utilisant \cite[6.16]{Hu} pour identifier $\Lie(J_{a}^{aug})$.
\end{proof}
On stratifie alors $\abda$ par l'invariant $\delta$ donné par:
\begin{center}
$\forall~ a\in\abda, \delta(a)=\dim (\cR_{a}^{aug})$
\end{center}
où $\cR_{a}^{aug}$ est la partie affine de $\cP_{a}^{aug}$.

\subsection{L'invariant $\delta$}
On définit un invariant $\delta$ sur l'ouvert $\abdan$ à partir de celui sur $\abda$. On commence par un lemme qui va se déduire de l'étude précédente:
\begin{lem}\label{surj3}
On a un morphisme fini surjectif:
\begin{center}
$\abdaan\rightarrow\abdan$
\end{center}
\end{lem}
\begin{proof}
D'après \eqref{equi2}, nous avons $\mathcal{A}_{\tau^{-1}(\la)}^{\Gamma}= H^{0}(X,\kc(\la,\tau))$.
En particulier, du morphisme fini $V_{T}(\la,\tau)/W\rightarrow\kcd$, nous obtenons, en utilisant \cite[Lem. 7.3]{N2}, une flèche propre :
\begin{center}
$\coprod\limits_{\tau\in S_{G})}\mathcal{A}_{\tau^{-1}(\la)}^{\Gamma}\rightarrow\abd$
\end{center}
De plus, elle est quasi-finie, puisque chaque fibre consiste en une $K$-orbite de sections de $\kc(\la,\tau)$.

Il ne nous reste plus que l'assertion de surjectivité. Cela vient du fait que pour tout $a\in\abdan(\bar{k})$, d'après la proposition \ref{surj1}, il existe $(E,\phi)\in\cmdan(\bar{k})$ qui s'envoie sur $a$. La paire $(E,\phi)$ admet un certain invariant $\tau$ et la construction du paragraphe précédent nous fournit d'après \eqref{eq2}, un relèvement $a'\in\mathcal{A}_{\tau^{-1}(\la)}^{\Gamma}$ de $a$.
\end{proof}

Comme l'invariant $\delta$ d'une section $a\in\abda$ ne dépend pas de l'action de $K$, on obtient en particulier, une stratification sur $\abdan$:
\begin{center}
$\abdan=\coprod\limits_{\delta\in\mathbb{N}}\mathcal{A}_{\delta}^{ani}$.
\end{center}

\subsection{Comparaison d'invariants $\delta$}
Soit $\kappa$ une donnée endoscopique elliptique de $G$, de groupe $H$.
D'après \eqref{transII}, on a  un morphisme:
\begin{center}
$\nu:\abdeth\rightarrow\abd$.
\end{center}
On s'intéresse dans un premier temps à la \og strate la plus haute $\la$\fg~ dans $\abdeth$. Nous avons alors par composition une flèche:
\begin{center}
$\nu_{aug}:\abdae\rightarrow\abd$.
\end{center}

\begin{lem}
Soit $a_{H}\in\abdae(\bar{k})$ d'image $a\in\abdk$, on a un morphisme 
\begin{center}
$J_{a}\rightarrow J_{a_{H}}^{aug}$
\end{center}
qui est génériquement un isomorphisme.
\end{lem}
\begin{proof}
La preuve est la même que \cite[Prop. 2.5.1]{N}, en utilisant la description galoisienne de $J_{a_{H}}^{aug}$ de la section \ref{pichu}.
\end{proof}
Nous obtenons alors un morphisme surjectif:
\begin{center}
$\cP_{a}\rightarrow\cP_{a_{H}}^{aug}$
\end{center}
de noyau:
\begin{center}
$\mathcal{R}_{H,a_{H}}^{G}=H^{0}(\bar{X},J_{a_{H}}^{aug}/J_{a})$
\end{center}
qui est de dimension:
\begin{center}
$r_{H}^{G}(\la):=\dim\cP_{a}-\dim(\cP^{aug}_{a_{H}})$.
\end{center}

Soit $J_{a_{H}}^{aug,\flat}$ le modèle de Néron de $J_{a_{H}}^{aug}$. Nous avons des morphismes de groupes
\begin{center}
$J_{a}\rightarrow J_{a_{H}}^{aug}\rightarrow J_{a_{H}}^{aug,\flat}$ 
\end{center}
qui sont génériquement des isomorphismes. En particulier, $J_{a_{H}}^{aug,\flat}$  est également le modèle de Néron de $J_{a}$. En se rappelant que $\mathcal{R}_{a}:=\Ker(\cP_{a}\rightarrow\cP_{a}^{\flat})$, nous obtenons alors la suite exacte:
$$\xymatrix{1\ar[r]&\mathcal{R}_{H,a_{H}}^{G}\ar[r]&\mathcal{R}_{a}\ar[r]&\mathcal{R}_{a_{H}}^{aug}\ar[r]&1}$$
qui nous fournit la formule de dimension:
\begin{center}
$\delta(a)-\delta_{H}(a_{H})=r_{H}^{G}(\la)$.
\end{center}
\begin{cor}\label{indep}
Nous avons $r_{G}^{H}(\la)=\dim\abd-\dim\abde$. En particulier, pour tout $a_{H}\in\abde(\bar{k})$, d'image $a\in\abdhk$,  $\delta_{a}-\delta_{H,a_{H}}$ est indépendant de $a_{H}$.
\end{cor}
\begin{proof}
Il résulte de la section \ref{dimhitchin2} et des propositions \ref{hu1} et \ref{hu2} que nous avons:
\begin{center}
$r_{H}^{G}(\la)=\dim\cP_{a}-\dim\cP_{H,a_{H}}=\dim\abd-\dim\abde$.
\end{center}
\end{proof}
$\rmq$ Pour les strates de $\abdeth$ associées à un cocaractère $\mu<\la$, l'application $\nu$ se factorise par $\mathcal{A}_{\mu}$ et on est ramené au cas précédent.
\subsection{Une fibration de Hitchin augmentée}\label{hitaug}
Fixons un invariant topologique $\tau$, comme nous avons une action de $\Gamma\times\Gamma$ sur $V_{G}^{\tau^{-1}(\la)}$, on forme le quotient:
\begin{center}
$V_{G}(\la,\tau)=V_{G}^{\tau^{-1}(\la)}/\Delta\Gamma$.
\end{center}
Partant d'une paire $(E,\phi)\in\cmd(\bar{k})$, on a construit une paire $(E',\phi')\in\overline{\cm}_{\tau^{-1}(\la)}$ telle que $\phi'$ se factorise en une section $\phi'\in H^{0}(X, V_{G}(\la,\tau)\wedge^{G_{sc}}E')$.
On pose alors 
\begin{center}
$\cmdaa:=\coprod\limits_{\tau\in S_{G}}\Hom(X,[V_{G}(\la,\tau)/G])$.
\end{center}
et en poussant $E'$ en un $G$-torseur par la flèche $G_{sc}\rightarrow G$, la paire $(E',\phi')$ nous fournit un point de $\cmda(\bar{k})$ qui s'envoie sur $(E,\phi)$.

\begin{lem}\label{surj4}
La flèche $\cmdaa\rightarrow\cmd$ est finie, surjective et nette.
\end{lem}
\begin{proof}
La surjectivité est claire en vertu de ci-dessus. Comme $V_{G}(\la,\tau)\rightarrow V_{G}^{\la}$ est finie surjective génériquement étale, il résulte de \cite[Lem. 7.3-p. 433]{N2}, que la flèche induite $\cmdaa\rightarrow\cmd$ est finie et non-ramifiée.
\end{proof}
On obtient donc un diagramme commutatif:
$$\xymatrix{\cmdaan\ar[r]^{p}\ar[d]_{f^{aug,ani}}&\cmdan\ar[d]^{f^{ani}}\\\abdaan\ar[r]^{p}&\abdan}$$
avec $f^{ani}$ propre d'après le théorème \ref{proprete} et les flèches horizontales finies surjectives d'après les lemmes \ref{surj3} et \ref{surj4}, en particulier, on en déduit que $f^{aug,ani}$ est aussi propre.

Le champ $\cmdaa$ est muni d'une action de $\cP^{aug}$, induite par l'action du champ de Picard de $\overline{\cm}_{\tau^{-1}(\la)}$, avec un ouvert régulier sur lequel il agit transitivement.
Enfin, une analyse similaire à la proposition \ref{ncl} nous donne:
\begin{prop}
On a un morphisme surjectif de faisceaux entre le faisceau constant $X_{*}(T)$ et $\pi_{0}(\cP^{aug})$. 
\end{prop}

\section{Construction d'une présentation géométrique}
Dans ce chapitre nous construisons une résolution des singularités en un sens faible  de $\cmd$, que l'on appelle une présentation géométrique du complexe d'intersection de $\cmd$. L'utilité d'une telle présentation vient du théorème du support. En effet, dans \cite[sect.7]{N}, on démontre un théorème du support pour une flèche $f:M\rightarrow S$ propre munie d'une action d'un schéma en groupes commutatif lisse $g:P\rightarrow S$. Néanmoins, dans l'énoncé de Ngô, la source $M$ est supposée lisse, ce qui n'est évidemment pas le cas dans la situation qui nous concerne. Nous sommes donc forcés d'introduire un espace qui fait office de résolution des singularités. Nous introduisons cette variante affaiblie de résolution des singularités, dans la mesure où nous ne sommes pas en mesure de prouver que l'espace que nous construisons au-dessus de $\cmd$ est effectivement lisse, toutefois nous n'avons pas besoin d'hypothèses si contraignantes.
Un tel espace au-dessus de $\cmd$ va s'obtenir par changement de base d'une résolution des singularités du champ de Hecke. Avant de commencer par des rappels locaux sur la grassmannienne affine, nous introduisons la définition suivante:
\begin{defi}\label{defiedulc}
Soit $M$ un champ algébrique de type fini, équidimensionnel sur $k$. On appelle $\pi:\tilde{M}\rightarrow M$ une présentation géométrique du complexe d'intersection de $M$ si:
\begin{itemize}
\item
$\pi$ est propre,
\item
Il existe un ouvert non vide $U$ de $M$ tel que la flèche $\pi:\pi^{-1}(U)\rightarrow U$ soit lisse,
\item
Le complexe d'intersection $IC_{M}$ est un facteur direct de $\pi_{*}\bql$.
\end{itemize}
 
\end{defi}
$\rmq$\begin{itemize}
\item
Nous raccourcissons dans la suite l'appellation présentation géométrique du complexe d'intersection en présentation géométrique.
\item
Il est immédiat qu'une résolution des singularités fournit une présentation géométrique.
\item
Un deuxième exemple vient de la situation suivante:
$$\xymatrix{Y'\ar[d]\ar[r]&X'\ar[d]\\Y\ar[r]^{\Delta}&X}$$
où $X'\rightarrow X$ est une résolution des singularités d'un champ algébrique de type fini, équidimensionnel sur $k$, $\Delta:Y\rightarrow X$ une immersion de codimension $d$ telle que $\Delta^{*}[-d]IC_{X}=IC_{Y}$.
\end{itemize}

\subsection{Rappels locaux et résolution de Demazure}
On rappelle des résultats sur la résolution des singularités de $\overline{\Gr}_{\la}$, on pourra consulter à cet égard \cite{Rich}.
Soit $F=k((\pi))$ et $\co=k[[\pi]]$. Soit $G$ connexe réductif déployé sur $k$ tel que $G_{der}$ est simplement connexe. Soit $K$ le $k$-schéma en groupes dont les $k$-points sont $G(\co)$.
Soit $B$ un sous-groupe de Borel de $G$, nous avons une flèche d'évaluation en zéro:
\begin{center}
$\ev:K\rightarrow G$.
\end{center}
Plus généralement, on considère la flèche pour $n\in\mathbb{N}$:
\begin{center}
$\ev^{n}:K\rightarrow G(\co/\pi^{n}\co)$.
\end{center}
On appelle sous-groupe de congruence de niveau $n$ le groupe $K^{n}:=\Ker(\ev^{n})$. On pose également $K_{n}:=G(\co/\pi^{n}\co)$.
Nous rappelons que nous avons une décomposition de Cartan:
\begin{center}
$G(F)=\coprod\limits_{\la\in X_{*}(T)^{+}}K\pi^{\la}K$
\end{center}
et que $\Gr_{\la}:=K\pi^{\la}K/K$ (resp. $\overline{\Gr}_{\la}:=\overline{K\pi^{\la}K}/K$). Dans la suite, nous  simplifierons la notation $\pi^{\la}$ en $\la$.
\begin{lem}\label{fini}
On considère pour un cocaractère dominant $\la$, le fermé $\overline{\Gr}_{\la}\subset\Gr$. L'action de $K$ à gauche sur $\overline{\Gr}_{\la}$ se factorise par un quotient algébrique de type fini $K\rightarrow K_{n_{\la}}$.
\end{lem}
\begin{proof}
On voit $G$ comme un sous-groupe de $K$ et on regarde son action sur $\Gr_{\la}$, le stabilisateur de l'élément $\la$ est un sous-groupe parabolique $P_{\la}$. Il résulte alors de  \cite[sect. 2, (2.4)]{MV} que l'on peut voir $\Gr_{\la}$ comme un fibré  $G$-équivariant sur $G/P_{\la}$ dont la fibre au-dessus de $G/P_{\la}$ est l'orbite $K^{1}.\la$. Le stabilisateur est donné par $K^{1}\cap\la K^{1}\la^{-1}$. Montrons qu'il contient un certain groupe de congruence.
Posons $n_{\la}:=\max\limits_{\alpha\in R_{+}}(\left\langle \alpha,\la\right\rangle)$ et pour $i\in\mathbb{Z}$ et $\alpha$ une racine de $G$, soit:
\begin{center}
$U_{\alpha,i}:=U_{\alpha}(\pi^{i}x)$.
\end{center}
Il résulte  alors de la formule:
\begin{center}
$\la U_{\alpha,i}\la^{-1}=U_{\alpha,i+\left\langle \alpha,\la\right\rangle}$
\end{center}
où $K^{n_{\la}}\subset K^{1}\cap\la K^{1}\la^{-1}$.
On en déduit alors que $K$ agit sur $\Gr_{\la}$ via $K/K^{n_{\la}}=K_{n_{\la}}$ et donc également sur $\overline{\Gr}_{\la}$ par continuité.
\end{proof}
On considère le sous-groupe d'Iwahori $I=\ev^{-1}(B)$. Posons $\mathcal{B}=G(F)/I$ la variété de drapeaux affines, elle admet une flèche de projection projective lisse $\pi:\cB\rightarrow\Gr$ qui lui donne une structure d'ind-schéma. Elle admet la description modulaire suivante tirée de  \cite[sect. 1.1.3]{Ga}
\begin{defi}
Pour une $k$-algèbre $R$, $\cB(R)$ est le groupoïde des triplets $(E,\beta,\eps)$ avec $E$ un $G$-torseur sur $R\hat{\times}\Spec(\co)$, $\beta$ une trivialisation sur $R\hat{\times}\Spec(F)$ de $E$ et $\eps$ une réduction  de $E_{\vert\Spec(R)}$ au Borel $B$.
\end{defi}
On définit le groupe de Weyl affine étendu $W_{aff}:=N_{G(F)}(T(F))/T(\co)$, il s'identifie à $W\rtimes X_{*}(T)$.
Sur $W_{aff}$, on a une fonction de longueur $l:W_{aff}\rightarrow\mathbb{N}$ définie par:
\begin{center}
$l(w\pi^{\la})=\sum\limits_{\substack{\alpha>0,\\ w\alpha<0}}\left|\left\langle \alpha,\la\right\rangle+1\right|+\sum\limits_{\substack{\alpha>0,\\ w\alpha>0}}\left|\left\langle \alpha,\la\right\rangle\right|$.
\end{center}
D'après la décomposition d'Iwahori-Bruhat, nous avons une décomposition en $I$-orbites:
\begin{center}
$G(F)=\coprod\limits_{w\in W_{aff}}IwI$.
\end{center}
Cette décomposition induit une décomposition de la variété de drapeaux affine:
\begin{center}
$\cB=\coprod\limits_{w\in W_{aff}}\cB_{w}$
\end{center}
où les strates $\cB_{w}:=IwI/I$ sont des espaces affines de dimension $l(w)$. Le groupe de Weyl affine $W_{aff}$ admet un ordre partiel, dit de Bruhat tel que:
\begin{center}
$\overline{\cB}_{w}=\coprod\limits_{w'\leq w}\cB_{w'}$.
\end{center}
Considérons un triplet $(\beta,(E,\eps),(E,\eps'))$ constitué d'un $G$-torseur sur $\Spec \co$ avec sa réduction en fibre spéciale au Borel $B$ et $\beta$ un isomorphisme générique sur $\Spec F$ entre ces deux-torseurs, leur position relative fournit un élément $w$ de $W_{aff}$, on note alors $\inv_{I}(E,E')=w$.
Comme $G_{der}$ est simplement connexe, le groupe de Weyl affine étendu a une structure de groupe de Coxeter engendrée par l'ensemble des réflexions simples $\{s_{i},0\leq i\leq r\}$, constitué des réflexions simples $s_{i}$ de $W$ pour $1\leq i\leq r$ auquel on ajoute la réflexion $s_{0}$ qui correspond à la racine affine simple $\alpha_{0}=\pi^{\check{\alpha_{0}}}s_{\alpha_{0}}$ où $\alpha_{0}$ est l'unique racine la plus haute. 
On rappelle qu'il existe une unique racine positive $\alpha_{0}=\sum n_{i}\alpha_{i}$ telle que pour toute racine  positive $\beta=\sum m_{i}\alpha_{i}$, on a $m_{i}\leq n_{i}$.
Faltings a alors introduit dans  \cite[sect.3]{Fal} une résolution des singularités à la Demazure de $\overline{\cB}_{w}$, pour $w\in W_{aff}$.
A chaque racine simple $\alpha_{i}$, on associe le parahorique $\cP_{i}:=I\cup Is_{i}I$ pour $0\leq i\leq r$.
Etant donné $w\in W_{aff}$, on peut donc écrire une décomposition réduite de $w=s_{a_{1}}\dots s_{a_{l}}$ et on considère le schéma:
\begin{center}
$C(s_{a_{1}},\dots,s_{a_{l}})=\cP_{a_{1}}\times\dots\times\cP_{a_{l}}/I^{l}$.
\end{center}
où $I^{l}$ agit à droite par:
\begin{center}
$(p_{1},\dots,p_{l}).(b_{1},\dots,b_{l})=(p_{1}b_{1},b_{1}^{-1}p_{2}b_{2},\dots,b_{l-1}^{-1}p_{l}b_{l}).$
\end{center}
Par construction, $C(s_{a_{1}},\dots,s_{a_{l}})$ est un fibré projectif au-dessus de $C(s_{a_{1}},\dots,s_{a_{l-1}})$ en $\cP_{a_{l}}/I=\mathbb{P}^{1}$. Par récurrence, il est donc lisse, propre et irréductible.
La multiplication des composantes nous fournit une flèche:
\begin{center}
$\theta_{w}:C(s_{a_{1}},\dots,s_{a_{l}})\rightarrow\overline{\cB}_{w}$.
\end{center}
Nous avons alors la proposition suivante tirée de \cite[sect. 3]{Fal}:
\begin{prop}
La flèche $\theta_{w}$ est une résolution des singularités au sens suivant:
\begin{itemize}
\item
Le schéma $C(s_{a_{1}},\dots,s_{a_{l}})$ est lisse.
\item
La flèche $\theta_{w}$ est propre surjective.
\item
La flèche $\theta_{w}$ est un isomorphisme au-dessus de $\cB_{w}$.
\item
La flèche $\theta_{w}$ est $I$-équivariante.
\end{itemize}
\end{prop}
Nous allons maintenant voir comment de cette résolution, nous en déduisons une pour $\overline{\Gr}_{\la}$ pour $\la\in X_{*}(T)^{+}$.
Tout d'abord, il résulte de la formule de dimension que $I\la I/I$ est de dimension $l(\la)=2\left\langle \rho,\la\right\rangle$.
On a une flèche de projection $p:\cB\rightarrow\Gr$ qui est projective lisse et $I\la I/I$ s'envoie par $p$ sur $I\la K/K$.
On note alors:
\begin{center}
$p_{\la}:I\la I/I\rightarrow I\la K/K$
\end{center}
la flèche induite.
\begin{lem}\label{grass}
La flèche $p_{\la}$ est un isomorphisme et $I\la K/K$ est  un ouvert dense de $\Gr_{\la}$.
Le morphisme $p_{\la}$ s'étend en un morphisme propre, surjectif et birationnel noté de la même manière 
\begin{center}
$p_{\la}:\overline{I\la I}/I\rightarrow\overline{\Gr}_{\la}$.
\end{center}
\end{lem}

\begin{proof}
Notre démonstration s'inspire de Ngô-Polo \cite[Lem 2.2]{NP}.
On définit le $k$-groupe $G(k[\pi^{-1}])$ qui représente le foncteur $R\mapsto G(R[\pi^{-1}])$.
Nous avons un morphisme 
\begin{center}
$\ev_{\infty}:G(k[\pi^{-1}])\rightarrow G$
\end{center}
qui envoie $\pi^{-1}$ sur 0. Soit $U$ le radical unipotent de $B$, on considère alors $J:=\ev_{0}^{-1}(U)\subset I$, et  $R=\Ker\ev_{\infty}$. On pose alors $J^{\la}=J\cap\la R\la^{-1}$. Pour $i\in\mathbb{Z}$, on note $U_{\alpha,i}:=U_{\alpha}(\pi^{i}x)$.

D'après \cite[Lem 2.2]{NP}, en choisissant un ordre total sur l'ordre des facteurs,  on a un isomorphisme de schémas donné par la multiplication:
\begin{center}
$\prod\limits_{\substack{\alpha\in R^{+},\\\left\langle \alpha,\la\right\rangle>0}}\prod\limits_{i=0}^{\left\langle \alpha,\la\right\rangle-1}U_{\alpha,i}\rightarrow J^{\la}$,
\end{center}
en particulier, $J^{\la}$ est un espace affine de dimension $2\left\langle \rho,\la\right\rangle$.
De plus, toujours d'après \cite[Lem 2.2]{NP}, la flèche 
\begin{center}
$j_{\la}:J^{\la}\rightarrow K\la K/K$,
\end{center}
donnée par $j\mapsto j\la$ est une immersion ouverte. 
La flèche $j_{\la}:J^{\la}\rightarrow K\la K/K$ se factorise en une flèche notée de la même manière 
\begin{center}
$j_{\la}:J^{\la}\rightarrow I\la K/K$.
\end{center}
Montrons que c'est un isomorphisme. Tout d'abord, nous avons qu'une $I$-orbite est également une $J$-orbite, ainsi on obtient un isomorphisme
\begin{center}
$J/(J\cap\la K\la^{-1})\rightarrow I\la K/K$.
\end{center}
Identifions le stabilisateur $J\cap\la K\la^{-1}$; nous avons les formules suivantes:
\begin{center}
$\forall~ i\geq0, \la^{-1}U_{\alpha,i}\la=U_{\alpha,i-\left\langle\alpha,\la \right\rangle}$,

$\forall~ k\geq1, \la^{-1}U_{-\alpha,k}\la=U_{-\alpha,k+\left\langle\alpha,\la \right\rangle}$.
\end{center}
On obtient alors un isomorphisme de schémas:
\begin{center}
$\prod\limits_{\alpha\in R^{+}}\prod\limits_{i\geq\left\langle \alpha,\la\right\rangle}U_{\alpha,i}\prod\limits_{k\geq0}U_{-\alpha,k}\rightarrow J\cap\la K\la^{-1}$.
\end{center}
Ainsi, $J^{\la}$ s'identifie naturellement à $J/(J\cap\la K\la^{-1})$, d'où l'on déduit que $j_{\la}$ est un isomorphisme et $I\la K/K$ est ouvert dans $\Gr_{\la}$.
La description ci-dessus nous montre également que la flèche de projection 
\begin{center}
$p_{\la}:I\la I/I\rightarrow I\la K/K$
\end{center}
est un isomorphisme.
Comme de plus, $I\la K/K$ est ouvert  dans $\Gr_{\la}$ et que $\Gr_{\la}$ est irréductible, nous en déduisons que $\overline{I\la K}/K=\overline{\Gr}_{\la}$.
Maintenant, on déduit de la propreté de $\cB\rightarrow\Gr$ et de l'irréductibilité des deux espaces que $p_{\la}$ se prolonge en une flèche:
\begin{center}
$p_{\la}:\overline{I\la I}/I\rightarrow \overline{\Gr}_{\la}$,
\end{center}
ce qu'on voulait.
\end{proof}
En particulier, une résolution des singularités de $\overline{\Gr}_{\la}$ nous sera donnée par une résolution des singularités de $\overline{I\la I}/I$ à la Demazure-Faltings.
On renvoie pour plus de détails à \cite{Rich} pour le traitement du cas quasi-déployé.
\subsection{Globalisation au champ de Hecke}
Dans cette section, on montre comment de la résolution locale de $\overline{\Gr}_{\la}$ l'on déduit une résolution des singularités globale de la strate correspondante du champ de Hecke.
On considère $\la=\sum\limits_{v\in S}\la_{v}[v]$, pour $S$ un ensemble fini de points fermés.
Soit $B$ un Borel de $G$, en chaque point $v\in S$, nous avons une flèche d'évaluation 
\begin{center}
$\ev_{v}:K_{v}:=G(\co_{s})\rightarrow G$,
\end{center}
et on pose $I_{v}=\ev_{v}^{-1}(B)$, $K_{S}=\prod\limits_{v\in S}K_{v}$, $\cB_{S}=\prod\limits_{v\in S}\cB_{v}$ et $I_{S}=\prod\limits_{v\in S}I_{v}$.\\
Pour $w=\sum\limits_{v\in S}w_{v}[v]$ avec $w_{s}\in W_{aff}$, on pose $\overline{\cB}_{w}:=\prod\limits_{v\in S}\overline{\cB}_{w_{v}}$.

De plus, sans restreindre la généralité, on suppose que pour tout $v\in S$, $l(w_{v})=n$. On écrit alors $w=s_{a_{1}}\dots s_{a_{n}}$ avec $s_{a_{i}}=\sum\limits_{v\in S}s_{a_{i}}^{v}[v]$ où chaque $s_{a_{i}}^{v}$ est une réflexion simple de $W_{aff}$ et on pose:
\begin{center}
$C(s_{a_{1}},\dots,s_{a_{n}}):=\prod\limits_{v\in S}C(s^{v}_{a_{1}},\dots, s^{v}_{a_{n}})$.
\end{center}
Dans la suite, on choisit $w$ tel que l'on a une flèche propre birationnelle surjective $p_{\la}:\overline{\cB}_{w}\rightarrow\grl$ qui est un isomorphisme au-dessus de $I\la K/K$.

Posons $D_{\la}=\sum\limits_{v\in S}n_{\la_{v}}[v]$, avec $n_{\la_{v}}$ donné par le lemme \ref{fini} et $K_{\la}=\prod\limits_{v\in S}K_{n_{\la_{v}},v}$ qui est un quotient de $K_{S}$.
On définit alors le $K_{\la}$-torseur $\Bunl$ au-dessus de $\Bun$, qui classifie les paire $(E,\iota)$ où $\iota$ est une structure de niveau $D_{\la}$ sur $E$.
Nous allons commencer par un lemme dans l'esprit du lemme \ref{va}:
\begin{prop}\label{vab}
\begin{enumerate}
\item
Le changement de base par le $K_{\la}$-torseur $\Bunl$ au-dessus de $\Bun$ fournit un isomorphisme canonique entre $\overline{\cH}_{\la}\times_{\Bun}\Bunl$ et $\overline{\Gr}_{\la}\times\Bunl$.
\item
La projection $\overline{\Gr}_{\la}\times\Bunl\rightarrow \overline{\Gr}_{\la}$ induit un morphisme lisse de champs algébriques entre le champ de Hecke $\overline{\cH}_{\la}$ et le champ quotient $[K_{\la}\backslash\overline{\Gr}_{\la}]$.
\end{enumerate}
\end{prop}
\begin{proof}
On considère le $K_{S}$-torseur $\Bun'$ au-dessus de $\Bun$ qui classifie les paires $(E,\theta)$ où $\theta$ est une trivialisation de $E$ sur $K_{S}$. 
Si on tire $\chm$ au-dessus de $\Bun'$, on obtient un isomorphisme canonique avec $\Bun'\times_{k}\grl$ compatible à l'action de $K_{S}$ des deux côtés.
Ainsi, nous avons un isomorphisme canonique:
\begin{center}
$\chm\stackrel{\cong}{\rightarrow}[(\Bun'\times_{k}\grl)/K_{S}]$.
\end{center}
Or, on a vu que l'action de $K_{S}$ sur $\grl$ se factorise par $K_{\la}$ et on obtient donc un isomorphisme canonique avec le champ $[(\Bunl\times_{k}\grl)/K_{\la}]$.
Ainsi, en considérant la flèche de projection $K_{\la}$-équivariante:
\begin{center}
$\Bunl\times_{k}\grl\rightarrow\grl$, 
\end{center}
elle induit par passage au quotient une flèche lisse $\chm\rightarrow [K^{\la}\backslash\grl]$.
\end{proof}
Ce lemme nous permet d'obtenir un modèle local pour les singularités du champ de Hecke. On commence par résoudre les singularités localement puis on tire la figure par le morphisme du global vers le local.
On a une flèche de réduction $\pi_{\la}:K_{S}\rightarrow K_{\la}$ et on pose $I_{\la}:=\pi_{\la}(I_{S})$, on forme les carrés cartésiens suivant:
$$\xymatrix{\Conv^{n}_{(s_{a_{i}})}(\cH)\ar[d]\ar[r]&[I_{\la}\backslash C(s_{a_{1}},\dots,s_{a_{n}})]\ar[d]^{\theta_{w}}\\\overline{\cH}_{w}^{I}\ar[d]\ar[r]&[I_{\la}\backslash\overline{\cB}_{w}]\ar[d]^{p_{\la}}\\\overline{\cH}_{\la}^{par}\ar[d]\ar[r]&[I_{\la}\backslash\grl]\ar[d]^{\phi}\\\chm\ar[r]&[K_{\la}\backslash\grl]}$$
où $w=s_{a_{1}}\dots s_{a_{n}}$, la flèche $\theta_{w}$ est la flèche de résolution des singularités, $\phi$ correspond à l'ajout de facteurs $G/B$, qui est projective lisse et $p_{\la}$ est birationnelle.
Toutes les flèches horizontales sont lisses.
Ce diagramme nous permet donc d'obtenir les objets globaux dont nous aurons besoin. Nous allons maintenant introduire les problèmes de modules qu'ils classifient.
\begin{defi}
On note $\cH_{S}^{I}$ le ind-champ qui classifie les uplets $((E,\eps),(E',\eps'),\beta)$ où $E$ et $E'$ sont des $G$-torseurs sur $X$, $\eps$ et $\eps'$ sont des $B$-réductions sur $S$ de $E$ et $E'$ et $\beta$ un isomorphisme au-dessus de $X-S$.
\end{defi}
$\rmq$ Aux points de $S$, nous pouvons regarder la position relative des paires $(E,\eps)$ et $(E',\eps')$ et nous obtenons pour tout point $v\in S$ un élément $w_{v}\in W_{aff}$. Ainsi, le champ $\cH_{S}^{I}$ admet une stratification par des sous-champs fermés $\overline{\cH}_{w}^{I}$ où $w=\sum\limits_{v\in S}w_{v}[v]$ où $\inv_{I}(E,E')\leq w$.

Considérons le champ $\Bunp$ qui classifie les paires $(E,\eps)$ constituées d'un $G$-torseur sur $X$ et d'une $B$-réduction sur $S$.

\begin{cor}\label{vat}
Après changement de base par $\Bunpl:=\Bunl\times_{\Bun}\Bunp$ lisse à fibres géométriquement connexes, le champ $\overline{\cH}_{w}^{I}\times_{\Bunp}\Bunpl$ s'identifie canoniquement à $\Bunpl\times\overline{\cB}_{w}$.
\end{cor}
\begin{proof}
La preuve résulte de la proposition \ref{vab} puisque qu'une fois que l'on a trivialisé la fibration $\overline{\cH}_{\la}$ au-dessus de $\Bunl$, tous les champs au-dessus se trivialisent simultanément.
\end{proof}
Nous allons maintenant introduire la résolution de Demazure de $\overline{\cH}_{w}^{I}$.

\begin{defi}\label{convp}
Soit $\Conv_{S}^{n}(\cH)$ le ind-champ dont les $T$-points sont donnés par les uplets\\
$((E_{0},\eps_{0}),.., (E_{n},\eps_{n}),\beta_{1},..,\beta_{n})$ où les $E_{k}$ sont des $G$-torseurs sur $X\times T$, les $\beta_{k}$ 
\begin{center}
$\beta_{k+1}:E_{k\scriptscriptstyle{\vert X\times T-\bigcup\limits_{s\in S}{\Gamma_{s\times T}}}}\rightarrow E_{k+1\scriptscriptstyle{\vert X\times T-\bigcup\limits_{s\in S}{\Gamma_{s\times T}}}}$ 
\end{center}
sont des isomorphismes et les $\eps_{k}$ sont des réductions des torseurs $E_{k\scriptscriptstyle{\vert\bigcup_{s\in S}{\Gamma_{s\times T}}}}$ au Borel $B$.
\end{defi}

Étant donné un élément $w=\sum\limits_{v\in S}w_{v}[v]$ avec $w_{v}\in W_{aff}$, on écrit alors $w=s_{a_{1}}\dots s_{a_{n}}$ avec $s_{a_{i}}=\sum\limits_{v\in S}s_{a_{i}}^{v}[v]$ où chaque $s_{a_{i}}^{v}$ est une réflexion simple de $W_{aff}$.\\
On considère alors le sous-champ $\Conv_{(s_{a_{i}})}^{n}(\cH)$ de $\Conv_{S}^{n}(\cH)$ où 
\begin{center}
$\forall ~i, \inv_{I}(E_{i},E_{i+1})=s_{a_{i+1}}$.
\end{center}
Il est muni d'une flèche:
\begin{center}
$\theta_{w}:\Conv_{(s_{a_{i}})}^{n}(\cH)\rightarrow\overline{\cH}_{w}^{I}$
\end{center}
qui envoie un uplet de $\Conv_{(s_{a_{i}})}^{n}(\cH)$ sur $((E_{0},\eps_{0}),(E_{n},\eps_{n}),\beta_{n}\circ\dots\circ\beta_{1})$ et par composition nous avons une flèche: 
\begin{center}
$p_{1}\circ \theta_{w}:\Conv_{(s_{a_{i}})}^{n}(\cH)\rightarrow\Bunp$.
\end{center}

\begin{prop}\label{va3}
En faisant le changement de base par $\Bunpl$ au-dessus de $\Bunp$, nous obtenons que $\Conv_{(s_{a_{i}})}^{n}(\cH)$ s'identifie après changement de base à\\ $\Bunpl\times C(s_{a_{1}},\dots,s_{a_{n}})$.
En particulier, $\Conv_{(s_{a_{i}})}^{n}(\cH)$ est lisse. De plus, la flèche $\theta_{w}$ est propre, surjective et birationnelle.
\end{prop}
\begin{proof}
La première partie de la preuve est un corollaire de la proposition \ref{vab}. Enfin, l'assertion sur $\theta_{w}$ vient de l'assertion locale correspondante.
\end{proof}
Il ne nous reste plus qu'à faire le lien avec le champ de Hecke classique.
On commence par faire le changement de base à $\Bunp$. On pose alors 
\begin{center}
$\overline{\cH}_{\la}^{par}=\overline{\cH}_{\la}\times_{\Bun}\Bunp$.
\end{center}
On prend alors $w$ associé à $\la$ défini dans la section précédente de telle sorte que nous avons une flèche 
\begin{center}
$\pi_{w}:\overline{\cH}_{w}^{I}\rightarrow\overline{\cH}_{\la}^{par}$.
\end{center}
\begin{prop}\label{va4}
La composée $\pi_{w}\circ \theta_{w}:\Conv_{(s_{a_{i}})}^{n}(\cH)\rightarrow\overline{\cH}_{\la}^{par}$ est une résolution des singularités, où les $(s_{a_{i}})$ sont les réflexions simples associées à $w$.
\end{prop}
\begin{proof}
D'après les lemmes \ref{va3} et \ref{vab}, $\overline{\cH}_{\la}^{par}$, s'identifie à $T\times\overline{\Gr}_{\la}$.
L'assertion vient alors du lemme \ref{va3} et du résultat correspondant dans le cas local.
\end{proof}

$\rmq$ On notera que pour passer à la situation analogue pour l'Iwahori au niveau du champ de Hecke, nous sommes obligés de considérer des $B$-réductions aux deux  torseurs $(E,E')$ et donc nous ajoutons un facteur $(G/B)^{\left|S\right|}$ à $\cH_{\la}$  pour en résoudre ensuite les singularités.
Localement, cela revient à faire l'opération $\phi:[I_{\la}\backslash\grl]\rightarrow[K_{\la}\backslash\grl]$, on considère ensuite la flèche birationnelle $p_{\la}$ et ensuite on résout les singularités à la Demazure-Faltings.
\subsection{Une présentation géométrique de l'espace de Hitchin}
Pour $w=\la\in W_{aff}$, nous avons vu d'après le lemme \ref{grass} que la flèche:
\begin{center}
$p_{\la}:\overline{I_{S}wI_{S}}/I_{S}\rightarrow \overline{\Gr}_{\la}$
\end{center}
est propre surjective et birationnelle au-dessus de $\Gr_{\la}$. 
On peut supposer que pour tout $v\in S$, $l(w_{v})=n$ et on écrit une décomposition en réflexion simples $w=s_{a_{1}}\dots s_{a_{n}}$.
On considère alors le champ $\Conv_{w}^{n}(\cm)$, qui s'insère dans le carré cartésien suivant:
$$\xymatrix{\Conv_{w}^{n}(\cm)\ar[r]\ar[d]_{\pi^{\cm}}&\Conv_{w}^{n}(\cH)\ar[d]\\\cmd\ar[r]&\overline{\cH}_{\la}}$$
Nous avons alors une flèche 
\begin{center}
$\pi^{\cm}:\Conv_{(s_{a_{i}})}^{n}(\cm)\rightarrow\overline{\cm}_{\la}$
\end{center}
et nous notons de la même manière la flèche induite sur l'ouvert $\overline{\cm}_{\la}^{\flat}$ et on pose $\Conv_{(s_{a_{i}})}^{n,\flat}(\cm)$ la restriction à cet ouvert de $\Conv_{(s_{a_{i}})}^{n}(\cm)$. On renvoie à l'énoncé du théorème \ref{transverse} pour la définition de $\cmd^{\flat}$.
\begin{prop}\label{edulcoree}
La flèche
\begin{center}
$\pi^{\cm}:\Conv_{(s_{a_{i}})}^{n,\flat}(\cm)\rightarrow\overline{\cm}_{\la}^{\flat}$
\end{center}
est une présentation géométrique de $\overline{\cm}_{\la}^{\flat}$ et la flèche $\pi^{\cm}$ est même projective. 
\end{prop}
\begin{proof}
Par changement de base et d'après les propositions \ref{va3} et \ref{va4}, on sait déjà que la flèche est projective et qu'elle est lisse sur un ouvert $U$ de 
$\overline{\cm}_{\la}^{\flat}$. Il nous faut donc voir que $IC_{\overline{\cm}_{\la}^{\flat}}$ est un facteur direct de $\pi_{*}^{\cm}\bql$.
Par le théorème \ref{transverse}, nous avons $\Delta^{*}[-d]IC_{\overline{\cH}_{\la}}=IC_{\cmdb}$, comme $\Conv_{w}^{n}(\cH)$ est lisse et résout les singularités de $\overline{\cH}_{\la}$, il résulte alors  du théorème de changement de base propre que $IC_{\overline{\cm}_{\la}^{\flat}}$ est un facteur direct de $\pi_{*}^{\cm}\bql$, ce qu'on voulait.
\end{proof}
Pour terminer ce paragraphe, nous voudrions voir si cette présentation géométrique est équivariante par rapport à notre champ de Picard.
\begin{prop}\label{equiv}
On a une action de $\cP$ sur $\Conv_{(s_{a_{i}})}^{n}(\cm)$ qui rend la flèche $\pi^{\cm}$ $\cP$-équivariante.
\end{prop}
\begin{proof}
Pour simplifier, on suppose que $n=2$.
On commence par donner une description adélique de $\Conv_{(s_{a_{i}})}^{2}(\cm)$, nous obtenons alors:
\begin{center}
$\Conv_{(s_{a_{i}})}^{2}(\cm)=G_{F}^{2}\backslash\{(g_{1},g_{2},\g_{1},\g_{2})\in (G_{\ab}/I_{\ab})^{2}\times (G_{F}^{rs})^{2}\vert~g_{1}^{-1}\g_{1}g_{2}\in I_{\ab}s_{a_{1}}I_{\ab}, g_{2}^{-1}\g_{2}g_{1}\in Is_{a_{2}}I\}$.
\end{center}
où $F$ est le corps de fonctions de la courbe $X$ et $G_{F}^{2}$ agit par:
\begin{center}
$(h_{1},h_{2}).(g_{1},g_{2},\g_{1},\g_{2})=(h_{1}g_{1},h_{2}g_{2},h_{1}\g_{1}h_{2}^{-1},h_{2}\g_{2}h_{1}^{-1})$.
\end{center}
Regardons les automorphismes d'un quadruplet  $(g_{1},g_{2},\g_{1},\g_{2})$, si nous posons $\mu_{1}=g_{1}^{-1}\g_{1}g_{2}$ et $\mu_{2}=g_{2}^{-1}\g_{2}g_{1}$, ils consistent en les paires $(h_{1},h_{2})$ telles que :
\begin{center}
$h_{1}^{-1}\mu_{1}h_{2}=\mu_{1}$ et $h_{2}^{-1}\mu_{2}h_{1}^{-1}=\mu_{2}$,
\end{center}
d'où l'on obtient :
\begin{center}
$h_{1}^{-1}\mu_{1}\mu_{2}h_{1}=\mu_{1}\mu_{2}$ et $h_{2}=\mu_{2}h_{1}^{-1}\mu_{2}^{-1}$.
\end{center}
En particulier, la paire $(h_{1},h_{2})$ est déterminée par $h_{1}\in I_{\g_{1}\g_{2}}$ et donc on obtient bien une action du centralisateur régulier $J_{\chi_{+}(\g_{1}\g_{2})}$.
\end{proof}

\section{Théorème du support}
\subsection{Un énoncé pour une fibration abélienne faible}

Dans ce chapitre, nous démontrons un théorème de support pour une flèche propre $f:M\rightarrow S$. La principale différence avec le résultat démontré par Ngô est que l'espace source n'est pas lisse, mais  singulier. Nous remplaçons alors le faisceau $\overline{\mathbb{Q}}_{l}$ par le complexe d'intersection.

\begin{defi}\label{fibfaible}
Soit $S$ un $k$-schéma de type fini géométriquement irréductible.
Soit une flèche projective $f:M\rightarrow S$ et un schéma en groupes lisse et commutatif $g:P\rightarrow S$ agissant sur $M$.
On dit que $f$ est une fibration abélienne faible si :
\begin{enumerate}
\item
$f$ et $g$ sont de même dimension relative $d$.
\item
L'action de $P$ sur $M$ a des stabilisateurs affines.
\item
Soit $P^{0}$ le sous-schéma en groupes ouvert des composantes neutres des fibres de $P$ au-dessus de $S$. Posons $g^{0}:P^{0}\rightarrow S$. Nous avons le faisceau des modules de Tate 
\begin{center}
$T_{\overline{\mathbb{Q}}_{l}}(P^{0})=H^{2d-1}(g_{!}^{0}\overline{\mathbb{Q}}_{l})(d)$.
\end{center}
Pour tout point géométrique $s$ de $S$, on  a un dévissage canonique de Chevalley de $P^{0}_{s}$
$$\xymatrix{1\ar[r]&R_{s}\ar[r]&P_{s}^{0}\ar[r]&A_{s}\ar[r]&1}$$
avec $A_{s}$ une variété abélienne et $R_{s}$ un groupe algébrique affine commutatif qui induit une décomposition analogue pour les modules de Tate
$$\xymatrix{0\ar[r]&T_{\overline{\mathbb{Q}}_{l}}(R_{s})\ar[r]&T_{\overline{\mathbb{Q}}_{l}}(P_{s}^{0})\ar[r]&T_{\overline{\mathbb{Q}}_{l}}(A_{s})\ar[r]&0}.$$

On dira que $T_{\overline{\mathbb{Q}}_{l}}(P^{0})$ est polarisable si localement pour la topologie étale de $S$, il existe une forme bilinéaire alternée
\begin{center}
$T_{\overline{\mathbb{Q}}_{l}}(P^{0})\times T_{\overline{\mathbb{Q}}_{l}}(P^{0})\rightarrow\overline{\mathbb{Q}}_{l}$
\end{center}
dont la fibre en chaque point géométrique $s$ s'annule sur $T_{\overline{\mathbb{Q}}_{l}}(R_{s})$ et induit un accouplement parfait sur $T_{\overline{\mathbb{Q}}_{l}}(A_{s})$.
\item
Nous avons une présentation géométrique $P$-équivariante $\pi:\tilde{M}\rightarrow M$, par exemple une résolution des singularités $P$-équivariante.
\item
La présentation géométrique $\tilde{M}$ est projective.
\item
On suppose que nous avons une stratification par des schémas lisses équidimensionnels $M_{\mu}$. Les strates sont indexées sur un ensemble $E$
muni d'un ordre partiel avec un plus grand élément $\la$ tel que $M_{\la}$ est un ouvert dense de $M$. On suppose que :
\begin{center}
$M=\coprod\limits_{\mu\leq\la}M_{\mu}$
\end{center}
et pour $\mu\in E$, posons $m_{\mu}=\dim M_{\mu}$.
Soit $\overline{M}_{\mu}$ l'adhérence de la strate $\mu$. Comme  $f$ est propre, elle induit une application propre $f_{\mu}:\overline{M}_{\mu}\rightarrow S_{\mu}$ sur un fermé $S_{\mu}\subset S$. Posons $a_{\mu}=\dime S_{\mu}$. On suppose alors que la dimension relative de $f_{\mu}$ est constante de dimension $d_{\mu}$. En particulier, $d_{\la}=d$.
\item
$\forall ~\mu\leq\la, m_{\la}-m_{\mu}\geq a_{\la}-a_{\mu}$.
\end{enumerate}
\end{defi}
$\rmq$\label{affaibli}\begin{itemize}
\item
Dans Ngô, les hypothèses à partir de (iv) étaient superflues étant donné que la source était lisse.
\end{itemize}

Pour tout point géométrique $s\in S$, on pose $\delta_{s}:=\dime R_{s}$, la dimension de la partie affine de $P_{s}$.
Pour $s\in S$ un point quelconque, le dévissage de Chevalley étant canonique et unique à changement de base radiciel, la fonction  $\delta$ est bien définie. Elle est de plus semi-continue supérieurement d'après la proposition \ref{scsup}.
L'énoncé est le suivant:

\begin{thm}\label{support}
Soit $f:M\rightarrow S$  une fibration abélienne faible. Notons $IC_{M}$ le complexe d'intersection sur $M$.

Soit $K$ un faisceau pervers géométriquement simple présent dans la décomposition d'un faisceau pervers de cohomologie $\mathstrut^{p}H^{n}(f_{*}IC_{M})$ et $Z$ son support. Soit $\delta_{Z}$ la valeur minimale que la fonction $\delta$ attachée à $P$ prend sur $Z$. Alors on a l'inégalité
\begin{center}
$\codim(Z)\leq\delta_{Z}$.
\end{center}
Si l'égalité est atteinte, il existe un ouvert $U$ de $S\otimes_{k}\bar{k}$ tel que $U\cap Z$ est non vide et un système local $\cL$ sur $U\cap Z$ tel que $i_{*}\cL$, $i$ étant l'immersion fermée de $U\cap Z\rightarrow U$, soit un facteur direct de la restriction du faisceau de cohomologie ordinaire de degré maximal $H^{2d}(f_{*}IC_{M})$ à $U$.
\end{thm}
La preuve de ce théorème va reprendre les arguments de Ngô. Nous allons voir que les seules différences qui apparaissent dans la preuve concernent l'argument de dualité de Goresky-McPherson ainsi que la liberté ponctuelle. 
\subsection{Une inégalité de Goresky-McPherson}
Dans cette section, on commence par démontrer une inégalité plus grossière sur la codimension des supports, à savoir qu'elle est plus petite que la dimension relative de la fibration $f$.

\begin{prop}\label{goresky}
Sous les hypothèses ci-dessus et avec les mêmes notations, on a
\begin{center}
$\codim Z\leq d=d_{\la}$.
\end{center}
\end{prop}
\begin{proof}

Soit $Z$ un sous-schéma fermé irréductible de $S\otimes_{k}\bar{k}$. On définit l'ensemble
$\occ(Z)$  des entiers $n$ tel que $Z$ apparaît comme support d'un facteur direct irréductible du faisceau pervers $\mathstrut^{p}H^{n}(f_{*}IC_{M})$.
Par dualité de Poincaré, on a un isomorphisme entre $\mathstrut^{p}H^{-i}(f_{*}IC_{M})$ et $\mathstrut^{p}H^{i}(f_{*}IC_{M})$.
Supposons $\occ(Z)\neq\emptyset$ et soit $n\in\occ(Z)$, alors quitte à changer $n$ en $-n$, on peut supposer $n\geq 0$.

Soit $U$ un ouvert de $S\otimes_{k}\bar{k}$ et $\cL$ un système local irréductible sur $U\cap Z$ tel que $i_{*}L[\dime(Z)]$ soit un facteur direct de la restriction de $\mathstrut^{p}H^{n}(f_{*}IC_{M})$ à $U$. Ainsi, $i_{*}L[\dime(Z)]$ est un facteur direct de $(f_{*}IC_{M})_{\scriptscriptstyle{\vert U}}$. En prenant le faisceau de cohomologie ordinaire, le faisceau $i_{*}L$ est un facteur direct de $H^{n-\dime(Z)}(f_{*}IC_{M})$.
La conclusion vient alors du lemme suivant:
\begin{lem}
Pour tout $k> 2d_{\la}-m_{\la}$, $H^{k}(f_{*}IC_{M})=0$.
\end{lem}
\begin{proof}
Soit $k> 2d_{\la}-m_{\la}$ et $s\in S$ un point géométrique, regardons la fibre du faisceau $H^{k}(f_{*}IC_{M})$.
Sur la fibre $M_{s}$, on a une stratification naturelle induite par celle de $M$,
\begin{center}
$M_{s}=\coprod\limits_{\mu\leq\la}M_{s}^{\mu}$.
\end{center}
Par un argument de suite spectrale standard, il suffit de démontrer que
\begin{center}
$\forall~ \mu\leq\la, H^{k}(M_{s}^{\mu},IC_{M})=0$,
\end{center}
où nous avons noté de la même manière le complexe d'intersection restreint à $M_{s}^{\mu}$.
Commençons par le cas $\mu=\la$. En ce cas, la restriction de $IC_{M}$ à $M_{s}^{\la}$ est égale à $\overline{\mathbb{Q}}_{l}[m_{\la}]$ comme $M_{\la}$ est lisse.
Ainsi, on obtient
\begin{center}
$H^{k}(M_{s}^{\la},IC_{M})=H^{k}(M_{s}^{\la},\overline{\mathbb{Q}}_{l}[m_{\la}])=H^{k+m_{\la}}(M_{s}^{\la},\overline{\mathbb{Q}}_{l})=0$.
\end{center}
Passons au cas $\mu\leq\la$. Alors $(IC_{M})_{\scriptscriptstyle{\vert M_{s}^{\mu}}}$ est concentré en degrés cohomologiques $[-m_{\la},-m_{\mu}]$. Et donc la propriété d'annulation revient à prouver que:
\begin{center}
$k+m_{\mu}>2d_{\mu}$
\end{center}
pour $k> 2d_{\la}-m_{\la}$. Le calcul nous donne alors
\begin{center}
$2(d_{\la}-d_{\mu})-(m_{\la}-m_{\mu})=(d_{\la}-d_{\mu})-(a_{\la}-a_{\mu})\leq 0$
\end{center}
où nous avons utilisé le fait que $m_{\mu}=d_{\mu}+a_{\mu}$ et de même pour $m_{\la}$ et la dernière inégalité est donnée par l'hypothèse de fibration abélienne faible, ce qui conclut.
\end{proof}

Montrons comment on en déduit la proposition; nous obtenons:
\begin{center}
$n-\dime(Z)\leq 2d_{\la}-m_{\la}$,
\end{center}
et des égalités $\dime Z +\codim Z=a_{\la}$ et $m_{\la}=d_{\la}+a_{\la}$, on obtient
\begin{center}
$\codim(Z)\leq d_{\la}$.
\end{center}
\end{proof}

Il va nous falloir maintenant améliorer cette inégalité.
On définit alors
\begin{center}
$\amp(Z)=\max(\occ(Z))-\mi(\occ(Z))$.
\end{center}

\begin{prop}
Sous les hypothèses de $\ref{support}$, $\occ(Z)\neq\emptyset$ et on a l'inégalité
\begin{center}
$\amp(Z)\geq2(d-\delta_{Z})$.
\end{center}
\end{prop}
Voyons comment on déduit le théorème \ref{support} grâce à cette proposition.
L'ensemble $\occ(Z)$ étant symétrique par rapport à zéro, cette contrainte impose l'existence d'un entier $n\geq d-\delta_{Z}$ dans $\occ(Z)$, il ne nous reste donc plus qu'à reprendre la fin de l'argument \ref{goresky} pour obtenir l'inégalité du théorème \ref{support}.
\subsection{Action par cap produit et liberté}
Ce chapitre suit \cite[sect. 7.4-7.5]{N}.
Considérons donc notre schéma en groupes lisse et commutatif $g:P\rightarrow S$ de dimension relative $d$ à fibres connexes. Posons le complexe suivant concentré en degrés négatifs :
\begin{center}
$\La_{P}=g_{!}\overline{\mathbb{Q}}_{l}[2d](d)$.
\end{center}
Nous avons $H^{0}(\La_{P})=\bql$ et $H^{-1}(\La_{P})=T_{\bql}(P)$. Plus généralement, nous avons la formule suivante pour les degrés supérieurs 
\begin{center}
$H^{-i}(\La_{P})=\bigwedge^{i}T_{\bql}(P)$.
\end{center}

Soit un morphisme  de type fini $f:M\rightarrow S$ sur lequel $P$ agit, on a un morphisme trace:
\begin{center}
$\act_{!}\act^{!}IC_{M}\rightarrow IC_{M}$
\end{center}
où $\act:P\times_{S}M\rightarrow M$.
Comme $\act$ est lisse, $\act^{!}=\act^{*}[2d](d)$ et $\act^{*}[d]IC_{M}=IC_{P\times_{S}M}$. En poussant cette flèche par $f_{!}$, on obtient
\begin{center}
$(g\times_{S}f)_{!}IC_{P\times_{S}M}[d](d)\rightarrow f_{!}IC_{M}$
\end{center}
et en utilisant l'isomorphisme de Künneth, on obtient un morphisme de cap-produit:
\begin{center}
$\La_{P}\otimes f_{!}IC_{M}\rightarrow f_{!}IC_{M}$
\end{center}
et en faisant $f=g$, on obtient un morphisme de complexes:
\begin{center}
$\La_{P}\otimes\La_{P}\rightarrow \La_{P}$
\end{center}
d'où une structure d'algèbre commutative graduée sur $\La_{P}$.
En multipliant par un entier $N$ dans $P$, cela induit la multiplication par $N$ sur le module de Tate et par $N^{i}$ sur $H^{-i}(\La_{P})=\bigwedge^{i}T_{\bql}(P)$. 
En utilisant la technique de Lieberman \cite[2A11]{Lieb}, on peut trouver des projecteurs appropriés de telle sorte que:
\begin{center}
$\La_{P}=\bigoplus\limits_{i\leq0}\bigwedge^{i}T_{\bql}(P)[i]$
\end{center}
de manière compatible à la multiplication.

L'ensemble des supports intervenant dans la décomposition en faisceaux pervers irréductibles est fini, notons-le $\mathfrak{U}$. Pour tout $\alpha\in\mathfrak{U}$, notons $Z_{\alpha}$ le support correspondant.
Pour tout $n$, on a une décomposition canonique
\begin{center}
$\mathstrut^pH^{n}(f_{*}IC_{M})=\bigoplus\limits_{\alpha\in\mathfrak{U}}K_{\alpha}^{n}$
\end{center}
où $K_{\alpha}^{n}$ est la somme directe des faisceaux pervers irréductibles de $\mathstrut^{p}H^{n}(f_{*}IC_{M})$ de support $Z_{\alpha}$.
On pose alors
\begin{center}
$K_{\alpha}:=\bigoplus\limits_{n\in\mathbb{Z}}K_{\alpha}^{n}$
\end{center}
que l'on suppose non nul pour tout $\alpha\in\mathfrak{U}$.
Pour $\alpha\in\mathfrak{U}$, soit $V_{\alpha}$ un ouvert de $Z_{\alpha}$ sur lequel la restriction de $\mathcal{K}_{\alpha}^{n}$ est de la forme $\mathcal{K}_{\alpha}[\dim V_{\alpha}]$ où $\mathcal{K}_{\alpha}^{n}$ est un système local pur de poids $n$ sur $V_{\alpha}$.

Quitte à rétrécir $V_{\alpha}$, il existe un changement de base fini radiciel $V_{\alpha}'\rightarrow V_{\alpha}$ tel que le schéma en groupes $P_{\vert V_{\alpha}'}$ admet un dévissage
$$\xymatrix{1\ar[r]&R_{\alpha}\ar[r]&P_{\vert V_{\alpha}'}\ar[r]&A_{\alpha}\ar[r]&1}$$
où $A_{\alpha}$ est un schéma abélien et $R_{\alpha}$ un schéma en groupes affine lisse à fibres connexes, tous deux définis sur $V_{\alpha}'$.
Comme $V_{\alpha}'\rightarrow V_{\alpha}$ est un homéomorphisme, on peut considérer que les modules de Tate $T_{\overline{\mathbb{Q}}_{l}}(A_{\alpha})$ et $T_{\overline{\mathbb{Q}}_{l}}(R_{\alpha})$, ainsi que les suites exactes qui les relient sont définis sur $V_{\alpha}$.
Enfin, quitte à rétrécir à nouveau $V_{\alpha}'$, on suppose que $V_{\alpha}\cap Z_{\alpha'}=\emptyset$ sauf si $Z_{\alpha'}\subset Z_{\alpha}$.

Ngô définit alors \cite[sect. 7.4.2]{N} une structure de $\La_{A_{\alpha}}$-module gradué sur $\mathcal{K}_{\alpha}=\bigoplus\limits_{n}\mathcal{K}_{\alpha}^{n}$.
Nous avons de même que dans Ngô la proposition cruciale suivante:
\begin{prop}
Sous les hypothèses de $\ref{support}$, pour tout point $u_{\alpha}\in V_{\alpha}$, la fibre $\mathcal{K}_{\alpha,u_{\alpha}}$ est un module gradué libre sur l'algèbre graduée $\La_{A_{\alpha},u_{\alpha}}$.
\end{prop}

Nous pouvons formuler l'énoncé de liberté sans faire intervenir de points bases.
\begin{lem}\cite[7.4.11]{N}
Soit $U$ un $\bar{k}$-schéma connexe. Soit $\La$ un système local gradué en un nombre fini de degrés négatifs avec $\La^{0}=\overline{\mathbb{Q}}_{l}$ muni d'une structure d'algèbre graduée $\La\otimes\La\rightarrow\La$. Soit $L$  un système local gradué muni d'une structure de module gradué
\begin{center}
$\La\otimes L\rightarrow L$.
\end{center}
Supposons qu'il existe un point géométrique $u$ de $U$ tel que $L_{u}$ soit un module libre sur la fibre  $\La_{u}$ de $\La$ en $u$. Supposons que $L$ soit semisimple comme système local gradué. Alors il existe un système local gradué $E$ sur $U$ de telle sorte que
\begin{center}
$L=\La\otimes E$
\end{center}
compatible aux structures de $\La$-modules.
\end{lem}

\subsection{Liberté ponctuelle}

\begin{prop}
Soit $s\in S$ un point géométrique, alors d'après les hypothèses de fibration abélienne faible, la partie abélienne de $A_{s}$ de $P_{s}$ agit avec des stabilisateurs finis sur $M_{s}$ et
\begin{center}
$\bigoplus H^{n}(M_{s},IC_{M})[-n]$
\end{center}
est un module libre sur $\La_{A_{s}}=\bigoplus\limits_{i}\bigwedge^{i}T_{\bql}(A_{s})[i]$.
\end{prop}
\begin{proof}
Nous avons une présentation géométrique $P$-équivariante $\pi:\tilde{M}\rightarrow M$.
En particulier, on obtient que $IC_{M}$ est un facteur direct de $\pi_{*}\bql[m_{\la}]$. Donc pour montrer la liberté, il suffit de démontrer la liberté pour le faisceau $\pi_{*}\bql[m_{\la}]$. Ensuite, comme $\La_{A_{s}}$ est une algèbre locale, tout module projectif sera libre, ce qui nous donnera le résultat.

L'action de $A_{s}$ sur $M_{s}$ est à stabilisateurs affines, il en est donc de même de l'action de $A_{s}$ sur $\tilde{M}_{s}$. En particulier, le quotient $N_{s}:=[\tilde{M}_{s}/A_{s}]$ est un champ propre à inertie finie.
Comme $\tilde{M}_{s}$ est projectif, on a une flèche $\theta:\tilde{M}_{s}\rightarrow N_{s}$ projectif lisse. D'après Deligne \cite{De2}, on a une décomposition:
\begin{center}
$\theta_{*}\bql[m_{\la}]\cong\bigoplus\limits_{i}R^{i}m_{*}\bql[m_{\la}-i]$
\end{center}
qui induit la dégénérescence de la suite spectrale de Leray 
\begin{center}
$H^{p}(N_{s},R^{q}m_{*}\bql[m_{\la}])\Rightarrow H^{p+q}(\tilde{M}_{s},\bql[m_{\la}])$
\end{center}
en $E_{2}$.

De plus, $R^{q}m_{*}\bql$ est un système local sur $N_{s}$ et le changement de base propre nous dit que sa fibre s'identifie à $H^{i}(A_{s})$, donc en fait le faisceau est constant sur $N_{s}$ de valeur $H^{i}(A_{s})$.
On en déduit donc une filtration $\La_{A_{s}}$-stable sur $\bigoplus H^{n}(N_{s},\bql[m_{\la}])$ dont le $j$-ième gradué est 
\begin{center}
$\bigoplus\limits_{i}H^{j}(N_{s},R^{i}m_{*}\bql[m_{\la}])=H^{j}(N_{s})\otimes \bigoplus H^{i+m_{\la}}(A_{s})$.
\end{center}
les gradués étant libres, il en est de même de la somme directe $\bigoplus H^{n}(N_{s},\bql[m_{\la}])[-n]$ et donc de $\bigoplus H^{n}(\tilde{M}_{s})[-n]$, ce qu'on voulait.
\end{proof}

Maintenant la fin de la preuve du théorème \ref{support}  est strictement analogue à  \cite[sect. 7.5-7.6]{N}.

\subsection{La $\delta$-régularité et la $\kappa$-variante}
On dit que le  $S$-schéma en groupes commutatif lisse $P$ est \textit{$\delta$-régulier} si pour tout $\delta\in\mathbb{N}$, on a 
\begin{center}
$\codim_{S}S_{\delta}\geq\delta$,
\end{center}
où $S_{\delta}$ désigne la strate à $\delta$-constant. A la suite de Ngô, nous faisons la définition suivante:
\begin{defi}
Une fibration abélienne faible, dont la composante $P$ est $\delta$-régulière sera appelée une \textit{fibration abélienne $\delta$-régulière.}
\end{defi}
La $\delta$-régularité est stable par changement de base. Nous en avons une caractérisation équivalente due à Ngô \cite[7.1.6]{N} ; soit $Z$ un fermé irréductible de $S$. Soit $\delta_{Z}$ la valeur minimale de $\delta$ sur $Z$. Alors, $P$ est $\delta$-régulier si et seulement si pour tout sous-schéma fermé irréductible $Z$ de $S$, on a $\codim(Z)\geq\delta_{Z}$.

Il résulte alors du théorème \ref{support} que nous obtenons le corollaire suivant
\begin{cor}
Soit $f:M\rightarrow S$ une fibration abélienne $\delta$-régulière. Soit $Z$ le support d'un constituant pervers irréductible de $f_{*}IC_{M}$, alors il existe un ouvert non vide $U$ de $S\otimes_{k}\bar{k}$ tel que $U\cap Z$ est non vide et un système local non trivial $L$ sur $U$ tel que $i_{*}L$, avec $i:U\cap Z\rightarrow U$, soit un facteur direct de la restriction du faisceau de cohomologie ordinaire de degré maximal $H^{2d}(f_{*}IC_{M})$ à $U$.
\end{cor}
\begin{proof}
Cela résulte du cas d'égalité du théorème \ref{support} par $\delta$-régularité.
\end{proof}
Nous avons également avoir besoin d'une version faisant intervenir $\pi_{0}(P)$. Supposons que ce groupe est fini abélien.

Le schéma $P$ agit alors sur $\mathstrut^{p}H^{n}(f_{*}IC_{M})$ via $\pi_{0}(P)$, qui est fini abélien. Pour tout caractère $\kappa:\pi_{0}(P)\rightarrow\overline{\mathbb{Q}}_{l}^{*}$, on considère le plus grand facteur direct $\mathstrut^{p}H^{n}(f_{*}IC_{M})_{\kappa}$ de $\mathstrut^{p}H^{n}(f_{*}IC_{M})$ où $\pi_{0}(P)$ agit par $\kappa$.
D'après Laumon-Ngô \cite[3.2.5]{LN}, on peut montrer qu'il existe un entier $N$ et une décomposition de $f_{*}IC_{M}$
\begin{center}
$f_{*}IC_{M}=\bigoplus\limits_{\kappa\in\pi_{0}(P)^{*}}(f_{*}IC_{M})_{\kappa}$.
\end{center}
tels que pour tout $\alpha\in\pi_{0}(P)$, $(\alpha-\kappa(\alpha)\Id)^{N}$ est nul sur $(f_{*}IC_{M})_{\kappa}$. En remplaçant $f_{*}IC_{M}$ par $(f_{*}IC_{M})_{\kappa}$ dans le théorème $\ref{support}$, nous obtenons
\begin{prop}\label{supportkappa}
On peut remplacer dans $\ref{support}$, $\mathstrut^{p}H^{n}(f_{*}IC_{M})$ par $\mathstrut^{p}H^{n}(f_{*}IC_{M})_{\kappa}$ et $H^{2d}(f_{*}IC_{M})$ par $H^{2d}(f_{*}IC_{M})_{\kappa}$.
\end{prop} 

\subsection{Preuve du théorème \ref{detsupport}}\label{final}
Nous allons maintenant appliquer le théorème du support au morphisme $\tilde{f}^{ani,\flat}:\tcmdanf\rightarrow\tabdanf$  muni d'une action $\tilde{g}^{ani,\flat}:\cPinf\rightarrow\tabdanf$. On pose dans la suite  $S:=\tabdanf$, $S^{bon}=\tabd^{bon}\cap\tabdanf$, $M=\tcmdanf$,  $P=\cPinf$, $f=\tilde{f}^{ani,\flat}$ et $g=\tilde{g}^{ani,\flat}$.
Le fait qu'on ait des champs de Deligne-Mumford (cf. Prop. \ref{corlisse}) ne pose pas de problèmes car la preuve s'étend telle quelle, comme on sait qu'une fibre de Hitchin est une variété projective d'après la proposition \ref{stabaffine}.
La flèche $g$ est bien lisse d'après  \ref{picardlisse} et $P$ agit avec des stabilisateurs affines d'après la proposition \ref{stabaffine}.
D'après la proposition\ref{equidim}, la flèche $f$ est plate et les inégalités sur les dimensions résultent de \ref{dimhitchin2}.
Enfin, d'après la proposition \ref{tate} le module de Tate est bien polarisable

D'après les propositions \ref{equiv} et \ref{edulcoree}, $M$ admet une présentation géométrique. Elle n'est a priori pas projective, mais en fait on a seulement besoin que les fibres au-dessus de $\abdanf$ soient projectives et cela vient du fait que la résolution est projective et que la fibre de Hitchin est homéomorphe à un schéma projectif.

Au-dessus de  $S^{bon}$, on a une fibration abélienne $\delta$-régulière par définition.
On commence par démontrer le théorème \ref{detsupport} dans le cas où $\kappa=1$:
\begin{prop}
Soit $\mathstrut^{p}H^{n}(f_{*}IC_{M})_{st}$ le plus grand facteur direct de $\mathstrut^{p}H^{n}(f_{*}IC_{M})$ sur lequel $X_{*}(T)$ agit trivialement. Soit $Z$ support  d'un faisceau pervers géométriquement irréductible $K$ qui est un facteur direct de $\mathstrut^pH^{n}(f_{*}IC_{M})$. Si $Z\cap S^{bon}\neq\emptyset$, alors $Z=S$.
\end{prop}
\begin{proof}
La preuve est la même que celle de Ngô \cite[7.8.3]{N}.
Comme $P$ est $\delta$-régulier au-dessus de $S^{bon}$, il résulte du théorème du support que $\codim(Z)=\delta_{Z}$. 
En particulier, par le théorème \ref{support}, il existe un ouvert $U$ de $S$ et un système local $L$ sur $U\cap Z$ tel que $i_{*}L$, avec $i:U\cap Z\rightarrow U$, soit un facteur direct de la restriction de $H^{2d}(f_{*}IC_{M})$ à $U$, dont on sait qu'il est constant d'après la proposition \ref{stabmax}, d'où $U\cap Z=U$ et donc $Z=S$.
\end{proof}
Passons à la preuve du théorème \ref{detsupport}:
\begin{proof}
Soit $Z\in\supp(f_{*}(IC_{M}))_{\kappa}$. D'après la proposition \ref{kappa}, on sait que $Z\subset\tabdeth$.
On rappelle d'après la définition \ref{changement} que nous avons:
\begin{center}
$\tabdeth=\bigcup\limits_{\mu\in\soc_{H}(\la)}\mathcal{A}_{\mu,H}$.
\end{center}
Soit alors $\mu\in\soc_{H}(\la)$ tel que $Z\cap\tilde{\mathcal{A}}_{\mu,H}^{bon}\neq\emptyset$, alors
\begin{center}
$\codim(Z)\geq\delta_{H,Z}$.
\end{center}
Pour tout $a_{H}\in\tilde{\mathcal{A}}_{\mu,H}\subset\tabdeth$, la différence
\begin{center}
$\delta(a_{H})-\delta_{H}(a_{H})$
\end{center}
est indépendante de $a_{H}$ et vaut la codimension de $\tilde{\mathcal{A}}_{\mu,H}$ dans $\tabd$ d'après le corollaire \ref{indep}.
On obtient alors
\begin{center}
$\codim(Z)\geq\delta(a_{H})+\codim(\tilde{\mathcal{A}}_{\mu,H})=\delta_{Z}$
\end{center}
On applique alors à nouveau le cas d'égalité du théorème $\ref{support}$, pour obtenir de la même manière que dans le théorème précédent que $Z=\tabdeth$.
\end{proof}

\subsection{Extension du théorème du support}
Jusqu'à maintenant pour le théorème du support, nous avons supposé que $G_{der}$ était simplement connexe. Nous allons maintenant expliquer les modifications nécessaires pour obtenir le cas général. On suppose $G$ semisimple, pour simplifier, le cas réductif est analogue en considérant un groupe $m:G'\rightarrow G$ tel que $G'$ vérifie $G'_{der}$ simplement connexe et $m$ une isogénie.

Dans le cas où $G$ est semisimple la fibration $\cmdan\rightarrow\abdan$ est incommode et il convient de la remplacer par la fibration $\cmdaan\rightarrow\abdaan$.
On rappelle que dans la section \ref{hitaug}, nous avons obtenu un diagramme commutatif:
$$\xymatrix{\cmdaan\ar[r]^{p}\ar[d]_{f^{aug,ani}}&\cmdan\ar[d]^{f^{ani}}\\\abdaan\ar[r]^{p}&\abdan}$$
où les flèches horizontales sont finies, surjectives, génériquement étales et les flèches verticales propres.
En particulier, on a :
\begin{center}
$p_{*}IC_{\cmdaan}=IC_{\cmdan}$.
\end{center}
On dispose de l'action d'un champ de Picard $\cPaan$ sur $\cmdaan$ qui induit une action par le lemme d'homotopie de $\pi_{0}(\cP^{aug})$ sur $f_{*}^{aug,ani}IC_{\cmdaan}$, en particulier, on peut définir pour chaque $\kappa\in X_{*}(T)$, on peut définir un facteur
$(f_{*}^{aug,ani}IC_{\cmdaan})_{\kappa}$. 
On définit alors  pour $\kappa\in X_{*}(T)$, le facteur direct $(f_{*}^{ani}IC_{\cmdan})_{\kappa}$ de $f_{*}^{ani}IC_{\cmdan}$ par:
\begin{center}
$(f_{*}^{ani}IC_{\cmdaan})_{\kappa}=p_{*}(f_{*}^{aug,ani}IC_{\cmdaan})_{\kappa}$. 
\end{center}
On applique alors de la même manière que dans le cas simplement connexe le théorème \ref{support} à $f^{aug,ani}:\cmdaan\rightarrow\abdaan$, que l'on projette ensuite par la flèche $p$ pour obtenir le théorème \ref{detsupport} dans le cas semisimple.
\subsection{Stabilisation géométrique}\label{stabI}
On note toujours $\tilde{\nu}$ l'application composée:
\begin{center}
$\tilde{\nu}:\tilde{\mathcal{A}}_{\eta^{*}\la,H}\rightarrow\tabd$
\end{center}
On note 
\begin{center}
$\tilde{f}_{H}:\overline{\cm}^{ani,\flat}_{\eta^{*}\la,H,\infty}\rightarrow\tabdeth$
\end{center}
la fibration de Hitchin pour $H$. Nous voulons identifier en fonction des groupes endoscopiques, la $\kappa$-composante du complexe $(\tilde{f}^{ani}_{*}IC)$. On a une flèche canonique $\eta:\hat{H}\rightarrow\hat{G}$ (comme l'endoscopie est déployée) d'où l'on déduit un foncteur de restriction:
\begin{center}
$\eta^{*}:\Rep(\hat{G})\rightarrow\Rep(\hat{H})$.
\end{center}
En particulier, par l'isomorphisme de Satake, on obtient un morphisme entre les algèbres de Hecke sphériques:
\begin{center}
$b:\cH_{G}\rightarrow\cH_{H}$.
\end{center}
Pour $\la\in X_{*}(T)^{+}$, soit $V_{\la}$ la représentation irréductible de plus haut poids $\la$, nous avons donc l'égalité:
\begin{center}
$b(f_{\la})=f_{\la}^{H}+\sum\limits_{\mu\in\soc_{H}(\la)}m_{\mu}f_{\mu}^{H}\in\cH_{H}$,
\end{center}
où $\soc_{H}(\la)$ désigne les cocaractères dominants de $H$ qui interviennent dans la décomposition de la représentation $\eta^{*}V_{\la}$.
On regarde alors la strate de la grassmannienne affine:
\begin{center}
$\overline{\Gr}_{\eta^{*}\la}:=\coprod\limits_{\mu\in\soc_{H}(\la)}\overline{\Gr}_{\mu}$
\end{center}
et le faisceau pervers pur $S_{\eta^{*}\la,H}:=\bigoplus\limits_{\mu\in\soc_{H}(\la)}IC_{\overline{\Gr}_{\mu}}^{\oplus m_{\mu}}$ dessus.
Maintenant, il résulte de la proposition \ref{bemolisse} que l'on a une flèche lisse:
\begin{center}
$\theta^{\flat}:\cmdb\rightarrow[\overline{\Gr}'_{\la}/G'_{N}]$
\end{center}
avec $G_{N}':=\Res_{N-n_{t}[t]/k}G$ où $t$ est le point auxiliaire. On tire donc ce faisceau par $\theta^{\flat}$ pour obtenir un faisceau pervers pur sur $\overline{\cm}_{\eta^{*}\la,H}$ noté de la même manière. Le théorème de transversalité nous dit que 
\begin{center}
$S_{\eta^{*}\la,H}:=\prod\limits_{\mu\in\soc_{H}(\la)}IC_{\overline{\mathcal{\cm}}^{\flat}_{H,\mu}}^{\oplus m_{\mu}}$.
\end{center}
\begin{thm}\label{stab}
On suppose $G_{der}$ simplement connexe déployé dont l'endoscopie est déployée. Il existe un isomorphisme défini sur $k$ entre les faisceaux pervers gradués:
\begin{center}
$\bigoplus\limits_{n}\mathstrut^{p}H^{n}(\tilde{f}^{ani}_{*}IC)_{\kappa}$
\end{center}
et
\begin{center}
$\bigoplus\limits_{n}\mathstrut^{p}H^{n}(\tilde{f}_{H,*}^{ani}S_{\eta^{*}\la,H})_{st}$.
\end{center}
\end{thm}

$\rmq$ Quitte à changer $X$ en  $X\otimes_{k}k'$, pour une extension finie $k'$ de $k$, et étant donné qu'une donnée endoscopique $(\kappa,\rho_{\kappa})$ sur $\bar{X}$ est définie sur une extension finie de $k$, on a le résultat analogue sur $\bar{k}$.
Nous donnerons la preuve de cet énoncé dans la section \ref{final1}.
\section{Comptage de points}
Sauf mention explicite, on suppose que  $G_{der}$ est simplement connexe et que son endoscopie est déployée.
\subsection{Comptage dans une fibre de Springer affine}
Comme nous sommes dans une situation locale, on suppose de plus que $G$ est simplement connexe.
Soit $v\in\left|X\right|$ un point fermé de $X$, on considère $F_{v}$ la complétion $v$-adique du corps de fonctions de $X$, d'anneau d'entiers $\co_{v}$, $k$ le corps résiduel de cardinal $q_{v}:=q^{\deg(v)}$. On désigne par $\bar{F}_{v}$ la complétion $v$-adique de $\bar{F}$, d'anneau d'entiers $\bar{\co}_{v}$ et on se donne une uniformisante $\pi_{v}$ (resp. $\pi_{\bar{v}}$) de $\co_{v}$ (resp. $\cobv$).
Soit $\la\in X_{*}(T)^{+}$ un cocaractère dominant.
Soit $a\in\kc^{\la,\heartsuit}(\co_{v})$, un point entier de fibre générique régulière semisimple. La fibre de Springer affine réduite $\overline{\mathcal{M}}^{red}_{\la,v}(a)$ est un $k$-schéma localement de type fini dont l'ensemble des $\bar{k}$-points est
\begin{center}
$\overline{\mathcal{M}}^{red}_{\la,v}(a):=\{g\in G(\bfv)/G(\cobv)\vert \ad(g)^{-1}\g_{0}\in \overline{G(\cobv)\pi_{\bv}^{\la}G(\cobv)}\}$
\end{center}
avec $\g_{0}=\eps_{+}(a)$.
Les nilpotents étant  inutiles pour le comptage, nous supprimons l'exposant \og réd\fg~ dans $\overline{\mathcal{M}}^{red}_{\la,v}(a)$, ainsi que pour tous les  autres schémas où ce sera nécessaire.

Posons $J_{a}=a^{*}J$; par la suite nous aurons à considérer un schéma en groupes $J_{a}'$ muni d'un morphisme $J_{a}'\rightarrow J_{a}$ qui est un isomorphisme sur la fibre générique, typiquement $J_{a}^{0}$, mais pas seulement.
Nous avons le $k$-schéma en groupes localement de type fini
\begin{center}
$\cP_{v}(J'_{a})=J_{a}(\fv)/J_{a}'(\cov)$.
\end{center}
La flèche $J_{a}'\rightarrow J_{a}$ permet de définir une action de $\cP_{v}(J'_{a})$ sur $\cma$ qui vérifie bien les hypothèses \cite[Hyp.8.2.1]{N}, nécessaires pour compter.
On commence par un petit lemme
\begin{lem}
Si la fibre spéciale de $J_{a}'$ est connexe, on a un isomorphisme
\begin{center}
$H^{1}(F_{v},J_{a})=H^{1}(k,\cP_{v}(J_{a}'))$.
\end{center}
\end{lem}
\begin{proof}
C'est une application du théorème de Steinberg sur l'annulation de $H^{1}(\bfv,J_{a})$ et de celui de Lang sur l'annulation de $H^{1}(k,J_{a}'(\co_{v}))$.
\end{proof}
On a une application;
\begin{center}
$\cma\rightarrow [G(\cov)\backslash\overline{\Gr}_{\la,v}]$
\end{center}
donnée par $g\mapsto g^{-1}\g g$. On considère alors  le faisceau $K$ sur la fibre de Springer affine qui s'obtient en tirant $IC_{\overline{\Gr}_{\la,v}}$.
On considère la fonction de Kazhdan-Lusztig:
\begin{center}
$f_{\la,v}(x):=\Tr(Fr_{x}, IC_{\overline{\Gr}_{\la},x})$.
\end{center}
On rappelle la formule suivante qui relie la fonction de Kazhdan-Lusztig à l'autre base de l'algèbre de Hecke sphérique \cite{Lusz3}:
\begin{equation}
f_{\la,v}=(-1)^{\left\langle 2\rho,\la\right\rangle}q_{v}^{-\left\langle \rho,\la\right\rangle}(c_{\la,v}+\sum\limits_{0\leq\mu<\la}a_{\la\mu}c_{\mu,v})
\label{luz3}
\end{equation}
où $c_{\la,v}$ sont les fonctions caractéristiques des strates $\Gr_{\la,v}$.

\begin{prop}\label{connexe}
Supposons la fibre spéciale de $J_{a}'$ connexe.
Soit $\kappa:H^{1}(F_{v},J_{a})\rightarrow\bar{\mathbb{Q}}_{l}^{*}$, on a la formule suivante : 
\begin{center}
$\sum\limits_{x\in[\cma/\cP_{v}(J'_{a})](k)}\Tr(\Fr_{\bar{x}}, K_{\bar{x}})_{\kappa}=\vol(J_{a}'(\co_{v}),dt_{v})\bo^{\kappa}_{a}(f_{\la,v},dt_{v})$
\end{center}
où $f_{\la,v}$ est la fonction associée à $\la$ dans la base de Kazhdan-Lusztig de l'algèbre de Hecke sphérique.
\end{prop}
\begin{proof}
On reprend la preuve de Ngô, la seule différence notable vient du faisceau.

Un point $x\in[\cma/\cP_{v}(J'_{a})](k)$ consiste en un couple $(m,p)$ formé d'un élément $m\in\cma$ et d'un élément $p\in\cP_{v}(J'_{a})$ tels qu'on ait $p\sigma(m)=m$. Une flèche entre $(m,p)$ et $(m',p')$ dans cette catégorie consiste en un seul élément $h\in\cP_{v}(J'_{a})$ tel que $m'=hm$ et $p'=hp\sigma(h)^{-1}$.
Pour chaque objet $x=(m,p)$ la $\sigma$-classe de conjugaison, définit un élément 
\begin{center}
$\cl(x)\in H^{1}(k,\cP_{v}(J'_{a}))=H^{1}(F_{v},J_{a})$.
\end{center}
Soit $\g_{0}$ l'image de $a$ dans $V_{G}$ par la section de Steinberg, on a un isomorphisme canonique entre $J_{a}$ et $I_{\g_{0}}$ et donc d'après \cite[Lem. 8.2.6]{N}, la classe $\cl(x)$ appartient en fait à
\begin{center}
$\Ker(H^{1}(F_{v},I_{\g_{0}})\rightarrow H^{1}(F_{v},G))$.
\end{center}
Pour $\xi\in \Ker(H^{1}(F_{v},I_{\g_{0}})\rightarrow H^{1}(F_{v},G))$, on considère la sous-catégorie $[\cma/\cP(J'_{a})]_{\xi}(k)$ des objets de $[\cma/\cP(J'_{a})](k)$ tels que $\cl(x)=\xi$.
Fixons $j_{\xi}\in I_{\g_{0}}$, la catégorie $[\cma/\cP(J'_{a})]_{\xi}(k)$ est alors équivalente à la catégorie $O_{\xi}$ (cf.\cite[p. 145]{N} pour les détails) dont les objets sont les éléments $g\in G(\bfv)/G(\co_{v})$:
\begin{itemize}
\item
$g^{-1}j_{\xi}\sigma(g)=1$
\item
$\ad(g)^{-1}\g_{0}\in \overline{G(\cobv)\pi_{\bv}^{\la}G(\cobv)}$
\end{itemize}
et dont les flèches $g\rightarrow g_{1}$ sont les éléments $h\in J_{a}(F_{v})/J'_{a}(\co_{v})$ tels que $g=hg_{1}$, où on a utilisé l'isomorphisme entre $J_{a}$ et $I_{\g_{0}}$ pour faire agir $h$ à gauche.
Comme $\xi$ est d'image triviale dans $H^{1}(F_{v},G)$, il existe $g_{\xi}\in G(F_{v})$ tel que 
$g_{\xi}^{-1}j_{\xi}\sigma(g_{\xi})=1$. Fixons un tel $g_{\xi}$ et en posant 
\begin{center}
$\g_{\xi}=\ad(g_{\xi})^{-1}\g_{0}$,
\end{center}
on peut réinterpréter la catégorie $O_{\xi}$ comme la catégorie dont les objets sont les éléments $g\in G(F_{v})/G(\co_{v})$ vérifiant:
\begin{center}
$\ad(g)^{-1}\g_{\xi}\in \overline{G(\co_{v})\pi_{v}^{\la}G(\co_{v})}$
\end{center}
toujours avec les mêmes flèches.
L'ensemble des classes d'isomorphismes de $O_{\xi}$ est l'ensemble des doubles classes
\begin{center}
$g\in I_{\g_{\xi}}(F_{v})\backslash G(F_{v})/G(\co_{v})$
\end{center}
telles que $\ad(g)^{-1}\g_{\xi}\in \overline{G(\co_{v})\pi_{v}^{\la}G(\co_{v})}$.
Le groupe des automorphismes est le groupe
\begin{center}
$(I_{\g_{\xi}}(F_{v})\cap g G(\co_{v})g^{-1})/J_{a}'(\co_{v})$
\end{center}
dont le cardinal est donné par le quotient de volumes
\begin{center}
$\displaystyle{\frac{\vol(I_{\g_{\xi}}(F_{v})\cap gG(\co_{v})g^{-1},dt_{v})}{\vol(J_{a}'(\co_{v}),dt_{v})}}$
\end{center}
On obtient donc 
\begin{center}
$\displaystyle{\sharp O_{\xi}=\sum\frac{\vol(I_{\g_{\xi}}(F_{v})\cap g G(\co_{v})g^{-1},dt_{v})}{\vol(J_{a}'(\co_{v})),dt_{v}}}$
\end{center}
où l'on somme sur l'ensemble des doubles classes 
\begin{center}
$g\in I_{\g_{\xi}}(F_{v})\backslash G(F_{v})/G(\co_{v})$
\end{center}
telles que $\ad(g)^{-1}\g_{\xi}\in \overline{G(\co_{v})\pi_{v}^{\la}G(\co_{v})}$.
Pour terminer le calcul, il ne nous reste plus qu'à identifier la fonction $\Tr(\Fr_{\bar{x}},K_{\bar{x}})$ pour $x$ un $k$-point de la fibre de Springer affine et nous avons alors par définition:
\begin{center}
$\Tr(\Fr_{\bar{x}},K_{\bar{x}})=f_{\la,v}(g^{-1}\g_{0}g)$,
\end{center}
d'où l'on déduit la formule:
\begin{center}
$\sum\limits_{x\in[\cma/\cP_{v}(J'_{a})](k)}\Tr(\Fr_{\bar{x}}, K_{\bar{x}})_{\kappa}=\vol(J_{a}'(\co_{v},dt_{v}))\bo^{\kappa}_{a}(f_{\la},dt_{v})$
\end{center}
comme souhaité.
\end{proof}

Il nous faut maintenant passer de $J_{a}'$ à $J_{a}$.
Pour $\kappa:H^{1}(k,\cP_{v}(J_{a}))\rightarrow\bar{\mathbb{Q}}_{l}^{*}$, en utilisant la flèche 
\begin{center}
$H^{1}(k,\cP_{v}(J_{a}^{0}))\rightarrow H^{1}(k,\cP_{v}(J_{a}))$,
\end{center}
on obtient un caractère de $H^{1}(k,\cP_{v}(J_{a}^{0}))=H^{1}(F_{v},J_{a})$, noté de la même manière.

\begin{prop}\label{kappa1}
Soit $\kappa$ un caractère de $H^{1}(k,\cP_{v}(J_{a}))$ et notons $\kappa:H^{1}(F_{v},J_{a})\rightarrow\bar{\mathbb{Q}}_{l}^{*}$ le caractère qui s'en déduit, alors nous avons la formule suivante avec la $\kappa$-pondération 
\begin{center}
$\sum\limits_{x\in[\cma/\cP_{v}(J_{a})](k)}\Tr(\Fr_{\bar{x}}, K_{\bar{x}})_{\kappa}=\vol(J_{a}^{0}(\co_{v}),dt_{v})\bo^{\kappa}_{a}(f_{\la,v},dt_{v})$
\end{center}
De plus, si le caractère $\kappa:H^{1}(F_{v},J_{a})\rightarrow\bar{\mathbb{Q}}_{l}^{*}$ ne provient pas de $H^{1}(k,\cP_{v}(J_{a}))$, alors la $\kappa$-intégrale orbitale est nulle.

\end{prop}
\begin{proof}
La preuve est la même que celle de \cite[Prop. 8.2.7]{N}, par comparaison avec le cas connexe et en étudiant la fibre du foncteur 
\begin{center}
$[\cma/\cP_{v}(J_{a}^{0})](k)\rightarrow[\cma/\cP_{v}(J_{a})](k)$
\end{center}
sur laquelle la fonction $f_{\la,v}$ est constante, donc cela n'affecte pas les calculs.
\end{proof}

\begin{cor}\label{cohloc}
Soit $\kappa$ un caractère de $H^{1}(F_{v},J_{a})$ et soit $J_{a}^{\flat,0}$ la composante neutre du modèle de Néron $J_{a}$.  Soit $\Lambda$ un sous-groupe sans torsion $\sigma$-stable de $\cP_{v}(J_{a}^{0})$ vérifiant les hypothèses \cite[Hyp. 8.2.1]{N} qui existe d'après \ref{varproj}, on a l'égalité
\begin{center}
$\sum\limits_{n}(-1)^{n}\Tr(\sigma,H^{n}([\cma/\Lambda],K))_{\kappa}=\vol(J_{a}^{\flat,0}(\co),dt_{v})\bo^{\kappa}_{a}(f_{\la},dt_{v}).$
\end{center}
De plus, si le caractère $\kappa:H^{1}(F_{v},J_{a})\rightarrow\bar{\mathbb{Q}}_{l}^{*}$ ne provient pas de $H^{1}(k,\cP_{v}(J_{a}))$, alors la $\kappa$-intégrale orbitale est nulle.
\end{cor}
\begin{proof}
Il suffit de combiner \cite[8.1.13]{N}, en remplaçant le faisceau constant par le faisceau $K$ et $\ref{connexe}$, pour obtenir
\begin{center}
$\sum\limits_{n}(-1)^{n}\Tr(\sigma,H^{n}([\cma/\Lambda],K)_{\kappa})=(\sharp\cP_{v}^{0}(J_{a}^{0})(k))\vol(J_{a}^{0}(\co_{v}),dt_{v})\bo^{\kappa}_{a}(f_{\la},dt_{v})$.
\end{center}
Il ne nous reste plus qu'à prouver l'égalité :
\begin{center}
$(\sharp\cP_{v}^{0}(J_{a}^{0})(k))\vol(J_{a}^{0}(\co_{v}),dt_{v})=\vol(J_{a}^{\flat,0}(\co_{v}),dt_{v})$
\end{center}
ce qui fait l'objet de \cite[cor. 8.2.10]{N}.
\end{proof}
\subsection{Comptage dans une fibre de Hitchin anisotrope}
Soit $k$ un corps fini et $a\in\abdan(k)$. Nous voulons calculer le nombre de points $[\cmd^{1,ani}(a)/\cP_{a}^{1,ani}](k)$. Il nous faut relier ce nombre de points à un calcul local, malheureusement nous n'avons pas de formule de produit pour les éléments de degré zéro, le lemme suivant pallie ce problème.
On a un foncteur naturel:
\begin{center}
$\phi:[\cmd^{1,ani}(a)/\cP_{a}^{1,ani}]\rightarrow[\cmd^{ani}(a)/\cP_{a}^{ani}]$.
\end{center}
\begin{lem}\label{degcp}
La flèche $\phi$ est une équivalence de catégories.
\end{lem}
\begin{proof}
Montrons que $\phi$ est pleinement fidèle.
Soient $m$ et $n$ dans $\cmd^{1,ani}(a)$  tels que $\phi(m)=\phi(n)$. En particulier, il existe $p\in\cP_{a}^{ani}$ tel que $pm=n$. Or, comme $m$ et $n$ sont de degré de zéro, cela impose que $p$ soit également de degré zéro, donc $\phi$ est bien pleinement fidèle.
Montrons qu'il est essentiellement surjectif. Soit $m\in\cmd^{1,ani}(a)$, on choisit $p\in\cP_{a}^{ani}$ tel que $\deg(p)=-\deg(m)$ alors nous avons $\deg(pm)=0$, ce qui conclut.
\end{proof}
On considère  $J_{a}'$ un $X$-schéma en groupes lisse commutatif muni d'un morphisme $J_{a}'\rightarrow J_{a}$ qui est un isomorphisme sur un ouvert $U$ et qui a toutes ses fibres connexes. Soit $\tilde{\cP}_{a}$ le champ de Picard correspondant, on a une flèche $\tilde{\cP}_{a}\rightarrow\cP_{a}\rightarrow X_{*}(G/G_{der})$ et on note $\tilde{\cP}^{1}$ ses éléments de degré zéro.
La flèche $\tilde{\cP}_{a}\rightarrow\cP_{a}$ nous permet également de définir une action de $\tilde{\cP}^{1}$ sur $\cmd^{1}$.

On note $U_{a}=a^{-1}(\mathfrak{C}_{+}^{\la,rs})$, d'après la formule du produit \ref{produithitchin} et le lemme \ref{degcp}, nous avons une équivalence de catégories:
\begin{center}
$[\cmd^{1,ani}(a)/\tilde{\cP}_{a}^{1,ani}]=\prod\limits_{v\in X-U_{a}}[\overline{\cm}_{\la,v}(a)/\tilde{\cP}_{a,v}]$
\end{center}
compatible à l'action de Galois. Elles ont donc le même nombre de $k$-points.
De plus, pour tout caractère $\sigma$-invariant 
\begin{center}
$\kappa:\pi_{0}(\cP_{a})_{\sigma}\rightarrow\mathbb{Q}_{l}^{*}$,
\end{center}
on a d'après \cite[Prop. 8.4.3]{N} l'égalité pour le $\kappa$-comptage, soit:
\begin{equation}
\sharp[\cmd^{1,ani}(a)/\tilde{\cP}^{1,ani}](k)_{\kappa}=\prod\limits_{v\in X-U_{a}}\sharp[\cma/\tilde{\cP}_{a,v}](k)_{\kappa}.
\label{kappamachin}
\end{equation}

Maintenant, si l'on suppose de plus $a\in\abdban(k)$, en utilisant le théorème de transversalité \ref{transverse} et en combinant l'égalité \ref{kappamachin} avec \cite[8.1.6 et 8.4.5]{N} nous obtenons:
\begin{equation}
\sum\limits_{n}(-1)^{n}\Tr(\sigma,H^{n}(\cmd^{1}(a),IC_{\cmd})_{\kappa})=(\sharp\tilde{\cP}_{a}^{0})(k)\prod\limits_{v\in X-U}\vol(J'_{a}(\co_{v}))\bo_{a,v}^{\kappa}(f_{\la_{v}},dt_{v}).
\label{kappa2}
\end{equation}

\subsection{Le cas général}
Pour établir le théorème \ref{stab}, nous avons besoin de généraliser le comptage aux groupes semisimples. En effet, partant d'un groupe $G$ tel que $G_{der}$ est simplement connexe, il n'y a aucune raison que ses groupes endoscopiques $H$ vérifient la même propriété. Néanmoins, nous n'avons qu'à nous intéresser à la partie stable de la cohomologie, ce qui est plus commode.
On note $\Gamma=\pi_{1}(G)$ et on reprend les notations de la section \ref{ssimple}.

Fixons donc $a\in\abdan(\bar{k})$, d'invariant topologique $\tau$, on sait qu'il peut se relever en une $\Gamma$-orbite de sections dans $\abdaan(\bar{k})$.
Soit $a'\in\abdaan(\bar{k})$ qui relève $a$, on a une action de $\cP_{a'}^{aug}$ sur la fibre $\overline{\cm}_{\la}(a')^{aug}$ et également sur la fibre de Hitchin $\cmda$. Évidemment cette action dépend du choix du relèvement, mais étant donné que ce qui nous intéresse c'est le nombre de points de $\cmda$, nous verrons que cela n'est pas gênant.
On se donne un point base $(E^{*},\phi^{*})$ dans $\cmda$ dont on sait qu'il existe d'après \ref{surj1}.
On commence donc par obtenir la proposition suivante 
\begin{prop}
Soit $a\in\abdan(\bar{k})$, alors on a une formule du produit analogue à \ref{produithitchin}:
\begin{center}
$[\cmd(a)/\cP_{a'}^{aug}]=\prod\limits_{v\in X-U_{a}}[\overline{\cm}_{\la,v}(a)/\cP^{aug}_{a',v}]$
\end{center}
avec $U_{a}$ l'image réciproque du lieu fortement régulier semisimple par $a$.
\end{prop}
\begin{proof}
La preuve est la même que \ref{produithitchin} en remplaçant le point base donné par la section de Steinberg par $(E^{*},\phi^{*})$. De plus, au-dessus de $U_{a}$, la section est fortement régulière semisimple, donc les paires de Hitchin ne diffèrent que par l'action de $\cP_{a'}^{aug}$.
\end{proof}
Pour compter le nombre de points globaux, on est donc ramené à un calcul local et à une fibre de Springer usuelle, on obtient alors par le même argument que dans le cas simplement connexe (\ref{kappa1} et \ref{kappa2}) les formules suivantes:
\begin{equation}
\sum\limits_{x\in[\cma/\cP_{v}(J^{aug}_{a'})](k)}\Tr(\Fr_{\bar{x}}, K_{\bar{x}})=\vol(J_{a'}^{aug,0}(\co_{v}),dt_{v})\so_{a}(f_{\la,v},dt_{v})
\label{coh0}
\end{equation}
et
\begin{equation}
\sum\limits_{n}(-1)^{n}\Tr(\sigma,H^{n}(\cmd^{1}(a),IC_{\cmd})_{st})=(\sharp\cP^{aug,0}_{a'})(k)\prod\limits_{v\in X-U}\vol(J^{aug}_{a'}(\co_{v}))\so_{a,v}(f_{\la_{v}},dt_{v}).
\label{coh1}
\end{equation}
$\rmq$\begin{enumerate}
\item
Ici, $\cP_{a'}^{aug,0}$ désigne la composante connexe de $\cP_{a'}^{aug}$.
\item
Il est important de noter que si l'on prend $\kappa$ non trivial, comme il n'y a pas de choix canonique, le comptage dépend du choix du point base $(E^{*},\phi^{*})$ choisi, ce qui pose problème pour le comptage. On se retrouve confronté au même problème pour les groupes unitaires, étant donné que nous n'avons  de sections  de Steinberg non plus.
\end{enumerate}
\subsection{Le côté endoscopique}
Soit $\kappa$ une donnée endoscopique de groupe $H$.
On a donc une flèche canonique $\eta:\hat{H}\rightarrow\hat{G}$ d'où l'on déduit un foncteur de restriction:
\begin{center}
$\eta^{*}:\Rep(\hat{G})\rightarrow\Rep(\hat{H})$.
\end{center}
En particulier, par l'isomorphisme de Satake, on obtient un morphisme entre les algèbres de Hecke sphériques:
\begin{center}
$b:\cH_{G}\rightarrow\cH_{H}$.
\end{center}
Pour $\la\in X_{*}(T)^{+}$, soit $V_{\la}$ la représentation irréductible de plus haut poids $\la$, nous avons donc l'égalité:
\begin{center}
$b(f_{\la})=f_{\la}^{H}+\sum\limits_{\mu\in\soc_{H}(\la)}m_{\mu}f_{\mu}^{H}\in\cH_{H}$,
\end{center}
où $\soc_{H}(\la)$ désigne les cocaractères dominants de $H$ qui interviennent dans la décomposition de la représentation $\eta^{*}V_{\la}$.
On regarde alors la fibration de Hitchin:
\begin{center}
$f_{H}:\overline{\cm}_{\eta^{*}\la,H}\rightarrow\mathcal{A}_{\eta^{*}\la,H}$,
\end{center}
et le faisceau pervers pur $S_{\eta^{*}\la,H}$ sur $\overline{\cm}_{\eta^{*}\la,H}$ obtenu par équivalence de Satake géométrique à partir de la fonction $b(f_{\la})$ d'après \cite[Thm. 1.17]{VLaf}. Un calcul analogue à la section précédente avec les mêmes notations nous donne alors

\begin{equation}
\sum\limits_{x\in[\cmaH/\cP_{v}(J^{aug}_{a_{H}'})](k)}\Tr(\Fr_{\bar{x}},K^{H}_{\bar{x}})=\vol(J_{a_{H}'}^{aug,0}(\co_{v}),dt_{v})\so_{a_{H}}(b(f_{\la,v}),dt_{v})
\label{coh2}
\end{equation}
et

\begin{equation}
\sum\limits_{n}(-1)^{n}\Tr(\sigma,H^{n}(\overline{\cm}^{1}_{\eta^{*}\la,H}(a_{H}),K^{H})_{st})=(\sharp\cP^{aug,0}_{a_{H}'})(k)\prod\limits_{v\in X-U}\vol(J^{aug}_{a_{H}'}(\co_{v}))\so_{a,v}(b(f_{\la,v}),dt_{v}).
\label{coh3}
\end{equation}
avec $K^{H}:=S_{\eta^{*}\la,H}$.

\subsection{Le lemme fondamental et stabilisation géométrique}\label{final1}
On commence par prouver le théorème \ref{stab} sur $\tilde{\mathcal{A}}_{\eta^{*}\la,H}^{ani,\flat}-\tilde{\mathcal{A}}_{\eta^{*}\la,H}^{bad}$. Pour simplifier, on pose $\tilde{A}_{H}:=\tilde{\mathcal{A}}_{\eta^{*}\la,H}^{ani,\flat}$ et $\tilde{A}=\tilde{\mathcal{A}}_{\la}^{\flat,ani}$ et on note sans \og tilde\fg~ les ouverts des bases de Hitchin correspondants.
On démontre le théorème \ref{stab} sur l'ouvert $\tilde{A}_{H}^{bon}$.
\begin{proof}
Nous avons  une immersion fermée \ref{reunion} 
\begin{center}
$\nu:\tilde{A}_{H}\rightarrow \tilde{A}$,
\end{center}
définie sur $k$.
D'après le théorème du support \ref{detsupport} sur $A_{H}^{bon}$, les faisceaux pervers purs
\begin{center}
$L^{n}_{\kappa}=\tilde{\nu}^{*}\mathstrut^{p}H^{n}(\tilde{f}_{*}IC_{\cmd})_{\kappa}$ et  $L^{n}_{H,st}=\mathstrut^{p}H^{n}(\tilde{f}_{H,*}S_{\eta^{*}\la,H})_{st}$
\end{center}
sont les prolongements intermédiaires de leur restriction à n'importe quel ouvert non vide $\tilde{\mathcal{U}}$ de $\tilde{A}_{H}^{bon}$. Pour démontrer que ces faisceaux pervers sont isomorphes, il suffit donc de le démontrer sur un ouvert dense $\tilde{\mathcal{U}}$ de $A_{H}^{bon}$.
Sur $\kc$, on a le diviseur discriminant $\mathfrak{D}_{G}:=(2\rho,\Delta_{G})$ et sur $\mathfrak{C}_{+,H}$ le diviseur
$\mathfrak{D}_{H}:=(2\rho_{H},\Delta_{H})$. En particulier, nous avons l'égalité:
\begin{center}
$\nu^{*}_{+}\mathfrak{D}_{G}=\mathfrak{D}_{H}+2\mathfrak{R}_{H}^{G}$
\end{center}
avec $\mathfrak{R}_{H}^{G}:=(\rho-\rho_{H},\prod\limits_{\alpha\in\Psi}(1-\alpha))$ pour $\Psi$ un sous-ensemble de $R-R_{H}$ tel que pour toute paire de racines opposées $\pm\alpha\in R-R_{H}$, $\Psi\cap\{\pm\alpha\}$ soit de cardinal un. C'est un diviseur effectif sur  $\mathfrak{C}_{+,H}$.
\begin{lem}
On suppose  que pour tout $\mu\in\soc_{H}(\la)$, $\mu\succ2g$. Alors, l'ensemble $a_{H}\in \mathcal{A}_{\mu,H}$ tel que $a_{H}(\bar{X})$ coupe transversalement $\mathfrak{D}_{H,\mu}+ \mathfrak{R}^{G}_{H,\mu}$ et ne coupe pas $S=\supp(\la)$, forme un ouvert non vide $\mathcal{U}_{\mu}$ de $\mathcal{A}_{\mu,H}$. En particulier, l'ouvert
\begin{center}
$\mathcal{U}:=\bigcup\limits_{\mu\in\soc_{H}(\la)}\mathcal{U}_{\mu}$
\end{center}
est un ouvert dense  de $A_{H}$.
\end{lem}
\begin{proof}
La preuve est la même que pour la non-vacuité de l'ouvert $\abdd$.
\end{proof}
Nous avons besoin de définir un autre diviseur sur $A_{H}$ qui va tenir compte de l'inclusion de $\mathfrak{C}_{+}^{\mu}\rightarrow\kcd$.
On définit donc:
\begin{center}
$\mathfrak{R}^{G}_{H,\eta^{*}\la}=\bigcup\limits_{\mu\in\soc_{H}(\la)}\mathfrak{R}^{G}_{H,\la\mu}$
\end{center}
avec $\mathfrak{R}^{G}_{H,\la\mu}=\mathfrak{R}^{G}_{H,\mu}+\left\langle \rho,\la-\mu\right\rangle$.
Soit $\tilde{\mathcal{U}}$, l'image réciproque de $\mathcal{U}$ dans $\tilde{A}_{H}$, quitte à le rétrécir, on peut supposer que pour tout $n\in\mathbb{Z}$, les restrictions (notées de la même manière):
\begin{center}
$L^{n}_{\kappa}=(K^{n}_{\kappa})_{\vert\tilde{\mathcal{U}}}$ et $L^{n}_{H,st}=(K^{n}_{H,st})_{\vert\tilde{\mathcal{U}}}$
\end{center}
sont des systèmes locaux purs de poids $n$.
Il résulte alors de la pureté de ces systèmes locaux ainsi que du théorème de Chebotarev \cite[8.5.3]{N} qu'il suffit de démontrer que pour toute extension finie $k'$ de $k$, pour tout $k'$-points $\tilde{a}_{H}\in\tilde{\mathcal{U}}(k')$ d'image $\tilde{a}\in \tilde{A}(k')$, on a l'égalité des traces
\begin{center}
$\sum\limits_{n}(-1)^{n}\Tr(\sigma_{k'},(L^{n}_{\kappa})_{\tilde{a}})=\sum\limits_{n}(-1)^{n}\Tr(\sigma_{k'},(L^{n}_{H,st})_{\tilde{a}_{H}})$.
\end{center}
On note alors $a_{H}\in A_{H}$ l'image de $\tilde{a}$ et de même $a\in A$ l'image de $\tilde{a}$. Nous avons besoin également d'un relèvement $a'_{H}\in A_{H}^{aug}$ de $a_{H}$.
On a un morphisme canonique
\begin{center}
$J_{a}\rightarrow J_{a'_{H}}^{aug}$
\end{center}
qui est génériquement un isomorphisme. On considère alors $J'_{a}$ un schéma en groupes commutatifs, lisse sur $X$, à fibres connexes, équipé d'une flèche $J_{a}\rightarrow J_{a}'$ de telle sorte que $\tilde{\cP}_{a}^{1}$, le champ de Picard des $J_{a}'$ torseurs de degré zéro soit représentable par un groupe algébrique de type fini.
On a des flèches naturelles de $\tilde{\cP}_{a}^{1}\rightarrow\cP_{a}^{1}$ et $\tilde{\cP}_{a}^{1}\rightarrow\cP_{a'_{H}}^{1,aug}$ qui nous permettent de faire agir ce champ de Picard sur les fibres de Hitchin correspondantes et donc de comparer les quotients $[\cmd^{1}(a)/\tilde{\cP}_{a}^{1}]$ et $[\overline{\cm}_{\eta^{*}\la,H}^{1}(a_{H})/\tilde{\cP}_{a}^{1}]$.
Par commodité, on pose:
\begin{center}
$LF_{a}(f_{\la,v})_{\kappa}=\sum\limits_{x\in[\cma/\cP_{v}(J'_{a})](k)}\Tr(\Fr_{\bar{x}}, K_{\bar{x}})_{\kappa}$
\end{center}
et
\begin{center}
$LF_{a_{H}}(b(f_{\la,v}))=\sum\limits_{x\in[\cmaH/\cP_{v}(J'_{a})](k)}\Tr(\Fr_{\bar{x}}, K^{H}_{\bar{x}})$.
\end{center}
On rappelle que du côté \og H\fg, le faisceau $K^{H}$ correspond au faisceau $S_{\eta^{*}\la,H}$.
La conjonction des égalités \ref{kappa2} et \ref{coh2} montre que l'égalité se ramène à prouver la proposition suivante:
\begin{prop}\label{bon}
Soit $a_{H}\in A_{H}(k)$. Soit $v$ un point fermé au-dessus duquel $a_{H}(\bar{X}_{v})$ coupe transversalement le diviseur $\mathfrak{D}_{H,\mu}+ \mathfrak{R}^{G}_{H,\mu}$ et ne coupe pas $S=\supp(\la)$ pour tout $\mu\in\soc_{H}(\la)$, on a l'égalité:
\begin{center}
$LF_{a}(f_{\la,v})_{\kappa}=(-1)^{2[\deg(v)r^{G}_{H,v}(a_{H})-\rho_{H}^{G}(\deg(v)\la_{v})]}q_{v}^{\deg(v)[r^{G}_{H,v}(a_{H})-\rho_{H}^{G}(\deg(v)\la_{v})]}LF_{a_{H}}(b(f_{\la,v}))$
\end{center}
\end{prop}
$\rmq$
\begin{itemize}
\item
Dans l'énoncé de la proposition,  nous avons $\mathfrak{R}^{G}_{H,\eta^{*}\la}:=\sum\limits_{v\in X}\deg(v)r^{G}_{H,v}(a_{H})$.
\item
Comme $\sum\limits_{v\in X}\deg(v)r^{G}_{H,v}(a_{H})=\deg(\left\langle2\rho^{G}_{H},\la\right\rangle)$, cela explique que globalement on ait pas de décalage, puisque les termes se compensent.
\end{itemize}
\subsection{Preuve de la proposition \ref{bon}}
Dans la suite, on suppose pour alléger les notations que $\deg(v)=1$, le cas général étant analogue et de même on note $f_{\la}$ pour $f_{\la,v}$.
Soit $\bar{v}$ un point géométrique au-dessus de $v$, on procède par disjonction de cas.
\begin{lem}
Supposons que $\bar{v}\notin\supp(\la)$.
Nous avons alors trois cas de figures qui s'offrent à nous:
\begin{itemize}
\item
 $a_{H}(\bar{v})\notin\mathfrak{D}_{H,\la}\cup\mathfrak{R}^{G}_{H,\la}$, alors
\begin{center}
$d_{H,\bar{v}}(a_{H})=0, d_{\bar{v}}(a)=0$ et $r^{G}_{H,\bv}(a_{H})=0$.
\end{center}
\item
$a_{H}(\bar{v})\in\mathfrak{D}_{H,\la}$, alors $a_{H}(\bar{v})\notin\mathfrak{R}^{G}_{H,\la}$, et on obtient
\begin{center}
$d_{H,\bar{v}}(a_{H})=1, d_{\bar{v}}(a)=1$ et $r^{G}_{H,\bv}(a_{H})=0$.
\end{center}
\item
$a_{H}(\bar{v})\in\mathfrak{R}^{G}_{H,\la}$ alors $a_{H}(\bar{v})\notin\mathfrak{D}_{H,\la}$ et
\begin{center}
$d_{H,\bar{v}}(a_{H})=0, d_{\bar{v}}(a)=2$ et $r^{G}_{H,\bv}(a_{H})=1$.
\end{center}
Alors, dans ces cas, la proposition \ref{bon} est vérifiée.
\end{itemize}
\end{lem}
Ces trois cas ne présentent aucune différence avec le cas de l'algèbre de Lie et sont traités dans \cite[Lem. 8.5.7]{N}. On passe ensuite au cas où l'on suppose que le polynôme $a_{H}$ est dans la strate \og la plus haute\fg~ de $\mathcal{A}_{\eta^{*}\la,H}$.
\begin{lem}\label{comptII}
Soit $\bar{v}\in\supp(\la)$ et $a_{H}\in\mathcal{A}_{\la,H}$, dans ce cas par hypothèse on a :
\begin{center}
$d_{H,\bar{v}}(a_{H})=0, d_{\bar{v}}(a)=0$ et $r^{G}_{H,\bv}(a_{H})=0$,
\end{center}
et la proposition \ref{bon} est vérifiée.
\end{lem}
\begin{proof}
Dans ce cas, les fibres de Springer affines sont de dimension zéro, on a donc une action simplement transitive de $\cP_{v}(J_{a})$ et $\cP_{v}^{aug}(J_{a'})$ sur celles-ci, nous avons alors que:
\begin{center}
$[\cma/\cP_{v}(J_{a})]_{\kappa}=1$ et $[\cmavH/\cP_{v}(J^{aug}_{a'})]=1$,
\end{center}
en particulier, nous obtenons que :
\begin{center}
$\sum\limits_{x\in[\cma/\cP_{v}(J_{a})](k)}\Tr(\Fr_{\bar{x}}, K_{\bar{x}})_{\kappa}=(-1)^{\left\langle 2\rho,\la\right\rangle}q^{-\left\langle \rho,\la\right\rangle}$
\end{center}
en vertu de la formule \eqref{luz3} pour $f_{\la}$ et du fait que les strates plus petites ne contribuent pas comme le discriminant est nul.
En utilisant l'équation \eqref{coh1}, nous avons:
\begin{center}
$(-1)^{\left\langle 2\rho,\la\right\rangle}q^{-\left\langle \rho,\la\right\rangle}=\vol (J^{0}_{a}(\co_{v}), dt_{v})\bO^{\kappa}_{a,v}(f_{\la},dt_{v})$
\end{center}
D'où l'on déduit que:
\begin{center}
$\sum\limits_{x\in[\cma/\cP_{v}(J'_{a})](k)}\Tr(\Fr_{\bar{x}}, K_{\bar{x}})_{\kappa}=(-1)^{\left\langle 2\rho,\la\right\rangle}q^{-\left\langle \rho,\la\right\rangle}\frac{\vol(J'_{a}(\co_{v}))}{{\vol(J^{0}_{a}(\co_{v}))}}$
\end{center}
De manière analogue, on obtient du côté endoscopique:
\begin{center}
$\sum\limits_{x\in[\cmavH/\cP_{v}(J'_{a})](k)}\Tr(\Fr_{\bar{x}}, K^{H}_{\bar{x}})=(-1)^{\left\langle 2\rho_{H},\la\right\rangle}q^{-\left\langle \rho_{H},\la\right\rangle}\frac{\vol(J'_{a}(\co_{v}))}{{\vol(J^{aug,0}_{a'_{H}}(\co_{v}))}}$.
\end{center}
Maintenant de même que dans \cite[Lem. 8.5.7]{N}, on a un isomorphisme canonique entre $J_{a}^{0}$ et $J_{a_{H}'}^{aug,0}$, on obtient alors :
\begin{center}
$LF_{a}(f_{\la})_{\kappa}=(-1)^{2\rho_{H}^{G}(\la)}q^{-\rho_{H}^{G}(\la)}LF_{a_{H}}(f^{H}_{\la})$.
\end{center}
\end{proof}
Il ne nous reste plus qu'à traiter  le cas des strates plus petites pour $\mu\in\soc_{H}(\la)$avec $\mu\neq\la$, qui est le plus subtil; cela conclura la preuve de la proposition \ref{bon}:
\begin{prop}\label{comptIII}
Soit $\bar{v}\in\supp(\la)$ et $a_{H}\in\mathcal{A}_{\mu,H}$ avec $\mu\neq\la\in\soc_{H}(\la)$.
Alors l'hypothèse $a_{H}\in U$ donne:
\begin{center}
$d_{H,\bar{v}}(a_{H})=0$, $a(\bar{v})\notin\mathfrak{D}_{\mu}$, $a_{H}(\bar{v})\notin\mathfrak{R}^{G}_{H,\mu}$,
\end{center}
et la proposition \ref{bon} est vérifiée.
\end{prop}
\begin{proof}
Soit $a(\bar{v}):=\nu(a_{H})(\bar{v})\in\kc^{\mu,rs}(\co_{\bar{v}})$, posons $\g_{0}:=\eps_{+}(a(\bar{v}))\in V_{G}^{\mu,rs}(\co_{\bar{v}})$. Il résulte de la formule de dimension que $c_{\bar{v}}(a)=0$ et comme $\g_{0}\in V_{G}^{\mu,rs}(\co_{\bar{v}})$, quitte à conjuguer, on peut écrire:
\begin{center}
$\g_{0}:=(\pi_{\bar{v}}^{-w_{0}\mu},\pi_{\bar{v}}^{\mu}t)\in G_{+}(F_{\bar{v}})$
\end{center}
avec $t\in T(\co_{\bar{v}})$. Dans la suite pour simplifier, on supprime les indices $\bar{v}$, posons $\g=\pi^{\mu}t$; pour démontrer la proposition \ref{bon}, on considère la fibre de Springer affine:
\begin{center}
$\kX^{\la}_{\mu}:=\{g\in G(F)/G(\co)~\vert~g^{-1}\pi^{\mu} g\in \overline{K\pi^{\la}K}\}$.
\end{center}
En tenant compte de l'égalité $X_{*}(T)=T(F)/T(\co)$, nous avons que:
\begin{center}
$\kX^{\la}_{\mu}=X_{*}(T)\times Y^{\la}_{\mu}$,
\end{center}
avec $Y^{\la}_{\mu}$ la fibre correspondante pour $U(F)/U(\co)$.
Soit l'automorphisme $f_{\g}:U(F)\rightarrow U(F)$ donnée par $u\mapsto u^{-1}\g u\g^{-1}$, la fibre $Y^{\la}_{\mu}$ se récrit:
\begin{center}
$Y_{\mu}^{\la}:=f_{\g}^{-1}(K\pi^{\la}K\pi^{-\mu}\cap U(F))/U(\co)$.
\end{center}
Comme $\g_{0}=(\pi^{-w_{0}\mu},\g)\in V_{G}(\co)^{rs}$, en regardant sur les groupes radiciels, on voit que $f_{\g}$ induit un isomorphisme sur $U(\co)$, en particulier, nous en déduisons que
\begin{center}
$f_{\g}:Y_{\mu}^{\la}\simeq K\pi^{\la}K\pi^{-\mu}\cap U(F)/U(\co)$.
\end{center}
On considère alors la variété introduite par Mirkovic-Vilonen \cite{MV} $S_{\mu}\cap\overline{\Gr}_{\la}$. On a alors une flèche de projection:
\begin{equation}
S_{\mu}\cap\overline{\Gr}_{\la}\rightarrow \overline{K\pi^{\la}K}\pi^{-\mu}\cap U(F)/U(\co),
\label{projI}
\end{equation}
de fibre l'espace affine $U(\co)/\pi^{\mu}U(\co)\pi^{-\mu}$  de dimension $\left\langle \rho,\la-\mu\right\rangle$.
On rappelle alors l'énoncé suivant qui va nous permettre de calculer les valeurs de $f_{\la}$ tiré de Ngô-Polo \cite[Thm.3.1]{NP}:
\begin{thm}
Le complexe $R\Gamma_{c}(S_{\mu},\mathcal{A}_{\la})$ est concentré en degré $\left\langle 2\rho,\mu\right\rangle$, de dimension $m_{\la\mu}$ la multiplicité de $\mu$ dans la représentation $V_{\la}$ et sur lequel le Frobenius agit par $q^{\left\langle \rho,\mu\right\rangle}$.
\end{thm}
On déduit  de ce théorème l'égalité:
\begin{equation}
\sum\limits_{x\in S_{\mu}\cap\overline{\Gr}_{\la}(\mathbb{F}_{q})}\Tr(\Fr_{x},IC_{\overline{\Gr}_{\la}})=\Tr(\Fr_{q},R\Gamma_{c}(S_{\mu},\mathcal{A}_{\la}))=(-1)^{\left\langle 2\rho,\mu\right\rangle}m_{\la\mu}q^{\left\langle \rho,\mu\right\rangle},
\label{MVI}
\end{equation}
Comme on s'intéresse à la $\kappa$-intégrale orbitale, il nous faut étudier la structure de $\cP(J_{a})$:
\begin{lem}\label{centII}
On a une suite exacte:
$$\xymatrix{1\ar[r]&\mathbb{G}_{a}^{\left\langle \rho,\la-\mu\right\rangle}\ar[r]&\cP(J_{a})\ar[r]&X_{*}(T)\ar[r]&1}.$$
En particulier, nous avons $H^{1}(k,\cP(J_{a}))=0$.
\end{lem}
\begin{proof}
On considère l'élément 
\begin{center}
$u:=\prod\limits_{1\leq i\leq r}x_{i}(\pi^{-\left\langle \omega_{i},\la-\mu\right\rangle})\in\prod\limits_{1\leq i\leq r}U_{i}$, 
\end{center}
où les $U_{i}$ sont les groupes radiciels de racine simple $\alpha_{i}$. On vérifie qu'il est dans l'ouvert régulier de la fibre de Springer affine $X^{\la}_{\mu}$. Pour déterminer la structure de $\cP(J_{a})$, comme le champ de Picard agit simplement transitivement sur l'ouvert régulier, il suffit de déterminer le stabilisateur de $u$ dans $T(\co)$ qui s'identifie à $J_{a}(\co)$.
On décompose ensuite $\cP(J_{a})$ en une extension de  $X_{*}(T)=T(F)/T(\co)$ avec $T(\co)/J_{a}(\co)$.
On cherche donc à résoudre l'équation:
\begin{center}
$u^{-1}tu\in G(\co)$ pour $t\in T(\co)$.
\end{center}
Cette équation revient à résoudre $u^{-1}tut^{-1}\in G(\co)$ qui se ramène aux équations:
\begin{center}
$\forall~ 1\leq i\leq r, \pi^{-\left\langle \omega_{i},\la-\mu\right\rangle}(1-\alpha_{i}(t))\in\co$
\end{center}
d'où l'on déduit 
\begin{center}
$\forall~ 1\leq i\leq r,\alpha_{i}(t)=1+\pi^{\left\langle \omega_{i},\la-\mu\right\rangle}\co$ 
\end{center}
et donc $T(\co)/J_{a}(\co)$ est un groupe additif, ce qui conclut.
\end{proof}
De ce lemme, on déduit, que la $\kappa$-intégrale se calcule comme l'intégrale usuelle et  en combinant \eqref{MVI} et \eqref{projI}, on obtient:
\begin{equation}
LF_{a}(f_{\la})_{\kappa}=(-1)^{\left\langle 2\rho,\mu\right\rangle}m_{\la\mu}q^{-\left\langle \rho,\mu\right\rangle}\frac{\vol(J'_{a}(\co))}{{\vol(J^{0}_{a}(\co))}}.
\end{equation}
Le côté endoscopique lui se calcule de la même manière que dans le cas \ref{comptII} et l'on obtient:
\begin{center}
$LF_{a_{H}}(b(f_{\la,v}))=(-1)^{\left\langle 2\rho_{H},\mu\right\rangle}m_{\la\mu}q^{-\left\langle \rho_{H},\mu\right\rangle}\frac{\vol(J'_{a}(\co))}{{\vol(J^{aug,0}_{a'_{H}}(\co))}}$,
\end{center}
où l'on rappelle que $b(f_{\la})=\sum\limits_{\nu\in\soc_{H}(\la)}m_{\la\nu}f^{H}_{\nu}$.
Pour conclure, il ne nous reste plus qu'à comparer $\vol(J^{0}_{a}(\co))$ avec $\vol(J^{aug,0}_{H, a'_{H}}(\co))$. On utilise alors le fait que $\nu(a_{H})$ se factorise par $\kc^{\mu,rs}(\co)$ et le lemme \ref{centII} pour obtenir l'égalité:
\begin{center}
$\vol(J^{0}_{a}(\co))=q^{\left\langle \rho,\la-\mu\right\rangle}\vol(J^{aug,0}_{H, a'_{H}}(\co))$.
\end{center}
\end{proof}
\end{proof}
Nous pouvons en déduire l'identité entre les intégrales orbitales locales:
\begin{cor}\label{rr}
On a l'égalité:
\begin{center}
$\bo_{a}^{\kappa}(f_{\la,v})=(-1)^{2[r_{H}^{G}(a)-\rho_{H}^{G}(\la)]}q^{r_{H}^{G}(a)-\rho_{H}^{G}(\la)}\so_{a_{H}}(f_{\la,v}^{H})$.
\end{center}
\end{cor}
\begin{proof}
Les arguments sont maintenant les mêmes que \cite[sect. 8.6]{N}.
\end{proof}
On peut terminer  la preuve du théorème \ref{stab} sur $\abdanf$:
\begin{proof}
Il suffit de combiner, de la même manière que \cite[sect. 8.7]{N}, les formules \eqref{kappa2}-\eqref{coh3}, avec \eqref{coh0}, \ref{kappa1} et  le corollaire \ref{rr}.
\end{proof}

   \footnotesize

  (A. Bouthier) \textsc{Einstein Institute of Mathematics, Hebrew University, Givat Ram, Jerusalem, 91904, Israël}\par\nopagebreak
  \textit{E-mail address}: \texttt{bouthier@math.huji.ac.il}

\end{document}